\documentclass[reqno,10pt,oneside]{amsart}

\usepackage{amsthm,amsmath,amssymb,amscd}
\usepackage{mathrsfs}
\usepackage{lscape}
\usepackage{graphics} 
\usepackage[all,cmtip]{xy}
\usepackage{graphicx}
\usepackage{subfigure}
\usepackage{url}	
\usepackage{hyperref}
\usepackage{color}
\usepackage[usenames,dvipsnames]{xcolor}
\usepackage{verbatim}
\usepackage{fancyhdr}
\usepackage{geometry}
\usepackage{cancel}
\usepackage[normalem]{ulem}
\usepackage{mathtools}
\usepackage[english]{babel}
\mathtoolsset{showonlyrefs,showmanualtags}
\numberwithin{equation}{section}

\newtheorem{teo}{Theorem}[section]
 
\newtheorem{prop}[teo]{Proposition}

\newtheorem{lemma}[teo]{Lemma}

\newtheorem{example}[teo]{Example}
\theoremstyle{definition} 
\newtheorem{definition}[teo]{Definition}
\theoremstyle{remark}
\newtheorem{remark}[teo]{Remark}

\newcommand{\del}{\partial}
\newcommand{\QQ}{\mathbb{Q}}

\newcommand{\Ric}{\text{Ric}}
\newcommand{\tr}{\text{tr}}

\newcommand{\st}{\left(}
\newcommand{\dt}{\right)}
\newcommand{\sq}{\left[}
\newcommand{\dq}{\right]}
\newcommand{\sg}{\left\{}
\newcommand{\dg}{\right\}}
\newcommand{\p}{\textbf{p}}
\newcommand{\q}{\textbf{q}}
\newcommand{\hg}{\textbf{h}}
\newcommand{\kg}{\textbf{k}}
\newcommand{\ag}{\textbf{a}}
\newcommand{\bg}{\textbf{b}}
\newcommand{\cg}{\textbf{c}}

\newcommand{\aaa}{{\bf a}}
\newcommand{\bbb}{{\bf b}}
\newcommand{\ccc}{{\bf c}}

\newcommand{\GGG}{ \mathbf G}
\newcommand{\etag}{\boldsymbol{\eta}}
\newcommand{\dd}{\partial\overline{\partial}}
\newcommand{\NN}{\mathbb{N}}

\newcommand{\s}{\textbf{S}}

\newcommand{\vol}{\text{Vol}_{\omega}}

\newcommand{\RR}{\mathbb{R}}
\newcommand{\ZZ}{\mathbb{Z}}
\newcommand{\CC}{\mathbb{C}}
\newcommand{\PP}{\mathbb{P}}

\newcommand{\Sp}{\mathbb{S}}

\newcommand{\ambbda}{\mathcal{B}_{\alpha}}
\newcommand{\dombd}{\mathfrak{B}\st\kappa, \beta,\sigma \dt}
\newcommand{\dombdmod}{\mathfrak{B}\st\kappa, \beta',\sigma' \dt}
\newcommand{\K}{K\"{a}hler}
\newcommand{\ww}{\bf{w}}
\newcommand{\aaaa}{\boldsymbol{\alpha}}
\newcommand{\bbbb}{\boldsymbol{\beta}}
\newcommand{\cccc}{\boldsymbol{\gamma}}

\newcommand{\hko}{\textbf{H}_{\hg,\kg}^{o}}

\newcommand{\Lg}{\mathbb{L}_{\omega}}

\newcommand{\Cc}{C_{\delta-4}^{0,\alpha}\st M_{\p} \dt}

\newcommand{\Cqdd}{C_{\delta}^{4,\alpha}\st M_{\p} \dt \oplus \mathcal{D}_{\p}\st \bg,\cg \dt}

\newcommand{\rep}{r_\varepsilon}
\newcommand{\xroi}{\chi_{j,r_0}}

\newcommand{\Rep}{R_\varepsilon}
\newcommand{\hkjj}{H_{\tilde{h}\tilde{k}}^{I}}

\newcommand{\hkii}{\textbf{H}_{\tilde{h}\tilde{k}}^{I}}
\newcommand{\Le}{\mathbb{L}_{\eta}}

\newcommand{\Ccx}{C_{\delta-4}^{0,\alpha}\st X \dt }

\newcommand{\csfii}{f_{B,\tilde{h},\tilde{k}}^{I}}

\newcommand{\cga}{c(\Gamma)}
\newcommand{\egaj}{e\st \Gamma_{j}\dt}
\newcommand{\egal}{e\st \Gamma_{N+l}\dt}
\newcommand{\cgaj}{c\st\Gamma_{j}\dt}
\newcommand{\cgal}{c\st\Gamma_{N+l}\dt}

\begin{document}

\title{On the resolution of constant scalar curvature K\"ahler orbifolds}
\author[Claudio Arezzo] {Claudio Arezzo}
\address{ICTP Trieste and Univ. of Parma, arezzo@ictp.it }
\author{Riccardo Lena}
\address{riccardo.lena@gmail.com}
\author{Lorenzo Mazzieri}
\address{Scuola Normale Superiore, Pisa, l.mazzieri@sns.it}

\begin{abstract}
In this paper, given a compact Kcsc orbifolds of any dimension and with nontrivial holomorphic vector fields,  we find sufficient conditions on the position of singular points
 in order to admit a Kcsc desingularization, generalizing the result of the first author with F. Pacard in the case of blowing up smooth points. A series of explicit examples are discussed.
\end{abstract}
\maketitle

\vspace{-,15in}


\section{Introduction}

\noindent The aim of this paper is to construct new families of K\"ahler constant scalar curvature (Kcsc from now on) metrics on 
compact complex manifolds and orbifolds. Thanks to the work of Rollin-Singer (\cite{RollinSinger}) and Arezzo-Pacard (\cite{ap1},\cite{ap2}) we have a good understanding of the behavior of the Kcsc equation under the standard blow up process when then base orbifold has no holomorphic vector fields, or when the base manifold is smooth but nontrivial holomorphic vector fields do exist.

\noindent This raises the natural question of whether more refined geometric constructions can be performed when the base object is singular and has continuous holomorphic symmetries. Even assuming isolated singular singular points on the base, this problem is widely open and delicate.
In this paper we address this question assuming the base $M$ has only {\em{isolated quotient singularities}}, hence {\em locally of the form $\CC^m/\Gamma_j$, where $m$ is the complex dimension of $M$, $j\in J$ parametrizes the set of points we want to desingularize, and $\Gamma_j$ is a finite subgroup of $U(m)$ acting freely away from the origin}.

\noindent Given such a singular object there is more than one way to ``desingularize" it, i.e. seeing it as a degeneration of smooth manifolds
 (or orbifolds with fewer singular points), one of which is of particular importance:  replace a small neighborhood of a singular point and replace it with a large piece of a K\"ahler {\em{resolution}} $\pi\colon (X_{\Gamma}, \eta) \rightarrow \CC^m/\Gamma$. For such a construction to even have a chance to preserve the Kcsc equation it is necessary that $(X_{\Gamma}, \eta)$ is scalar flat, i.e. it is necessary to assume that {\em{$\CC^m/\Gamma_j$ has a scalar flat ALE resolution.}}

\medskip

\noindent More precisely, having fixed a set of singular points $\{p_1, \dots, p_n\} \subset M$ each corresponding to a group $\Gamma_j$, and denoted by
$ B_{j,r}  : =  \{  z \in {\mathbb C}^{m} / \, \Gamma_j
\, : \,  |z| < r \}, $
we can define, for
all $r >0$ small enough (say $r \in (0, r_0)$)
\begin{equation}
M_r : = M \setminus \cup_j \, B_{j,r}  . \label{eq:2.3}
\end{equation}

\noindent On the other side, for each $j = 1, \ldots, n$, we are given
a $m$-dimensional \K\ manifold $(X_{\Gamma_j}, \eta_j)$, with one end biholomorphic to a
neighborhood of infinity in ${\mathbb C}^{m} / \, \Gamma_j$. 
\noindent Dual to the previous notations on the base manifold, we set 
$C_{j , R}  : =  \{ x \in {\mathbb C}^{n} / \, \Gamma_j \,
: \, |x | >  R \}, $
the complement of a closed large ball and the complement of an open
large ball in $X_{\Gamma_j}$ (in the  coordinates which parameterize a
neighborhood of infinity in $X_{\Gamma_j}$). We define, for all $R > 0$ large
enough (say $R > R_0$)

\begin{equation}
X_{\Gamma_j, R} : = X_{\Gamma_j} \setminus C_{j , R}. \label{eq:2.5}
\end{equation}
which corresponds to the manifold $X_{\Gamma_j}$ whose end has been
truncated. The boundary of $X_{\Gamma_j,R}$ is denoted by $\del C_{j,R}$.

\medskip

\noindent We are now in a position to describe the generalized connected sum
construction. Indeed, for all
$\varepsilon \in (0, r_0/R_0)$, we choose $r_{\varepsilon} \in (\varepsilon \, R_0 , r_0)$ and
define
\begin{equation}
R_{\varepsilon} : =  \frac{r_{\varepsilon}}{\varepsilon} . \label{eq:2.6}
\end{equation}
By construction
\[
M_{\varepsilon} : = M \sqcup _{{p_{1}, \varepsilon}} X_{\Gamma_1} \sqcup_{{p_{2},\varepsilon}} \dots
\sqcup _{{p_n, \varepsilon}} X_{\Gamma_n} ,
\]
is obtained by connecting  $M_{r_\varepsilon}$ with the truncated ALE spaces
$X_{\Gamma_{1, R_\varepsilon}}, \ldots, X_{\Gamma_{n , R_\varepsilon}}$. The identification of the
boundary $\del B_{j , r_{\varepsilon}}$ in $M_{r_{\varepsilon}}$ with the boundary $\del
C_{j , R_\varepsilon}$ of $X_{\Gamma_j, R_\varepsilon}$ is performed using the change of
variables
\[
(z^{1} , \ldots, z^{m} )  = \varepsilon \, (x^{1} , \ldots, x^{m}) ,
\]
where $(z^{1}, \ldots, z^{m} )$ are the coordinates in $B_{j , r_0}$
and $(x^{1}, \ldots, x^{m})$ are the coordinates in $C_{j , R_0}$.

\vspace{3mm}

\noindent It was proved in \cite{ap1} that if no nontrivial holomorphic vector fields exist on $M$ the ALE scalar flat condition is also sufficient to construct a family parametrized by the gluing parameter $\varepsilon$ on the manifold (or orbifold) obtained by this procedure. On the other hand, the known picture for the blow up, suggests that the number and position of points should be relevant to achieve the same existence theorem in presence of continuous symmetries. The main result of this paper is to exhibit this obstruction (the {\em{balancing condition}}) in this new singular case.

\noindent The content of the analysis performed in this paper is that the results strongly depends on $\Gamma_j$ or in turns on $(X_{\Gamma_j}, \eta_j)$.
Indeed, we will show that there are essentially two possibilities, and the new balancing condition takes a very different shape:
\begin{enumerate}
\item
there exists $K\subset J$ s.t. $(X_{\Gamma_k}, \eta_k)$ is scalar flat {\em{but not Ricci flat}} for $k\in K$
(hence $\Gamma_j$ has to contained in $U(m)$ but not in $SU(m)$ unless it is trivial);
\item
 $(X_{\Gamma_j}, \eta_j)$ is Ricci flat for all $j\in J$ (hence $\Gamma_j$ has to be nontrivial and contained in $SU(m)$).
\end{enumerate}

\noindent Of course we are not interested in flat models since they wouldn't change the original metric. A notable special case of the first situation is when $K=J$, but it easy to see that the original analysis carried out in \cite{ap2} would continue to hold and so the old balancing condition as stated for standard blow ups will continue to suffice to get the result. 

\noindent This fact keeps on holding as long as $K$ is nonempty and in fact we will see that Ricci-flat singularities do not contribute to the balancing condition, and only the balancing condition on scalar (non Ricci) flat points contribute. The key difference in handling these cases come form the following observation:
it is a standard fact  (see Proposition \ref{asintpsieta2}) that a scalar flat metric on an ALE space can be written as

\begin{equation}
\eta _j=i\partial\overline{\partial}\st \frac{|x|^2}{2}+e({\Gamma_j})|x|^{4-2m} - c({\Gamma_j})|x|^{2 - 2m} +\psi_{\eta_j}\st x \dt \dt \,,
\end{equation}
for $m \geq 3$ and 
\begin{equation}
\eta_{j}=i\dd\st \frac{|x|^{2}}{2}+e({\Gamma_j})\log\st |z| \dt - c({\Gamma_j})|z|^{-2}+\mathcal{O}\st |z|^{-4} \dt \dt
\end{equation}
for $m=2$, with $e({\Gamma_j}), c({\Gamma_j})\in \RR$ and where the function $\psi_\eta$ satisfies the estimate
\begin{equation}
\psi_{\eta}\st x \dt=\mathcal{O}\st |x|^{ - 2m} \dt \, .\label{eq:asintpsieta4-2m}
\end{equation}

\noindent  Ricci flat models are forced to have $e(\Gamma_j) = 0$  and $c(\Gamma_j) >0$ unless it is flat, even though it is an important open problem whether this condition characterizes such metrics. In any case, this missing term is exactly what force us to find a genuinely new route to prove the existence of Kcsc metrics.

\noindent In order to state our results, let us recall that the linear part of the scalar curvature map is given by the (Lichnerowicz) operator
\begin{equation}
{\mathbb L}_{\omega} f \,  =  \, \Delta^{2}_{\omega} f \,   +  \, 4 \, \langle \, \rho_{\omega} \, | \,   i\dd f \, \rangle  \, , 
\end{equation}

\noindent and from now on we denote by $\varphi_{1},\ldots,\varphi_{d}$  an independent set of smooth functions of zero mean
such that $\{1,\varphi_{1},\ldots,\varphi_{d}\}$ forms a basis of ker$\Lg$. It is a classical fact that this space is in fact isomorphic to the space of holomorphic vector fields vanishing somewhere on $M$. The case when $K \neq \emptyset$ is in fact much easier 
than the case $K=\emptyset$, and in fact can be handled by simple variations of the original Arezzo-Pacard construction
 (it was in fact already stated in complex dimension $2$ without proof in \cite{RollinSingerII}). 
 We will state and prove the output of this analysis in this case also because we believe it largely clarifies the most difficult case.

\begin{teo}
\label{belliebrutti}
Let $(M,g,\omega)$ be a Kcsc orbifold with isolated singularities.
\begin{itemize} 
\item Let $\p\:=\sg p_1,\ldots,p_{N}\dg\subseteq M$ the set of points  with neighborhoods biholomorphic to  a ball of  $\CC^m/\Gamma_{j}$ with $\Gamma_{j}$ nontrivial such that $\CC^{m}/\Gamma_{j}$ admits an ALE  Kahler Ricci-flat resolution $\st X_{\Gamma_{j}},h,\eta_{j} \dt$.

\item Let $\q:=\sg q_1,\ldots,q_{K}\dg\subseteq M$
be a set of points with neighborhoods biholomorphic to a ball of  $\CC^m/\Gamma_{N+l}$ such that
$\CC^{m}/\Gamma_{N+l}$ admits a scalar flat ALE resolution $(Y_{\Gamma_{N+l}},k_{l},\theta_{l})$ with $e({\Gamma_{N+l}})\neq 0$.
\end{itemize}

\noindent 
If there exist $\ag:=\st a_{1},\ldots,a_{K} \dt\in \st\RR^{+}\dt^{K}$ such that

	\begin{displaymath}
	\left\{\begin{array}{lcl}
	\sum_{l=1}^{K}\frac{a_{l}e\st \Gamma_{N+l} \dt}{|\st \Gamma_{N+l} \dt|}\varphi_{i}\st q_{l} \dt=0 && i=1,\ldots, d\\
	&&\\
	\st \frac{a_{l} e(\Gamma_{N+l})}{|e(\Gamma_{N+l})|}  \varphi_{i}\st q_{l} \dt \dt_{\substack{1\leq i\leq d\\1\leq l\leq K}}&& \textrm{has full rank}
	\end{array}\right.
	\end{displaymath}


then there exists $\varepsilon_0>0$ such that  for any $\varepsilon < \varepsilon_0$ and any ${\bf{b}} = (b_1, \dots, b_n) \in \st\RR^{+}\dt^{N}$

\[
{\tilde{M}} : = M \sqcup _{{p_{1}, \varepsilon}} X_{\Gamma_1} \sqcup_{{p_{2},\varepsilon}} \dots
\sqcup _{{p_{N}, \varepsilon}} X_{\Gamma_{N}}\sqcup_{{q_{1}, \varepsilon}} X_{\Gamma_{N+1}} \sqcup_{{q_{2},\varepsilon}} \dots \sqcup _{{q_{N+K}, \varepsilon}}X_{\Gamma_{N+K}},
\]

\noindent has a Kcsc metric in the class  $\pi^*[\omega] + \sum_{l=1}^K\varepsilon^{2m-2} \mathfrak{a}_{l} ^{2m-2}[\tilde{\theta_{l}}] + \sum_{j=1}^N\varepsilon^{2m}\mathfrak{b}_{j}^{2m}  [\tilde{\eta_j}]   $ where
\begin{itemize}
\item 
$\mathfrak{i}_{l}^{*}\sq \tilde{\theta}_{l} \dq=[\theta_{l}]$
with  $\mathfrak{i}_{l}:Y_{\Gamma_{N+l},\Rep}\hookrightarrow \tilde{M}$ the standard inclusion (and analogously for $\tilde{\eta}_{j}$),
\item for some $\gamma >0$
\begin{align}
\left|\mathfrak{a}_{l}^{2m-2} - \frac{|\Gamma_{N+l}|a_{l}}{4 \, |\Sp^{3}||e({\Gamma_{N+l}})|}\right| \leq C \varepsilon^{\gamma}&\qquad &\textrm{for }m= 2\\
\left|\mathfrak{a}_{l}^{2m-2} - \frac{|\Gamma_{N+l}|a_{l}}{8(m-2)(m-1) \, |\Sp^{2m-1}||e({\Gamma_{N+l}})|}\right| \leq C \varepsilon^{\gamma}&\qquad& \textrm{for }  m\geq 3
 \end{align}
\item
and $\mathfrak{b}_{j}$ is a continuous function of $\varepsilon$ and $\ag$, such that $|\mathfrak{b}_{j}^{2m}-b_j| = o(1)$.

\end{itemize}
\end{teo}

\noindent To check whether the above conditions are satisfied, it is then required to know only the sign of the coefficients $e(\Gamma)$, which is 
in turn a problem of great independent interest because of the importance of  these models in physical context (see for example \cite{LeBrun}, \cite{Honda}, \cite{CalderbankSinger}).
On the other hand it is clear that a change of ${\bf{b}}$ changes the final K\"ahler class where the Kcsc metric lives.

\noindent Whether, given $\Gamma$ in $U(m)$ ($SU(m)$ respectively), a scalar flat (respectively Ricci flat) K\"ahler resolution exists
is by itself an important problem in different areas of mathematics and we will not digress on it here. It suffices to recall the reader that 
Ricci flat  models do exist for any subgroup of $SU(m)$ with $m=2$, thanks to the work of Kronheimer, while in higher dimension one need to assume the existence of a K\"ahler crepant resolution and then apply deep results by Joyce \cite{j} and Van Coevering (\cite{VanCoevering}). As for scalar (non Ricci)  flat models even in dimension $2$, Calderbank-Singer have produced large families of examples associated to cyclic subgroups (\cite{CalderbankSinger}).

\medskip

\noindent  The main result of this paper is then the following which handles the case $K=\emptyset$:

\begin{teo}\label{maintheorem}
			Let $\st M,g,\omega\dt$ be a compact Kcsc orbifold with isolated singularities. Let $\p\:=\sg p_1,\ldots,p_{N}\dg\subseteq M$ the set of points  with neighborhoods biholomorphic to  a ball of  $\CC^m/\Gamma_{j}$ with $\Gamma_{j}$ nontrivial such that $\CC^{m}/\Gamma_{j}$ admits an ALE  Kahler Ricci-flat resolution $\st X_{\Gamma_{j}},h,\eta_{j} \dt$. Suppose moreover that there exists $ \bg \in \mathbb{R}_{+}^{N}$ and $\cg\in\RR^{N}$ such that
			
		\begin{displaymath}
		\left\{\begin{array}{lcl}
		\sum_{j=1}^{N}b_{j}\Delta_{\omega}\varphi_{i}\st p_{j} \dt+c_{j}\varphi_{i}\st p_{j} \dt=0 && i=1,\ldots, d\\
		&&\\
		\st b_{j}\Delta_{\omega}\varphi_{i}\st p_{j} \dt+c_{j}\varphi_{i}\st p_{j} \dt \dt_{\substack{1\leq i\leq d\\1\leq j\leq N}}&& \textrm{has full rank}
		\end{array}\right.
		\end{displaymath}
	
	If 
	\begin{equation}\label{eq:tuning}
	c_{j}=s_{\omega}b_{j}
	\end{equation}

	Then  there exists $\varepsilon_0>0$ such that  for any $\varepsilon < \varepsilon_0$ 
	\[
	\tilde{M} : = M \sqcup _{{p_{1}, \varepsilon}} X_{\Gamma_1} \sqcup_{{p_{2},\varepsilon}} \dots
	\sqcup _{{p_N, \varepsilon}} X_{\Gamma_N}
	\]

	\noindent has a Kcsc metric in the class $\pi^{*}\sq\omega \dq+ \sum_{j=1}^{N}\varepsilon^{2m}\mathfrak{b}_{j}^{2m}\sq \tilde{\eta}_{j} \dq$ where
	$\mathfrak{i}_{j}^{*}\sq \tilde{\eta}_{j} \dq=[\eta_{j}]$
	with  $\mathfrak{i}_{j}:X_{\Gamma_{j},\Rep}\hookrightarrow \tilde{M}$,
	and 
	$\left|\mathfrak{b}_{j}^{2m} - \frac{|\Gamma_{j}|b_{j}}{2\st m-1 \dt}\right| \leq C \varepsilon^{\gamma}$, for some $\gamma>0$.
	
\end{teo}

\vspace{4mm}

\noindent The above results deserves few comments: first of all it would of great interest to interpret these new balancing conditions in terms of the algebraic data of the orbifold, at least when starting with a polarized object, very much in the spirit of Stoppa's interpretation of Arezzo-Pacard's result (\cite{Stoppa}).

\noindent Our results can also be seen as ``singular perturbation" results applied to the original singular space {\em{fixing}} the complex structure and deforming the K\"ahler class. A very different, though parallel in spirit, analysis can be done by thinking of keeping the K\"ahler class fixed and {\em{moving the complex structure}}. Unfortunately nobody has been able to prove gluing theorems for integrable complex structures so far, but {\em{assuming that such a deformation exists}}, this dual analysis, with no holomorphic vector fields and in complex dimension two, has been done in the important work on Spotti (\cite{Spotti}) in the Einstein  and special ordinary double point case, and by Biquard-Rollin (\cite{BiquardRollin}) in the Kcsc case for general $\QQ$-Gorenstein singularities. Finally, no doubt our results have a version for {\em{extremal}} metrics exactly as Arezzo-Pacard's papers did, notably in \cite{aps} and in the works of G. Szekelyhidi (\cite{Gabor1}, \cite{Gabor2}). In fact the second author will soon complete the proof of the analogue result (\cite{Lena}). In fact the present work and \cite{Gabor2} share some intriguing analogies. While technically, as the reader of this paper will see, a gluing theorem in the presence of a $4-2m$ term on the model is easier than the $2-2m$ case, since in \cite{Gabor2} one wants to keep control of the next asymptotic in the blow up, one is forced to face some problems similar to ours. It would  certainly  be interesting to go deeper in this connection.

\vspace{2mm}

\noindent Turning back to our results, we can then look for new examples of full or partial desingularizations of Kcsc orbifolds. 
Of course it will be very hard on a general orbifold to compute $\Delta_{\omega}\varphi_j$. On the other hand, assuming for example that $M$ is Einstein and using \begin{equation}
\Delta_{\omega}\varphi_{j}=-\frac{s_{\omega}}{m}\varphi_{j}\, ,
\end{equation}
\noindent the balancing condition requires only the knowledge of the value of the $\varphi_j$ at the singular points. Moreover these values are easily computed for example in toric setting by the well known relationship between the evaluation of the potentials $\varphi_j$ and the image point via the moment map. 
 With these classical observations one can then look for toric K\"ahler-Einstein orbifolds with isolated quotient singularities to test to which of them our results can be applied. In complex dimension $2$ things are pretty simple and in fact two such examples are 
\begin{itemize}
\item
$\st\PP^{1}\times\PP^{1},\pi_{1}^{*}\omega_{FS}+\pi_{2}^{*}\omega_{FS}\dt$ with $\ZZ_{2}$ acting by 
\begin{equation}\st[x_{0}:x_{1}],[y_{0}:y_{1}]\dt\longrightarrow  \st[x_{0}:-x_{1}],[y_{0}:-y_{1}]\dt\end{equation} 

\item

$\st\PP^{2},\omega_{FS}\dt$ with $\ZZ_{3}$ acting by
\begin{equation}[z_{0}:z_{1}:z_{2}]\longrightarrow  [x_{0}:\zeta_{3}x_{1}:\zeta_{3}^{2}x_{2}]\qquad \zeta_{3}\neq 1, \zeta_{3}^{3}=1\end{equation} 

\end{itemize}

\noindent In both cases we will show in Section $7$ that our results provide a {\em{full}} Kcsc desingularization (in the first case applied to $4$ singular $SU(2)$ points, while $3$ $SU(2)$ points in the second).

\noindent Working out higher dimensional examples turned out to be much more challenging than we expected. Even making use of the beautiful 
database of Toric Fano Threefolds run by G. Brown and A. Kasprzyk (\cite{GRD}, see also \cite{kasprzyk}) and their amazing help in implementing 
a complete search of Einstein ones with isolated singularities, we could only extract orbifolds where only a partial Kcsc resolution is possible.
In fact they produced a complete list (see \cite{GRD2}) of toric Fano threefolds s.t.

\begin{itemize}
\item
 they have only isolated quotient singular points;
\item
their moment polytope has barycenter in the origin (this implies the Einstein condition, thanks to a well known result by Mabuchi \cite{Mabuchi});
\item
each singular point is a $\CC^3/\Gamma$, $\Gamma \in U(3)$.
\end{itemize}

\noindent For example, let  $X^{(1)}$ be the toric K\"ahler-Einstein threefold whose  1-dimensional fan  $\Sigma^{(1)}_{1}$ is generated by points
\begin{equation}\Sigma^{(1)}_{1}=\left\{(1,3,-1), (-1,0,-1), (-1,-3,1), (-1,0,0), (1,0,0), (0,0,1), (0,0,-1), (1,0,1)\right\}
\end{equation}
and its $3$-dimensional fan $\Sigma^{(1)}_{3}$  is generated by $12$ cones 
\begin{align}
C_{1}:=& \left<        (-1,  0, -1),(-1, -3,  1),(-1,  0,  0)\right>\\
C_{2}:=& \left<        ( 1,  3, -1),(-1,  0, -1),(-1,  0,  0)\right>\\
C_{3}:=& \left<        (-1, -3,  1),(-1,  0,  0),( 0,  0,  1)\right>\\
C_{4}:=&\left<        ( 1,  3, -1),(-1,  0,  0),( 0,  0,  1)\right>\\
C_{5}:= &\left<        ( 1,  3, -1),(-1,  0, -1),( 0,  0, -1)\right>\\
C_{6}:=&\left<        (-1,  0, -1),(-1, -3,  1),( 0,  0, -1)\right>\\
C_{7}:= & \left<        (-1, -3,  1),( 1,  0,  0),( 0,  0, -1)\right>\\
C_{8}:=& \left<        (1,  3, -1),(1,  0,  0),(0,  0, -1)\right>\\
C_{9}:= &\left<        (1,  3, -1),(0,  0,  1),(1,  0,  1)\right>\\
C_{10}:=& \left<        (-1, -3,  1),( 1,  0,  0),( 1,  0,  1)\right>\\
C_{11}:=& \left<        (1,  3, -1),(1,  0,  0),(1,  0,  1)\right>\\
C_{12}:= & \left<        (-1, -3,  1),( 0,  0,  1),( 1,  0,  1)\right>
\end{align}

\noindent All these cones are singular and $C_{1},C_{4},C_{5},C_{7},C_{11},C_{12}$ are cones relative to affine open subsets of $X^{(1)}$ containing a $SU(3)$ singularity, while the others  are cones relative to affine open subsets of $X^{(1)}$ containing a $U(3)$ (non Ricci flat) singularity. 
\noindent We will show in Section $7$ that these $6$ $SU(3)$ singularities do satisfy all the requirements of Theorem \ref{maintheorem}.

\vspace{3mm}

\noindent {\bf{Aknowledgments:}} We wish to thank Frank Pacard and Gabor Szekelyhidi for many discussions on this topic. We also wish to express our deep gratitude to Gavin Brown and Alexander Kasprzyk for their help in not drowning in the Fano toric threefolds world.

\section{Notations and preliminaries}


\subsection{Eigenfunctions and eigenvalues of $\Delta_{\Sp^{2m-1}}$.}

\label{eigen}

In order to fix some notation which will be used throughout the paper, we agree that $\Sp^{2m-1}$ is the unit sphere of real dimension $2m-1$, equipped with the standard round metric inherited from $(\mathbb{C}^m, g_{eucl})$. We will denote by $\{\phi_k\}_{k\in \NN}$ a complete orthonormal system of the Hilbert space $L^2(\Sp^{2m-1})$, given by eigeinfunctions of the Laplace-Beltrami operator $\Delta_{\Sp^{2m-1}}$, so that, for every $k \in \NN$,
$$
\Delta_{\Sp^{2m-1}} \phi_k \,  = \,  \lambda_k \phi_k
$$ 
and $\{\lambda_k \}_{k \in \mathbb{N}}$ are the eigenvalues of $\Delta_{\Sp^{2m-1}}$ {\em{counted with multiplicity}}. We will also indicate by $\Phi_j$ the generic element of the $j$-th eigenspace of $\Delta_{\Sp^{2m-1}}$, so that, for every $j \in \NN$, 
$$
\Delta_{\Sp^{2m-1}} \Phi_j \,  = \,  \Lambda_j \Phi_j \, 
$$ and $\{\Lambda_j \}_{j \in \mathbb{N}}$ are the eigenvalue of $\Sp^{2m-1}$ {\em{counted without multiplicity}}. In particular, we have that
$\Lambda_j =  - j(2m - 2 + j)$, for every $j \in \NN$. If $\Gamma \triangleleft U(m)$ is a finite subgroup of the unitary group acting on $\CC ^m$ having the origin as its only fixed point, we denote by $\{\Lambda^{\Gamma}_j\}_{j \in \NN}$ the eigenvalues {\em{counted without multiplicity}} of the operator $\Delta_{\Sp^{2m-1}}$ restricted to the $\Gamma$-invariant functions.

\subsection{The scalar curvature equation}

We let $(M,g,\omega)$ be a \K\ orbifold with complex dimension equal to $m$, where $g$ is the \K\ metric and $\omega$ is the \K\ form. Notice that we allow the Riemannian manifold $(M,g)$ to be incomplete, since in the following we will be eventually led to consider punctured manifolds. We denote by $s_\omega$ the scalar curvature of the \K\ metric $g$ and by $\rho_\omega$ its Ricci form.
%
In the following it will be useful to consider cohomologous deformations of the \K\ form $\omega$. Hence, for a smooth real function $f \in C^\infty(M)$ such that $\omega + i  \dd  f  >  0 $, we set
\[
\omega_f \,   =  \, \omega + i  \dd  f \,  ,
\]
and we will refer to $f$ as the deformation potential.
Since we want to understand the behavior of the scalar curvature under deformations of this type, it is convenient to consider the following differential operator 
$$
\mathbf{S}_\omega(\cdot) \, : \, C^\infty(M) \longrightarrow C^\infty(M) \, , \qquad \qquad
 f \,\longmapsto \,\mathbf{S}_{\omega}(f) := s_{\omega+ i\dd f} \, ,
$$ 
which associate to a deformation potential $f$ the scalar curvature of the corresponding metric. Following the formal computations given in~\cite{ls}, we obtain the formal expansion 
\begin{equation}
\label{eq:espsg}
\mathbf{S}_{\omega}(f) \,\, = \,\,  s_{\omega} - \, \frac{1}{2} \, {\mathbb L}_{\omega}f \, + \,  \mathbb{N}_{\omega}(f) \, ,
\end{equation}
where the linearized scalar curvature operator $\Lg$ is given by
\begin{equation}
{\mathbb L}_\omega f \,  =  \, \Delta^{2}_\omega f \,   +  \, 4 \, \langle \, \rho_\omega \, | \,   i\dd f \, \rangle  \, , \label{eq:defLg}
\end{equation}
whereas the nonlinear remainder $\NN_g$ takes the form
\begin{equation}
\label{eq:defQg}
\NN_{\omega}(f) \, = \,  \tr\,(i\dd f \circ i\dd f \circ \rho_\omega) \, - \,  \tr \, (i\dd f \circ i\dd \, \Delta_{\omega}f ) \, + \, \frac{1}{2} \Delta_{\omega} \, \tr \, (i\dd f\circ i\dd f) \, + \,  \RR_{\omega}(f)\, ,
\end{equation}
and $\RR_\omega(f)$ is the collections of all higher order terms.

\subsection{The K\"ahler potential of a Kcsc orbifold}

We let $(M,g,\omega)$ be a compact constant scalar curvature \K\ orbifold  without boundary with complex dimension equal to $m$. Unless otherwise stated the singularities are assumed to be isolated.
Combining the local $\dd$-lemma with the equations of the previous subsection, we are now in the position to give a more precise description of the local structure of the \K\ potential of a Kcsc metric.

\begin{prop}
\label{proprietacsck}
Let $(M,g, \omega)$ be a \K\ orbifold. Then, given any point $p\in M$,
 there exists a holomorphic coordinate chart $(U, z^1, \ldots, z^m)$ centered at $p$ such that the \K\ form can be written as
$$
\omega \,\, = \,\, i \dd \, \bigg(\,\frac{|z|^2}{2} + \psi_\omega \bigg) \, , \qquad \hbox{with} \qquad \psi_\omega \, = \, \mathcal{O}(|z|^4) \, .
$$
If in addition the scalar curvature $s_g$ of the metric $g$ is constant, then $\psi_g$ is a real analytic function on $U$, and one can write
\begin{equation}
\label{eq:decpsig}
\psi_\omega (z, \overline{z}) \, = \,  \sum_{k=0}^{+\infty}\Psi_{4+k}(z, \overline{z}) \, ,
\end{equation}
where, for every $k\in \mathbb{N}$, the component $\Psi_{4+k}$ is a real homogeneous polynomial in the variables $z$ and $\overline{z}$ of degree $4+k$. In particular, we have that $\Psi_4$ and $\Psi_5$ satisfy the equations
\begin{align}
\label{eq:bilapp4}
\Delta^2 \, \Psi_4 & =  -2s_\omega \, ,\\
\label{eq:bilapp5} 
\Delta^2 \, \Psi_5 & =  0 \, ,
\end{align}
where $\Delta$ is the Euclidean Laplace operator of $\mathbb{C}^m$. Finally, the polynomial $\Psi_4$ can be written as
\begin{equation}
\label{eq:decp4}
\Psi_4\st z, \overline{z}\dt \,\, = 
\,\,  \Big(- \frac{s_{\omega}}{16m(m +1)} \, + \, \Phi_2 \, + \,  
\Phi_4  \, \Big) \,  |z|^4 \, ,
\end{equation}
where $\Phi_2$ and $\Phi_4$ are functions in the second and fourth eigenspace of $\Delta_{\mathbb{S}^{2m-1}}$, respectively.

%

%


%





%



\end{prop}

\begin{proof}
Without loss of generality, we assume that $p$ is a smooth point, since, if it is not, it is sufficient to consider the local lifting of the quantities involved. The first assertion is a consequence of the $\dd$-lemma combined with the existence of normal coordinates and it is a classical fact. The real analiticity of $\psi_\omega$ follows by elliptic regularity of solutions of the constant scalar curvature equation $\mathbf{S}_{eucl}(\psi_\omega) \, = \, s_\omega$, which, according to~\eqref{eq:espsg}, \eqref{eq:defLg} and \eqref{eq:defQg}, reads
\begin{equation}
\Delta^2\psi_\omega \,\, = \,\,  -2s_\omega \, + \, 8 \,  \tr(i\dd \psi_\omega \circ i\dd \Delta \psi_\omega) \, + \, 4 \, \Delta \, \tr (i\dd \psi_\omega \circ i\dd \psi_\omega) \, + \,  2 \, \RR_{eucl}(\psi_\omega) \,.
\end{equation}
Having the expansion~\eqref{eq:decpsig} at hand, the equations \eqref{eq:bilapp4}, \eqref{eq:bilapp5} are now obvious, while to prove equation \eqref{eq:decp4}
we just observe that since $\Psi_4$ is a real polynomial of order $4$, it must be an even function. In particular, its restriction to $\mathbb{S}^{2m-1}$ is forced to have trivial projection along the eigenspaces of $-\Delta_{\Sp^{2m-1}}$ corresponding to the eigenvalues $\Lambda_{2k+1}$, for every $k\geq 0$. Hence, $\Psi_4$ can be written as
$$
\Psi_4 \st z,\overline{z}\dt \,\, = \,\,  \big( \Phi_0 + \Phi_2 + \Phi_4  \big) |z|^{4} \, ,
$$
where the $\Phi_k$'s are functions in the $k$-th eigenspace of $\Delta_{\Sp^{2m-1}}$. The fact that $\Phi_0 = - s_\omega/16m(m+1)$ is now an easy consequence of equation~\eqref{eq:bilapp4}.
\end{proof}

\subsection{The \K\ potential of a scalar flat ALE \K\ manifold.}

We start by recalling the concept of Asymptotically Locally Euclidean (ALE for short) K\"ahler manifold. We let $\Gamma\triangleleft U(m)$ be a finite subgroup of the unitary group and we say that a complete noncompact \K\ manifold $(X_\Gamma, h, \eta)$ of complex dimension $m$, where $h$ is the \K\ metric and $\eta$ is the \K\ form, is an ALE \K\ manifold with group $\Gamma$ if there exist a positive radius $R>0$ and a quotient map   
$\pi: X_{\Gamma}\rightarrow \CC^{m}/\Gamma$,
such that 
\begin{equation}
\pi: X_{\Gamma}\setminus \pi^{-1}(B_R)  \longrightarrow \st\CC^{m} \setminus B_{R} \dt/\Gamma
\end{equation}
is a biholomorphism and in standard Euclidean coordinates the metric $\pi_*h$ satisfies the expansion 
\begin{equation}
\left|  \frac{\partial^\alpha}{\partial x^\alpha} \left(  \big (\pi_{*}h)_{i \bar{j}}  \, - \, \frac{1}{2} \, \delta_{i\bar{j}}  \right) \right| \,\, = \,\,  \mathcal{O}\st |x|^{-\tau - |\alpha|}\dt\,, 
\end{equation}
for some $\tau>0$ and every multindex $\alpha \in \NN^m$. 

\begin{remark}
In the following, we will make as systematic use of $\pi$ as an identification and, consequently, we will make no difference between $h$ and $\pi_{*} h$ as well as between $\eta$ and $\pi_* \eta$.
\end{remark}

%

\begin{remark}
\label{nolinear}
It is a simple exercise to prove that if $\Gamma$ is nontrivial, then there are no $\Gamma$-invariant linear functions on $\CC ^m$, and thus, with the notations introduced in section~\ref{eigen}, we have that $\Lambda^{\Gamma}_1 > \Lambda_1$. This will be repeatedly used in our arguments.
\end{remark}

We are now ready to present a result which describe the asymptotic behaviour of the \K\ potential of a scalar flat ALE \K\ metric. This can be though as the analogous of Proposition~\ref{proprietacsck}. We omit the proof because in the spirit it is very similar to the one of the aforementioned proposition and the details can be found in~\cite{ap1} 

\begin{prop}\label{asintpsieta2}
Let $(X_\Gamma, h, \eta)$ be a scalar flat ALE \K\ manifold, with $\Gamma \triangleleft U(m)$, and let $\pi: X_\Gamma \to \mathbb{C}^m/\Gamma$ be the quotient map. Then for $R>0$ large enough, we have that on $X_\Gamma\setminus \pi^{-1}(B_R)$ the K\"ahler form can be written as
\begin{equation}
\eta \,\, = \,\, i\partial\overline{\partial} \st \,  \frac{|x|^2}{2} \, + \, e({\Gamma}) \, |x|^{4-2m}  \, -  \, c({\Gamma}) \, |x|^{2 - 2m} \,  + \, \psi_{\eta}\st x \dt \dt \,,  \qquad \hbox{with} \qquad \psi_\eta \, = \,  \mathcal{O}(|x|^{-2m}) \, ,
\end{equation}
for some real constants $e({\Gamma})$ and $c({\Gamma})$.
Moreover, the radial component $\psi_{\eta}^{(0)}$ in the Fourier decomposition of $\psi_\eta$ is such that
$$
\psi_{\eta}^{(0)}\st |x| \dt=\mathcal{O}\st |x|^{6-4m} \dt \, .
$$ 
\end{prop}

In the case where the ALE \K\ metric is Ricci-flat it is possible to improve the estimates for the deviation of the \K\ potential from the Euclidean one. This is far form being obvious and in fact it is an improvement of an important result of Joyce (\cite{j}, Theorem 8.2.3 pag 175).

\begin{prop}
\label{asintpsieta}
Let $( {X}_{\Gamma}, h, \eta )$ be a Ricci flat $ALE$ \K\ manifold, with $\Gamma \triangleleft SU(m)$ nontrivial, and let $\pi: X_\Gamma \to \mathbb{C}^m/\Gamma$ be the quotient map. Then for $R>0$ large enough, we have that on $X_\Gamma\setminus \pi^{-1}(B_R)$ the K\"ahler form can be written as
\begin{equation}
{\eta} \,\, = \,\, i\partial\overline{\partial} \st \, \frac{|x|^2}{2} \, - \, c({\Gamma}) \, |x|^{2 - 2m} \, + \, \psi_{\eta} \st x \dt \dt  \,,  \qquad \hbox{with} \qquad \psi_\eta \, = \,  \mathcal{O}(|x|^{-2m}) \, ,
\end{equation}
for some positive real constant $c({\Gamma})>0$. 
Moreover, the radial component $\psi_{\eta}^{(0)}$ in the Fourier decomposition of $\psi_\eta$ is such that
$$
\psi_{\eta}^{(0)}\st |x| \dt=\mathcal{O}\st |x|^{2-4m} \dt \, .
$$ 
\end{prop} 

\begin{proof}
By~\cite[Theorem 8.2.3]{j}, we have that on $X_\Gamma\setminus\pi^{-1}(B_R)$ the K\"ahler form $\eta$ can be written  as
$$
{\eta} \,\, =  \,\, i\dd\st \, \frac{|x|^{2}}{2} - c({\Gamma}) \, |x|^{2 - 2m} +\psi_{\eta}\st x \dt\dt   
\quad\quad \hbox{with} \quad\quad 
\psi_{\eta}\st x \dt=\mathcal{O}\st |x|^{2 - 2m - \gamma} \dt \, ,
$$
for some $\gamma \in (0,1)$.
Since $({X}_{\Gamma},h)$ is Ricci flat, it is also scalar flat and so, arguing as in Proposition~\ref{proprietacsck}, we deduce that $\psi_{\eta}$ is a real analytic function. 
To obtain the desired estimates on the decay of $\psi_\eta$, we are going to make use of the equation $\mathbf{S}_{eucl} (\psi_{\eta} -c({\Gamma})|x|^{2-2m}) = 0$. By means of identity~\eqref{eq:espsg}, \eqref{eq:defLg} and \eqref{eq:defQg}, this can be rephrased in terms of $\psi_\eta$ as follows
\begin{align}
\label{eq:psieta}
\nonumber
\frac12{\Delta^2\psi_\eta} \,\, = &  \,\,\,   4\, \tr  \big(  i\dd \st \psi_{\eta} -c({\Gamma})\, |x|^{2-2m}\dt \circ i\dd\Delta\psi_{\eta} \big) \, \\
 &+ \,  2 \, \Delta \, \tr \big(  i\dd\st \psi_{\eta} -c({\Gamma}) \, |x|^{2-2m} \dt  \circ i\dd\st \psi_{\eta} -c({\Gamma}) \, |x|^{2-2m} \dt         \big) 
 \\&\,\, + \RR_{eucl}\big( \psi_{\eta} -c({\Gamma}) \,|x|^{2-2m}\big) \, ,
\end{align} 
where, in writing the first summand on the right hand side, we have used the fact that $\Delta |x|^{2-2m} = 0$. Since $\psi_\eta = \mathcal{O} ( |x|^{2 - 2m - \gamma} )$, for some $\gamma \in (0,1)$, it is straightforward to see that all of the terms on the right hand side can be estimated as $\mathcal{O}(|x|^{-2-4m-\gamma})$, with the only exception of the purely radial term 
$$
\Delta \,  \tr \big( (i\dd |x|^{2 - 2m} ) \circ (i\dd |x|^{2 - 2m} ) \big) \,\, = \,\, \mathcal{O}(|x|^{-2-4m})\,.
$$
For sake of convenience, we set now the right hand side of the above equation equal to $F/2$, so that 
$$
\Delta^2 \psi_\eta \,\, = \,\, F \, .
$$
It is now convenient to expand both $\psi_\eta$ and $F$ in Fourier series as 
$$
\psi_{\eta}(x) \, = \, \sum_{k=0}^{ +\infty} \psi_{\eta}^{(k)} (|x|) \, \phi_{k}(x/|x|) \quad \quad \hbox{and} \quad \quad F(x) \, = \, \sum_{k=0}^{ +\infty} F^{(k)}(|x|) \, \phi_{k}(x/|x|) \, ,
$$
where the functions $\{\phi_k\}_{k \in \mathbb{N}}$, are the eigenfunctions of the spherical laplacian $\Delta_{\Sp^{2m-1}}$ on $\Sp^{2m-1}$, counted with multplicity. Since $\phi_0 \equiv |\Sp^{2m-1}|^{-1/2}$, we will refer to $\psi_\eta^{(0)}$ and $F^{(0)}$ as the radial part of $\psi_\eta$ and $F$, respectively. We also notice that in the forthcoming discussion it will be important to select among the eigenfunctions $\phi_k$'s, only the ones which are $\Gamma$-invariant, in order to respect the quotient structure. So far, we have seen that $F^{(0)} = \mathcal{O}(|x|^{-2-4m})$ and $F^{(k)} = \mathcal{O}(|x|^{-2-4m-\gamma})$, for $k \geq 1$. On the other hand, using the linear ODE satisfied by the components $\psi_\eta^{(k)}$, it is not hard to see that their general expression is given by
$$
\psi_{\eta}^{(k)} (|x|)\,\, = \,\, a_k |x|^{4 - 2m - \alpha(k)} +b_k|x|^{2 - 2m - \alpha(k)}+ c_k |x|^{\alpha(k)} + d_k |x|^{\alpha(k) + 2} +\tilde{\psi}_{\eta}^{(k)}(|x|) \,,
$$
where, in view of the behavior of the $F^{(k)}$'s, the functions $\tilde{\psi}_\eta^{(k)}$ are such that
$$
\tilde{\psi}_\eta^{(0)} = \mathcal{O}(|x|^{2-4m}) \quad \quad \hbox{and} \quad \quad \tilde{\psi}_\eta^{(k)} = \mathcal{O}(|x|^{2-4m-\gamma}), \quad \hbox{for $k \geq 1$}\, ,
$$
and the integers $\alpha(k)$'s are such that $\alpha(k)=h$ if and only if $\phi_k$ belongs to the $h$-th eigenspace.
Since the cited Joyce's result implies that $\psi_\eta^{(k)} = \mathcal{O}(|x|^{2-2m-\gamma})$, it is easy to deduce that $c_k =0 =d_k$, for every $k \in \mathbb{N}$. Moreover, we have that $a_0 = 0 = b_0$ and thus $\psi_\eta^{(0)} = \mathcal{O}(|x|^{2-4m})$, as wanted. The same kind of considerations imply that the components $\psi_\eta^{(k)}$'s satisfy the desired estimates for every $k \geq 2m + 1$, that is for every $k$ such that $\alpha(k) \geq 2$. For $1\leq k \leq 2m$, we have that $a_k = 0$, but a priori nothing can be said about the $b_k$'s and thus at a first glance, one has that
$$
\psi_{\eta}^{(k)} (|x|)\,\, = \,\, b_k|x|^{1 - 2m } +\tilde{\psi}_{\eta}^{(k)}(|x|) \, , \quad \hbox{for $1\leq k \leq 2m$} \, .
$$
As it has been pointed out in Remark~\ref{nolinear}, there are no $\Gamma$-invariant eigenfunctions for $\Delta_{\Sp^{2m-1}}$ in the first eigenspace. This means that the components $\psi_\eta^{k}$'s, with $1\leq k \leq 2m$ do not appear in the Fourier expansion of $\psi_\eta$ and hence $\psi_\eta(x) = \mathcal{O}(|x|^{-2m})$.
\end{proof}




\section{Linear analysis on a Kcsc orbifold}

\subsection{The bounded kernel of $\Lg$.} As usual we let $(M, g, \omega)$ be a compact Kcsc  orbifold with isolated singularities and we assume that the kernel of the linearized scalar curvature operator $\Lg$ defined in~\eqref{eq:defLg} is nontrivial, in the sense that it contains also nonconstant functions. By the standard Fredholm theory for self-adjoint elliptic operators, we have that such a kernel is always finite dimensional. Throughout the paper we will assume that it is $(d+1)$-dimensional and we will set
\begin{equation}
\label{nontrivial_ker}
\ker (\Lg) \,\, = \,\, span \, \{\varphi_0, \varphi_1, \ldots , \varphi_d \} \, ,
\end{equation}
where $\varphi_0 \equiv 1$, $d$ is a positive integer and $\varphi_1, \ldots, \varphi_d$ is a collection of linearly independent functions in $\ker(\Lg)$ with zero mean and normalized in such a way that $||\varphi_i||_{L^2(M)} = 1$, $i=1, \ldots, d$, for sake of simplicity. From~\cite{ls} we recover the following charachterization of $\ker(\Lg)$.
\begin{prop}
\label{espsg}
%
Let $(M,g,\omega)$ be a compact constant scalar curvature \K\ orbifold  with isolated singularities.
Then, the subspace of $\ker(\Lg)$ given by the elements with zero mean
is in one to one correspondence with the space of holomorphic
vector fields which vanish somewhere in $M$.
\end{prop}

The aim of this section is to study the solvability of the linear problem
\begin{equation}\label{eq:linear}
 \Lg u = f
\end{equation}
on the complement of the singular points in $M$. In order to do that, we introduce some notation as well as an appropriate functional setting. Let us recall that we are distinguishing between    points $\sg p_1,\ldots,p_{N}\dg$  with neighborhoods biholomorphic to  a ball of  $\CC^m/\Gamma_{j}$ with $\Gamma_{j}$ nontrivial such that $\CC^{m}/\Gamma_{j}$ admits an ALE  Kahler Ricci-flat resolution $\st X_{\Gamma_{j}},h,\eta_{j} \dt$ and points $\sg q_1,\ldots,q_{K}\dg$ with neighborhoods biholomorphic to a ball of  $\CC^m/\Gamma_{N+l}$ such that
$\CC^{m}/\Gamma_{N+l}$ admits a scalar flat ALE resolution $(Y_{\Gamma_{N+l}},k_{l},\theta_{l})$ with $e({\Gamma_{N+l}})\neq 0$. To simplify the notation we set
\begin{equation*}
\p \, := \, \sg p_{1},\ldots,p_{N}\dg, \quad \q \, := \, \sg q_{1},\ldots,q_{K}\dg  , \quad \hbox{and} \quad M_{\p,\q} \, := \, M\setminus \st \p\cup\q \dt\,.
\end{equation*}
We agree that, if $\q=\emptyset$, then $M_{ \p }:=M_{\p,\emptyset}$. To introduce weighted function spaces, we consider geodesics balls $B_{r_{0}}\st p_{j} \dt, B_{r_{0}}\st q_{l} \dt$ of radius $r_0>0$, with K\"ahler normal coordinates centered at the points $p_{j}$'s and $q_{l}$'s and we set
$$
M_{r_{0}} \,\, := \,\, M\setminus \bigg(   \bigcup_{j=1}^{N} B_{r_{0}}\st p_{j} \dt\,\cup\,  \bigcup_{l=1}^{K} B_{r_{0}}\st q_{l} \dt  \bigg) \, .
$$
For $\delta\in\mathbb{R}$ and $\alpha\in (0,1)$, we define the weighted H\"older space $C_{\delta}^{k,\alpha}\st M_{\p,\q} \dt$
as the set of functions $f\in C_{loc}^{k,\alpha}\st M_{\p,\q} \dt$ such that the norm
\begin{align*}
\left\|f\right\|_{C_{\delta}^{k,\alpha}\st M_{\p,\q} \dt}
\, := \,\,\,  &\left\|f\right\|_{C^{k,\alpha}( M_{r_{0}} )}
\hspace{-1,5cm}
&+\sup_{ 0<r\leq r_{0}} r^{-\delta}\sum_{j=1}^{N} \left\|\left.f(r\cdot)\right|_{B_{r_{0}}\st p_{j} \dt}\right\|_{C^{k,\alpha}\st B_{2}\setminus B_{1}  \dt}\\
& &+\sup_{ 0<r\leq r_{0}} r^{-\delta}\sum_{l=1}^{K}\left\|\left.f(r\cdot)\right|_{B_{r_{0}}\st q_{l} \dt}\right\|_{C^{k,\alpha}\st B_{2}\setminus B_{1}  \dt}
\end{align*}
is finite.
We observe that the typical function $f\in C_{\delta}^{4,\alpha}\st M_{\p,\q} \dt$
beheaves like
\begin{equation*}
f(\cdot)
\, \,= \,\, \mathcal{O}\big( d_{\omega}\st p_{j}, \cdot \dt^{\delta}\big) \,, \quad \hbox{on} \quad {B_{r_{0}}\st p_{j}\dt}\qquad \hbox{and} \qquad   f (\cdot)
\, \, = \,\, \mathcal{O}\big( d_{\omega}\st q_{j}, \cdot \dt^{\delta}\big)
\,, \quad \hbox{on} \quad {B_{r_{0}}\st q_{j}\dt} \, ,
\end{equation*}
where $d_{\omega}$ is the Riemannian distance induced by the Kahler metric $\omega$.

We are now in the position to solve equation~\eqref{eq:linear} in the case where the datum $f$ is {\em orthogonal} to $\ker(\Lg)$. By this we mean that, looking at $f$ as a distribution, we have \begin{equation}
\label{eq:orto}
\langle f  \, | \, \varphi_i \rangle_{\mathscr{D}' \times \mathscr{D} } \,\, = \,\, 0 \, ,
\end{equation}
for every $i=0,\ldots ,d$, where we denoted by $\langle \cdot \, | \, \cdot\cdot \, \rangle_{\mathscr{D}' \times \mathscr{D} }$ the distributional pairing and the functions $\varphi_i$'s are as in~\eqref{nontrivial_ker}. It is worth pointing out that since the functions in $\ker(\Lg)$ are smooth, everything makes sense.

To solve equation \eqref{eq:linear} we need to ensure the Fredholmness of the operator $\Lg$ on the functional spaces we have chosen. The Fredholm property depends heavily on the choice of weights, indeed the operator $\Lg$ is Fredholm if and only if the weight is not an indicial root (for definition of indicial roots we refer to \cite{ap2}) at any of the points $p_{j}$'s or $q_{l}$'s. Since in normal coordinates on a punctured ball, the principal part of our operator $\Lg$ is 'asymptotic' to the Euclidean Laplacian $\Delta$, then the set of indicial roots of $\Lg$ at the center of the ball coincides with the set of indicial roots of $\Delta$ at $0$. We recall that the set of indicial roots of $\Delta$ at $0$ is given by $\ZZ\setminus \sg 5-2m, \ldots ,-1 \dg$ for $m\geq 3$ and $\ZZ$ for $m=2$.

By the analysis in~\cite{ap1}, we recover the following result, which provides the existence of solutions in Sobolev spaces for the linearized equation together with {\em a priori} estimates in suitable weighted H\"older spaces.

\begin{teo}
\label{invertibilitapesatobase}
For every $f \in L^p(M)$, $p>1$, satisfying the orthogonality condition~\eqref{eq:orto}, there exists a unique solution $u \in W^{4,p}(M)$ to
$$
\Lg u \,\, = \,\, f \, ,
$$
which satisfy the condition~\eqref{eq:orto}.
Moreover, the following estimates hold true.
\begin{itemize}
\item
If $m\geq 3$ and in addition $f\in C^{0,\alpha}_{\delta-4}(M_{\p,\q})$ with $\delta \in (4-2m \, ,0 )$, then the solution $u$ belongs to $C^{4,\alpha}_{\delta}(M_{\p,\q})$ and satisfy the estimates
\begin{equation}
\left\|u\right\|_{C^{4,\alpha}_{\delta}(M_{\p,\q})}  \,\, \leq \,\,  C \, \left\|f\right\|_{C^{0,\alpha}_{\delta-4}(M_{\p,\q})} \, ,
\label{eq:stimapesatabase}
\end{equation}
for some positive constant $C>0$.
\item
If $m=2$ and in addition $f\in C^{0,\alpha}_{\delta-4}(M_{\p,\q})$ with $\delta \in (0 \,,1)$, then the solution $u$ belongs to $C^{4,\alpha}_{loc}(M_{\p,\q})$ and satisfy the following estimates
\begin{equation}
\bigg\| \,  u - \sum_{j=1}^{N}u ( p_{j}) \chi_{p_j}  - \sum_{l=1}^{K}u (q_{l})\chi_{q_l} \, \bigg\|_{C^{4,\alpha}_{\delta}(M_{\p,\q})}   \!\!\!\! + \,
\sum_{j=1}^{N}|u (p_{j})| \, + \sum_{l=1}^{K}|u(q_{l})|  \,\, \leq \,\,  C  \, \left\|f\right\|_{C^{0,\alpha}_{\delta-4}(M_{\p,\q})} \, ,
\label{eq:stimapesatabase2}
\end{equation}
where $C>0$ is a positive constant and the functions $\chi_{p_1}, \ldots, \chi_{p_N}$ and $\chi_{q_1}, \ldots, \chi_{q_K}$ are smooth cutoff functions supported on small balls centered at the points $p_{1}, \ldots, p_N$ and $q_1, \ldots, q_K$, respectively  and identically equal to $1$ in a neighborhood of these points.
\end{itemize}
\end{teo}

\begin{remark}
Some comments are in order about the choice of the weighted functional setting. Concerning the case $m\geq 3$ we observe that
the choice of the weight $\delta$ in the interval $(4-2m,0)$ is motivated by the fact that only for $\delta$ in this range the kernel of $\Lg$ viewed as an operator from $C_{\delta}^{4,\alpha}\st M_{\p,\q} \dt$ to $C_{\delta-4}^{0,\alpha}\st M_{\p,\q} \dt$ coincides with the bounded kernel, which has been denoted for short by $\ker\st \Lg \dt$.

In the case $m=2$ it is no longer possible to make a similar choice, since $4-2m$ becomes $0$ and thus, at a first glance, the  natural choice for the weight is not evident. One possibility is to take the weight in the first indicial interval before $0$, which for $m=2$ is given $(-1,0)$. In this case, one would get a functional space which is strictly larger than the bounded kernel $\ker\st \Lg \dt$. We prefer instead to choose the weight in the first indicial interval after $0$, which for $m=2$ is given by $(0,1)$. This time, the bounded kernel of $\Lg$ is no longer contained in the possible domains of our operator, since the functions belonging to these spaces have to vanish at points $\p$ and $\q$.  On one hand this is responsible for the more complicate expression in the {\em a priori} estimate~\eqref{eq:stimapesatabase2}, but one the other hand this choice of the weight will reveal to be more fruitful.
Indeed, in view of the linear analysis on ALE \K\ manifolds performed in section~\ref{lineareALE} and with the notation introduced therein, one has that the corresponding linearized scalar curvature operator
\begin{equation*}
\Le : C_{\delta}^{4,\alpha}\st  X_{\Gamma} \dt\rightarrow C_{\delta-4}^{0,\alpha}\st  X_{\Gamma} \dt
\end{equation*}
admits an inverse (up to a constant) for $\delta \in (0,1)$. Since the possibility of choosing the same weight for the linear analysis on both the base orbifold and the model spaces will be crucial in the subsequent nonlinear arguments, this yields a reasonable justification of our choices. In the same spirit, we point out that, for $m=3$ and $\delta \in (4-2m,0)$ the operator $\mathbb{L}_\eta$ defined above is invertible, as it is proven in Theorem~\ref{isomorfismopesati}.
\end{remark}

In order to drop the orthogonality assumption~\eqref{eq:orto} in Theorem~\ref{invertibilitapesatobase} and tackle the general case, we first need to investigate the behaviour of the fundamental solutions of the operator $\Lg$. This will be done in the following subsection.


\subsection{Multi-poles fundamental solutions of $\Lg$.}

The aim of this subsection is twofold. On one hand, we want to produce the tools for solving equation~\eqref{eq:linear} on $M_{\p,\q}$, when $f$ is not necessarily {\em orthogonal} to $\ker{(\Lg)}$. On the other hand, we are going to determine under which global conditions on $\ker(\Lg)$ we can produce a function, which near the singularities behaves like the principal non euclidean part of the \K\ potential of the corresponding ALE resolution. In concrete, building on Propositions~\ref{asintpsieta} and~\ref{asintpsieta2}, we aim to establish the existence of a function, which blows up like $|z|^{2-2m}$ near the $p_j$'s and like $|z|^{4-2m}$ near the $q_l$'s. Such a function will then be added to the original \K\ potential of the base manifold in order to make it closer to the one of the resolution. At the same time, for obvious reasons, it is important to guarantee that this new \K\ potential will produce on $M_{\p,\q}$ the smallest possible deviation from the original scalar curvature, at least at the linear level.
Thinking of $g$ as a perturbation of the flat metric at small scale, we have that $\Lg$ can be thought of as a perturbation of $\Delta^2$. Since $|z|^{2-2m}$ and $|z|^{4-2m}$ satisfy equations of the form
$$
\Delta^2(A|z|^{2-2m} + B |z|^{4-2m}) = C \Delta\delta_0 + D\delta_0 \,\, ,
$$
where $\delta_0$ is the Dirac distribution centered at the origin and $A,B,C$ and $D$ are suitable constants, we are led to study these type of equations on $M$ for the operator $\Lg$.


\begin{prop}
\label{balancrough}
Let $(M,g,\omega)$ be compact Kcsc orbifold of complex dimension $m$ and let ${\rm ker} {(\Lg)} = { span}\{\varphi_0, \varphi_1, \dots,\varphi_d\}$, as in~\eqref{nontrivial_ker}. Let then $\{p_1, \dots, p_N\}$ and $\{q_1, \dots, q_K\}$ be two disjoint sets of points in $M$ and let $(f_0, \ldots, f_d )$ be a vector in $\RR^{d+1}$. Assume that the following {\em linear balancing condition} holds
\begin{eqnarray}
\label{eq:generalbal}
f_i \, + \, \sum_{l=1}^K a_l\varphi_i(q_l) \, +\, \sum_{j=1}^Nb_j(\Delta\varphi_i)(p_j) \, + \, \sum_{j=1}^Nc_j\varphi_i(p_j)   &  = &  0 \,, \quad\quad\quad\qquad\qquad \hbox{$i = 1, \dots, d$} \, , \\
\label{eq:fixnu}
f_0 \, {\rm Vol}_\omega(M)   \, + \, \sum_{l=1}^K a_l \, + \, \sum_{j=1}^N c_j  & = & \nu  \, {\rm Vol}_\omega(M) \, ,
\end{eqnarray}
for some choice of the coefficients $\nu$,  ${\bf a}=(a_1, \dots, a_K)$, ${\bf b} =(b_1, \dots, b_N)$ and ${\bf c} =(c_1, \dots, c_N)$.
Then, there exist a distributional solution ${U} \in \mathscr{D}'(M)$ to the equation
\begin{equation}
\Lg   [{U}]  \,  + \, \nu  \,\,\,  = \,\,\,   \sum_{i=0}^d f_i \, \varphi_i \, + \, \sum_{l=1}^K a_l\, \delta_{q_{l}} \, + \, \sum_{j=1}^N  b_j \, \Delta\delta_{p_{j}} \, + \, \sum_{j=1}^N c_j \, \delta_{p_{j}} \, , \qquad \hbox{in \,\,\,$M$}\, . \label{eq:LGabcd}
\end{equation}
\end{prop}

\begin{proof}
Let us first remark that equations \eqref{eq:generalbal} and \eqref{eq:fixnu} imply that, for any $\varphi \in \ker(\Lg)$, one has that $\langle \, T \, | \,\varphi \, \rangle_{\mathscr{D}' \times \mathscr{D}} \, = \, 0$, where $T \in \mathscr{D}'$ is the distribution defined by
$$
T \,\,  = \,\,  \sum_{i=1}^{d}f_{i} \, \varphi_{i} \, + \, \sum_{l=1}^K a_l \, \delta_{q_{l}} \, + \, \sum_{j=1}^N b_j \, \Delta\delta_{p_{j}} \, + \, \sum_{j=1}^N c_j \, \delta_{p_{j}}  \, -  \, \nu \, .
$$
Having this in mind, we let $U \in \mathscr{D}'$ be the unique distribution such that, for every $\psi \in C^\infty(M)$
$$
\left< \, U \, |\, \psi\, \right>_{\mathscr{D}^{'} \times \mathscr{D}} \,\,  = \,\,
\left< \, T  \, |\,  \mathbb{J}_{\omega}[ \psi^{\perp}] \, \right>_{\mathscr{D}^{'} \times \mathscr{D}} \, ,
$$
where $\psi^\perp$, the component of $\psi$ which is {\em orthogonal} to $\ker(\Lg)$, is given by
\begin{equation*}
\psi^{\perp} \,\, = \,\, \psi \, - \, \frac{1}{\vol\st M \dt}\int_{M}\psi \,d\mu_{\omega} \, - \, \sum_{i=1}^{d}\varphi_{i}\int_{M}\psi \varphi_{i}\,d\mu_{\omega}\,,
\end{equation*}
and $\mathbb{J}_{\omega} : L^{2}\st M \dt/\ker\st \Lg \dt\rightarrow W^{4,2}\st M \dt/\ker\st \Lg \dt$ is inverse of $\Lg$ restricted to the orthogonal complement of $\ker(\Lg)$, given by Proposition \ref{invertibilitapesatobase}.
We claim that the distribution $U$ defined above satisfies the equation \eqref{eq:LGabcd} in the sense of distributions. With the notations just introduced, we need to show that, for every $\psi \in C^\infty(M)$, it holds
$$
\left< \, \Lg  [U ]\, |\, \psi \, \right>_{\mathscr{D}^{'} \times \mathscr{D}} \,\, = \,\,
\left< \, T \, | \, \psi \, \right> _{\mathscr{D}^{'} \times \mathscr{D}} \, .
$$
Using the definition of $U$ and the fact that $\Lg$ is formally selfadjoint, we compute
\begin{eqnarray*}
\left< \, \Lg  [U ]\, |\, \psi \, \right>_{\mathscr{D}^{'} \times \mathscr{D}} & = & \left< \, U \, |\, \Lg [\psi] \, \right>_{\mathscr{D}^{'} \times \mathscr{D}} \,\,\, = \,\,\, \left< \, U \, |\, \Lg [\psi^\perp] \, \right>_{\mathscr{D}^{'} \times \mathscr{D}} \,\,\, = \,\,\, \left< \, T \, |\, \mathbb{J}_\omega  \big[   (\Lg [\psi^\perp])^\perp  \big] \, \right>_{\mathscr{D}^{'} \times \mathscr{D}}  \\
& = & \left< \, T \, | \, \psi^\perp  \right> _{\mathscr{D}^{'} \times \mathscr{D}} \,\, \,= \,\,\, \left< \, T \, | \, \psi \, \right> _{\mathscr{D}^{'} \times \mathscr{D}} \,,
\end{eqnarray*}
since $\psi - \psi^\perp \in \ker(\Lg)$, and thus $\left< \, T \, | \, \psi - \psi^\perp  \right> _{\mathscr{D}^{'} \times \mathscr{D}} = 0$, by a previous observation.
%
This completes the proof of the proposition.
\end{proof}

\begin{remark}
\label{gabc}
When $f_i = 0$, for $i=0, \ldots, d$, we only impose the balancing condition~\eqref{eq:generalbal}, which specializes to
\begin{equation*}
\sum_{l=1}^K a_l\varphi_i(q_l) \, +\, \sum_{j=1}^Nb_j(\Delta\varphi_i)(p_j) \, + \, \sum_{j=1}^Nc_j\varphi_i(p_j)  \,\,  = \,\,  0 \,,
\end{equation*}
and we obtain a real number $\nu_{\aaa,\ccc}$, defined by the relation
\begin{equation*}
\sum_{l=1}^K a_l \, + \, \sum_{j=1}^N c_j  \,\,= \,\, \nu_{\aaa,\ccc}  \, {\rm Vol}_\omega(M) \, ,
\end{equation*}
and a distribution $\GGG_{\aaa,\bbb,\ccc} \in \mathscr{D}'(M)$, which satisfies the equation
\begin{eqnarray*}
\Lg \left[  \mathbf{G}_{\aaa,\bbb,\ccc} \right] \,  + \, \nu_{\aaa,\ccc}  & =&   \sum_{l=1}^K a_l\, \delta_{q_{l}} \, + \, \sum_{j=1}^N  b_j \, \Delta\delta_{p_{j}} \, + \, \sum_{j=1}^N c_j \, \delta_{p_{j}} \, , \qquad \hbox{in \,\,\,$M$}\, .
\end{eqnarray*}
We will refer to $\GGG_{\aaa,\bbb,\ccc}$ as a {\em multi-poles fundamental solution} of $\Lg$.
\end{remark}

The following two lemmata and the subsequent proposition~\eqref{loc_structure} will give us a precise description of the behavior of a {\em multi-poles fundamental solution}  $\GGG_{\aaa,\bbb,\ccc}$ of $\Lg$ around the singular points. The same considerations obviously apply to a distributional solution $U$ of the equation~\eqref{eq:LGabcd}. The first observation in this direction can be found in \cite{ap2} and we report it here for sake of completeness. 	
\begin{lemma}
\label{Gbilapl}
Let $(M,g,\omega)$ be a Kcsc orbifold of complex dimension $m\geq 2$ and let $M_q = M \setminus \{ q\}$, with $q \in M$. Then, the following holds true.
\begin{itemize}
\item If $m\geq 3$, there exists a function
$G_{\Delta\Delta}(q,\cdot) \in \mathcal{C}_{4-2m}^{4,\alpha}(M_q) \cap \mathcal{C}^{\infty}_{loc}(M_q)$, orthogonal to $\ker(\Lg)$ inthe sense of~\eqref{eq:orto}, such that
\begin{equation}\label{eq:greenbilapl}
\Lg[G_{\Delta\Delta}(q,\cdot)]  \,\, + \,\,
\frac{2(m-1) \, |\Sp^{2m-1}|}{|\Gamma|} \,\, \big[\, 4(m-2) \,\, \delta_q \, \big] \,\,  \in \, \mathcal{C}^{0,\alpha}(M) \, ,
\end{equation}
where $|\Gamma|$ is the order of the orbifold group at $q$.
Moreover, if $z$ are holomorphic coordinates centered at $q$, it holds the expansion
\begin{equation}
\label{eq:expgreenbilapl}
G_{\Delta\Delta}(q,z)  \,\, = \,\,  |z|^{4-2m} + \, \mathcal{O}(|z|^{6-2m}) \, .
\end{equation}
\item If $m=2$, there exists a function
$G_{\Delta\Delta}(q,\cdot) \in \mathcal{C}^{\infty}_{loc}(M_q)$, orthogonal to $\ker(\Lg)$ inthe sense of~\eqref{eq:orto}, such that
\begin{equation}
\label{eq:bluebilapl2}
\Lg[G_{\Delta\Delta}(q,\cdot)] \,\, - \,\, \frac{4|\Sp^{3}|}{|\Gamma|} \, \delta_q  \,\, \in  \,\, \mathcal{C}^{0,\alpha}(M) \, ,
\end{equation}
 where $|\Gamma|$ is the order of the orbifold group at $q$.
Moreover, if $z$ are holomorphic coordinates centered at $q$, it holds the expansion
\begin{equation}
\label{eq:expbluebilapl2}
G_{\Delta\Delta}(q,\cdot) \,\, = \,\, \log(|z|) \, + \, C_{q} \, + \, \mathcal{O}(|z|^{2}) \, ,
\end{equation}
for some constant $C_q \in \RR$.
\end{itemize}
\end{lemma}

Before stating the next lemma, it is worth pointing out that $G_{\Delta \Delta} (q, \cdot)$ has the same rate of blow up as the Green function of the biharmonic operator $\Delta^2$. Since we want to produce a local approximation of the {\em multi-poles fundamental solution} $\GGG_{\aaa,\bbb,\ccc}$\,, we also need a profile whose blow up rate around the singular points is the same as the one of the Green function of the Laplace operator. This will be responsible for the $\Delta \delta_p$'s terms. 
\begin{lemma}
\label{Glapl}
Let $(M,g,\omega)$ be a Kcsc orbifold of complex dimension $m\geq 2$ and let $M_p = M \setminus \{ p \}$, with $p \in M$. Then, the following holds true.
\begin{itemize}
\item
If $m \geq 3$, there exists a function
$G_{\Delta}(p,\cdot) \in \mathcal{C}_{2-2m}^{4,\alpha}(M_p) \cap \mathcal{C}_{loc}^{\infty}(M_p)$, orthogonal to $\ker(\Lg)$ inthe sense of~\eqref{eq:orto}, such that
\begin{equation}
\label{eq:greenlapl}
\Lg[G_{\Delta}(p,\cdot)] \,\,  - \,\, \frac{2(m-1) \, |\Sp^{2m-1}|}{|\Gamma|}  \,\, \Big[   \, \Delta\delta_p \,\,
+\,\,
\frac{s_\omega (m^2-m+2)}{m(m+1)}
\, \,  \delta_p  \, \Big]    \,\,  \in \, \mathcal{C}^{0,\alpha}(M) \, ,
\end{equation}
{where $|\Gamma|$ is the cardinality of the orbifold group at $p$} and $s_\omega$ is the constant scalar curvature of the orbifold.
Moreover, if $z$ are holomorphic coordinates centered at $p$, it holds the expansion
\begin{equation}
\label{eq:expgreenlapl}
G_{\Delta }(p,\cdot) \,\,  =  \,\, |z|^{2-2m}  \, + \, |z|^{4-2m} \, ( \, \Phi_2 + \Phi_4 \, )  \, + \,  |z|^{5-2m} \,
\sum_{j=0}^2\Phi_{2j+1}
 \, + \,  \mathcal{O}(|z|^{6-2m}) \, ,
\end{equation}
for suitable smooth $\Gamma$-invariant functions $\Phi_j$'s defined on $\mathbb{S}^{2m-1}$ and belonging to the $j$-th eigenspace of the operator $\Delta_{\mathbb{S}^{2m-1}}$.

\smallskip

\item
If $m = 2$,  there exists a function
$G_{\Delta}(p,\cdot) \in \mathcal{C}_{-2}^{4,\alpha}(M_p) \cap \mathcal{C}_{loc}^{\infty}(M_p)$, orthogonal to $\ker(\Lg)$ inthe sense of~\eqref{eq:orto}, such that
\begin{equation}
\label{eq:greenlapl2}
\Lg[G_{\Delta}(p,\cdot)] \,\,  - \,\, \frac{ |\Sp^{3}|}{|\Gamma|}  \,\, \Delta\delta_p \,\,
- \,\,\frac{ \,  s_\omega \, 2  \, |\Sp^{3}|}{3 \, |\Gamma|} \, \,  \delta_p  \,\,  \in \, \mathcal{C}^{0,\alpha}(M) \, ,
\end{equation}
where $|\Gamma|$ is the cardinality of the orbifold group at $p$ and $s_\omega$ is the constant scalar curvature of the orbifold.
Moreover, if $z$ are holomorphic coordinates centered at $p$, it holds the expansion
\begin{equation}
\label{eq:expbluelapl2}
G_{\Delta}(p,\cdot) \,\, = \,\,  |z|^{-2} \, + \, \log(|z|)(\Phi_2 + \Phi_4) \, + \, C_{p} \, + \,  |z| \,
 \sum_{h=0}^{2}\Phi_{2h+1} \,  + \,  \mathcal{O}(|z|^{2})
\end{equation}
for some constant $C_p \in \RR$, some $H \in \mathbb{N}$ and suitable smooth $\Gamma$-invariant functions $\Phi_h$'s defined on $\mathbb{S}^{3}$ and belonging to the  $h$-th eigenspace of the operator $\Delta_{\mathbb{S}^{3}}$. 
\end{itemize}
\end{lemma}

%

%




%

%

\begin{proof} We focus on the case $m\geq 3$ and since the computations for the case $m=2$ are very similar, we left them to the interested reader.
To prove the existence of $G_{\Delta}\st p,\cdot\dt$, we fix a coordinate chart centered at $p$ and we consider the Green function for the Euclidean Laplacian $|z|^{2-2m}$. In the spirit of Proposition~\ref{proprietacsck}, we compute
\begin{align*}
\Lg [ \,  |z|^{2-2m} \, ] \,\,\, =  \,\,\,&\st \, \Lg \, - \, \Delta^{2} \, \dt [\, |z|^{2-2m} \, ]\\
= \,\,\, &  - \, 4 \, \tr  \st \,  i\dd |z|^{2-2m}\circ i\dd\Delta \psi_{\omega} \, \dt  \, - \, 4 \, \tr \st \,  i\dd \psi_{\omega} \circ i\dd\Delta |z|^{2-2m} \, \dt\\
&-4\, \Delta\, \tr \st\,  i\dd \psi_{\omega} \circ i\dd|z|^{2-2m} \, \dt \, + \, \mathcal{O}\st|z|^{2-2m}\dt \\
=&-\frac{m}{4|z|^{2m}}\Delta^{2}\Psi_{4}+\frac{m\st m+1 \dt}{|z|^{2m+2}}\Delta\Psi_{4}-\frac{m}{4}\Delta\st \frac{\Delta\Psi_{4}}{|z|^{2m}} \dt\\
&+4m\st m+1 \dt\Delta\tr\st \frac{\Psi_{4}}{|z|^{2m+2}}  \dt+\mathcal{O}\st|z|^{2-2m}\dt
\end{align*}
where we used the explicit form of  $\Psi_{4}$	
\begin{equation}
\Psi_{4}\st z,\overline{z} \dt=-\frac{1}{4}\sum_{i,j,k,l=1}^{m}R_{i\bar{\jmath}k\bar{l}}z^{i}\overline{z^{j}}z^{k}\overline{z^{l}}
\end{equation}
and the complex form of the euclidean laplace operator
\begin{equation}
\Delta=4\sum_{i=1}^{m} \frac{\partial^{2}}{\partial z^{i}\partial \overline{z^{i}}}\,.  
\end{equation}
Expanding the real analytic function $\psi_\omega$ as $\psi_\omega \, = \, |z|^4 \, (\Phi_0 + \Phi_2 + \Phi_4) \, + \, |z|^5 \, (\Phi_1 + \Phi_3 + \Phi_5) \, + \, \mathcal{O}(|z^6|)$, where, for $h=0, 1,2$, the $\Phi_{2h}$'s and the $\Phi_{2h+1}$'s are suitable $\Gamma$-invariant functions in the $h$-th eigenspace of $\Delta_{\Sp^{2m-1}}$, we obtain
\begin{align*}
\Lg [ \,  |z|^{2-2m} \, ] \,\,\, =  \,\,\, &{|z|^{-2m}} \, \sum_{h=0}^2 c_{2h} \, \Phi_{2h} \, + \,
|z|^{1-2m} \, \sum_{h=0}^2 c_{2h+1} \, \Phi_{2h+1} \,  + \, \mathcal{O}\st |z|^{2-2m} \dt\,,
\end{align*}
where $c_0, \ldots, c_5$ are suitable constants. It is a straightforward but remarkable consequence of formula~\eqref{eq:decp4}, the fact that $c_0=0$.
\,It is then possible to introduce the corrections
$$
V_{4} \,\, = \,\, |z|^{4-2m} \, (\, C_2 \, \Phi_2 \, + \, C_4 \, \Phi_4 \,)
\qquad \hbox{and} \qquad  V_{5} \,\, = \,\, |z|^{5-2m} \, \sum_{h=0}^2 C_{2h+1} \, \Phi_{2h+1}  \,,
$$
where the coefficients $C_1, \ldots, C_5$ are so chosen that
$$
\Delta^2 \, [\, V_4 \, + \,  V_5 \, ] \,\, = \,\, {|z|^{-2m}} \,
(\, c_2 \, \Phi_2 \, + \, c_4 \, \Phi_4 \,)
\, + \,
|z|^{1-2m} \, \sum_{h=0}^2 c_{2h+1} \, \Phi_{2h+1} \, .
$$
This implies in turn that
$
\Lg \,\big[ \,  |z|^{2-2m}  - \, V_4 \,- \, V_5  \,\big] \, \, = \, \, \mathcal{O}(|z|^{2-2m}) \, .
$
Using the fact that in normal coordinates centered at $p$ the Euclidean biharmonic operator $\Delta^2$ yields a good approximation of $\Lg$, it is not hard to construct a function $W\in C_{6-2m}^{4,\alpha} ( B_{r_{0}}^{*} )$ on a sufficiently small punctured ball $B_{r_0}^*$ centered at $p$, such that
$$
\Lg \,\big[ \,  |z|^{2-2m}  - \, V_4 \,- \, V_5  \, - \, W \,\big] \, \,  \in \,\,   C^{0,\alpha} ( B_{r_{0}}^{*} )    \, .
$$
By means of a smooth cut-off function $\chi$, compactly supported in $B_{r_0}$ and identically equal to $1$ in $B_{r_0/2}$, we obtain a globally defined function in $L^1(M)$, namely
$$
U_p \,\, = \,\,  \chi \, \bigg( \,   |z|^{2-2m}  - \, |z|^{4-2m} \, (\, C_2 \, \Phi_2 \, + \, C_4 \, \Phi_4 \,)  \, - \,  |z|^{5-2m} \, \sum_{h=0}^2 C_{2h+1} \, \Phi_{2h+1}  \, - \, W \, \bigg)
$$
In order to guarantee the orthogonality condition~\eqref{eq:orto}, we set
\begin{equation*}
G_{\Delta} (p,\cdot \,)  \,\,\, = \,\,\, U_{p} (\cdot) \,\, - \,\,  \frac{1}{\vol\st M \dt} \int_{M} U_{p} \, d\mu_{\omega} \,\, -\,\, \sum_{i=1}^{d} \, \varphi_{i}(\cdot)\int_{M} U_{p} \, \varphi_{i} \, d\mu_{\omega}
\end{equation*}
and we claim that $\Lg [G_{\Delta} (p,\cdot \,) ]$ satisfies the desired distributional identity. To see this, we set $M_{\varepsilon} \, = \, M\setminus B_{\varepsilon}$, where $B_\varepsilon$ is a ball of radius $\varepsilon$ centered at $p$, and we integrate $\Lg [G_{\Delta} (p,\cdot \,) ] \, = \, \Lg\, [U_p]$ against a test function $\phi \in C^{\infty}(M)$.
Setting
$$
\rho_{\omega}^{0} \,\, = \,\, \rho_{\omega} \, - \, \frac{s_{\omega}}{2m}\omega \, ,
$$
and using formula \eqref{eq:defLg}, it is convenient to write
\begin{equation*}
\Lg [U_p] \,\, = \,\, \Delta_{\omega}^{2} \, U_p\, + \, \frac{s_{\omega}}{m}\, \Delta_{\omega} \, U_p\, + \, 4\left<\, \rho_{\omega}^{0} \, | \, i\dd U_p \, \right> \,,
\end{equation*}
so that we have
\begin{align*}
\int_{M_{\varepsilon}}  \phi \,\, \Lg\sq U_p \dq \,d\mu_{\omega} \,\,\,  = \,\,\, &\int_{M_{\varepsilon}}  \phi \,\Big(\Delta_{\omega}^{2} \, + \, \frac{s_{\omega}}{m}\Delta_{\omega} \Big) \big[ U_p \big]\, d\mu_{\omega}
\, + \, 4 \int_{M_{\varepsilon}}  \phi \,\left<\, \rho_{\omega}^{0} \, | \, i\dd U_p \, \right>\, d\mu_{\omega}\,.
\end{align*}
We first integrate by parts the first summand on the right hand side and we take the limit for $\varepsilon \to 0$, obtaining
\begin{align*}
\lim_{\varepsilon\rightarrow 0} \int_{M_{\varepsilon}} \!\! \phi \,\Big(\Delta_{\omega}^{2} \, + \, \frac{s_{\omega}}{m}\Delta_{\omega} \Big) \big[ U_p \big]\, d\mu_{\omega}
\,\,\, = \,\,\, & \int_{M} \!\! U_p  \, \,\Big(\Delta_{\omega}^{2} \, + \, \frac{s_{\omega}}{m}\Delta_{\omega} \Big) \big[ \phi \big]\, d\mu_{\omega} \, + \, \lim_{\varepsilon\rightarrow 0}\int_{\partial M_{\varepsilon}} \!\!\! \phi \,\, \partial_{\nu}(\Delta_{\omega} U_p) \, d\sigma_{\omega}\\
&+\lim_{\varepsilon\rightarrow 0}\int_{\partial M_{\varepsilon}} \!\!\!\!   (\Delta_{\omega} \phi) \,\, \partial_{\nu} U_p \, d\sigma_{\omega}
\, + \, \frac{s_{\omega}}{m} \, \lim_{\varepsilon\rightarrow 0}\int_{\partial M_{\varepsilon}} \!\!\!\! \phi \,\, \partial_{\nu} U_p \, d\sigma_{\omega}
\end{align*}
where $d\sigma_{\omega}$ is the restriction of the measure $d\mu_{\omega}$ to $\partial M_{\varepsilon}$ and $\nu$ is the exterior unit normal to $\partial M_{\varepsilon}$. Combining the definition of $U_p$ with the standard development of the area element, it is easy to deduce that
$$
\lim_{\varepsilon\rightarrow 0}\int_{\partial M_{\varepsilon}} \!\!\!\!   (\Delta_{\omega} \phi) \,\, \partial_{\nu} U_p \, d\sigma_{\omega}
\, + \, \frac{s_{\omega}}{m} \, \lim_{\varepsilon\rightarrow 0}\int_{\partial M_{\varepsilon}} \!\!\!\! \phi \,\, \partial_{\nu} U_p \, d\sigma_{\omega} \,\,\, = \,\,\, \frac{2\, (m-1) \, |\Sp^{2m-1}|}{\left|\Gamma\right|} \, \Big[ \,  \Delta_{\omega} \phi \, ( p ) \, + \, \frac{s_{\omega}}{m} \, \phi \, (p) \,  \Big] \, .
$$
To treat the last boundary term, we use Proposition~\ref{proprietacsck} and we compute
\begin{align*}
\partial_{\nu}\, (\Delta_{\omega} U_p) \,\,\,
&= \,\,\, {\left|z\right|^{1-2m}}   \Big(  \, \frac{2s_{\omega} \st m-1 \dt^{3}}{m\st m+1 \dt } \, + \, K_2 \, \Phi_{2} \, + \, K_4 \, \Phi_{4} \, \Big) \, + \, \mathcal{O}\big( \left|z\right|^{2-2m} \big)\,,
\end{align*}
for suitable constants $K_2$ and $K_4$. Hence, we get
$$
\lim_{\varepsilon\rightarrow 0}\int_{\partial M_{\varepsilon}} \!\!\! \phi \,\,  \partial_{\nu}(\Delta_{\omega} U_p) \, d\sigma_{\omega} \,\,\, = \,\,\, \frac{2\, (m-1) \, |\Sp^{2m-1}|}{\left|\Gamma\right|}  \, \Big[ \frac{s_\omega (m-1)^2}{m(m+1)} \, \phi \, (p) \Big] \, .
$$
In conclusion we have that
\begin{align*}
\left<\, \Big(\Delta_{\omega}^{2} \, + \, \frac{s_{\omega}}{m}\Delta_{\omega} \Big) \big[ U_p \big]  \, \Big| \, \phi    \,\right>_{\mathscr{D}' \times \mathscr{D}} \,\, = \,\, &\int_{M} \!\! U_p  \, \,\Big(\Delta_{\omega}^{2} \, + \, \frac{s_{\omega}}{m}\Delta_{\omega} \Big) \big[ \phi \big]\, d\mu_{\omega} \\
& + \, \frac{2\, (m-1) \, |\Sp^{2m-1}|}{\left|\Gamma\right|} \, \Big[ \Delta_\omega \phi \,(p) \, + \, \frac{s_\omega (m^2-m+2)}{m(m+1)} \, \phi \, (p) \, \Big] \, .
\end{align*}

We now pass to consider the term contanining $\rho_{\omega}^{0}$. An integration by parts gives
\begin{align*}
\lim_{\varepsilon\rightarrow 0}\int_{M_{\varepsilon}}\!\!\! \phi \,  \left< \, \rho_{\omega}^{0} \, | \,  i\dd U_p \, \right> \, d\mu_{\omega}
\,\,\, = \,\,\, & \int_{M}   U_p \,\left< \, \rho_{\omega}^{0} \, | \, i\dd \phi \, \right>\,d\mu_{\omega} \\
& + \, \lim_{\varepsilon\rightarrow 0}  \int_{\partial M_{\varepsilon}} \!\!\! \phi \,\,
X(U_p) \lrcorner \,  d\mu_{\omega} \, + \, \lim_{\varepsilon\rightarrow 0}  \int_{\partial M_{\varepsilon}} \!\!\! U_p \,\,
\overline{X(\phi)} \lrcorner \,  d\mu_{\omega}  \, ,
\end{align*}
where, for a given function $u \in C^1(M_p)$, the vector field $X(u)$ is defined as $
X(u)  =  \big( \, \rho_\omega^0 (\partial^\sharp u \, , \, \cdot \, ) \,  \big)^\sharp $. It is easy to check that second boundary term vanishes in the limit. We claim that the same is true for the first boundary term. To prove this, we recall the expansions
\begin{eqnarray*}
\st\rho_{\omega}^{0}\dt_{i\bar{\jmath}} &=& \st \lambda_{i}\st p \dt- \frac{s_{\omega}}{2m} \dt\delta_{i\bar{\jmath}} \, + \, \mathcal{O}\st |z| \dt\,,\\
\partial^{\sharp} U_p &=&\sum_{i=1}^{m}   \big(  \st 1-m \dt \, {|z|^{-2m}}{z^{i}} \, + \, \mathcal{O}\st |z|^{2-2m}\dt  \big)  \, \frac{\partial}{\partial z^i} \\
d\mu_{\omega}&=&\st 1 +\mathcal{O}\st |z|^{2} \dt\dt \, d\mu_{0}\,,
\end{eqnarray*}
where the $\lambda_{i}$'s are the eigenvalues of the matrix $\st \rho_{\omega}^{0}\dt_{i\bar{\jmath}}$ and $d \mu_0$ is the Euclideam volume form. This implies
\begin{align*}
X(U_p) \,  \lrcorner \, d \mu_{\omega}
\,\,\,\, = \,\,\,\, &\st 1-m \dt  \, \sum_{i=1}^{m}\st \lambda_{i}\st p \dt- \frac{s_{\omega}}{2m} \dt  z^{i} \, \frac{\partial}{\partial z^{i}} \, \lrcorner \, d\mu_{0}  \, + \, \mathcal{O}\st |z| \dt\,.
\end{align*}
On the other hand, by the symmetry of $d \mu_0$, it is easy to deduce that
\begin{equation*}
\int_{\partial M_{\varepsilon}} z^{1} \, \frac{\partial}{\partial z^{1}} \, \lrcorner  \, d\mu_{0} \,\, = \,\, \ldots \,\, = \,\, \int_{\partial M_{\varepsilon}}z^{m} \, \frac{\partial}{\partial z^{m}} \, \lrcorner  \, d\mu_{0} \, .
\end{equation*}
The claim is now a straightforward consequence. In synthesis, we have obtained
\begin{align*}
\left<\, \Lg \big[ U_p \big]  \, \big| \, \phi   \,\right>_{\mathscr{D}' \times \mathscr{D}} \,\,\, = \,\,\, &\int_{M} \!\! U_p  \, \,\Lg \big[ \phi \big]\, d\mu_{\omega} \, + \, \frac{2\, (m-1) \, |\Sp^{2m-1}|}{\left|\Gamma\right|} \, \Big[ \,\Delta_\omega \phi \,(p) \,\, + \,\, \frac{s_\omega (m^2-m+2)}{m(m+1)} \, \phi \, (p) \, \Big] \,
\end{align*}
and the lemma is proven.
\end{proof}
%
%
%

Having at hand the above lemmata, we are now in the position to describe the local structure around the singular points of the {\em multi-poles fundamental solutions} $\GGG_{\aaa,\bbb,\ccc}$ constructed in Remark~\ref{gabc} through Proposition~\ref{balancrough}. For $m\geq 3$, it is sufficient to apply the operator $\Lg$ to the expression
\begin{eqnarray*}
\GGG_{\aaa,\bbb,\ccc} & + & \sum_{l=1}^K  \, \frac{a_l}{4(m-2)} \,\, \bigg[   \, \frac{|\Gamma_{N+l}|}{ 2 (m-1) |\Sp^{2m-1}|} \,\,  G_{\Delta\Delta}(q_l,\cdot)  \, \bigg]  \\
& + & \sum_{j=1}^N  \, \bigg(     \frac{c_j}{4(m-2)}  \, - \, \frac{s_\omega \, (m^2-m+2) \, b_j}{(m-2)m(m+1)}  \,   \bigg) \,\, \bigg[   \, \frac{|\Gamma_{j}|}{ 2 (m-1) |\Sp^{2m-1}|} \,\,  G_{\Delta\Delta}(p_j,\cdot)  \, \bigg]  \\
& - & \sum_{j=1}^N  \, b_j \,\, \bigg[   \, \frac{|\Gamma_{j}|}{ 2 (m-1) |\Sp^{2m-1}|} \,\, G_{\Delta}(p_j,\cdot)  \, \bigg]  \, ,
\end{eqnarray*}
to get a function in $C^{0,\alpha}(M)$. For $m=2$, one can obtain the same conclusion, applying the operator $\Lg$ to the expression 
\begin{eqnarray*}
\GGG_{\aaa,\bbb,\ccc} & - & \sum_{l=1}^K  \, \frac{a_l}{4} \,\, \bigg[   \, \frac{|\Gamma_{N+l}|}{ |\Sp^{3}|} \,\, G_{\Delta\Delta}(q_l,\cdot)  \, \bigg]  \\
& - & \sum_{j=1}^N  \, \bigg(     \frac{c_j}{4}  \, - \, \frac{s_\omega  \, b_j}{6}  \,   \bigg) \,\, \bigg[   \, \frac{|\Gamma_{j}|}{ |\Sp^{3}|} \,\, G_{\Delta\Delta}(p_j,\cdot)  \, \bigg]  \\
& - & \sum_{j=1}^N  \, b_j \,\, \bigg[   \, \frac{|\Gamma_{j}|}{ 2|\Sp^{3}|} \,\,  G_{\Delta}(p_j,\cdot) \, \bigg]  \, .
\end{eqnarray*}
Combining the previous observations with the standard elliptic regularity theory, we obtain the following proposition.
\begin{prop}
\label{loc_structure}
Let $(M,g,\omega)$ be a compact Kcsc orbifold of complex dimension $m \geq 2$ and let ${\rm Ker}{(\Lg)} = {span}\{\varphi_0, \varphi_1, \dots,
\varphi_d\}$, as in~\eqref{nontrivial_ker}. Let then $\{p_1, \dots, p_N\}$ and $\{q_1, \dots, q_K\}$ be two disjoint sets of points in $M$ and let
$\GGG_{\aaa,\bbb,\ccc}$ be as in Remark~\ref{gabc}. Then, we have that
$$
\GGG_{\aaa,\bbb,\ccc}  \,\,\, \in \,\,\,   \mathcal{C}^{\infty}_{loc}(M_{\p,\q}) \, .
$$
Moreover, if $z^1, \ldots , z^m$ are local coordinates centered at the singular points, then the following holds.
\begin{itemize}
\item If $m\geq 3$, then $\GGG_{\aaa,\bbb,\ccc}$  blows up like $|z|^{2-2m}$ at the points points of $p_1, \ldots, p_N$ and like $|z|^{4-2m}$ at the points $q_1, \ldots, q_K$.
\item If $m=2$, then $\GGG_{\aaa,\bbb,\ccc}$  blows up like $|z|^{-2}$ at the points $p_1, \ldots, p_N$ and like $\log\st |z|\dt$ at the points $q_1, \ldots, q_K$.
\end{itemize}
\end{prop}

%

%




%

\subsection{Resolution of the linearized scalar curvature equation.}

In this subsection, we are going to describe the possible choices for a right inverse of the operator $\Lg$, in a suitable functional setting. Since this operator is formally selfadjoint and since we are assuming that its kernel is nontrivial, we expect the presence of a nontrivial cokernel. To overcome this difficulty, we are going to consider some appropriate finite dimensional extensions of the natural domain of $\Lg$, which, according to Theorem~\ref{invertibilitapesatobase}, is given by $C^{4,\alpha}_\delta(M_{\p,\q})$, with $\delta \in (4-2m,0)$ if $m\geq 3$ and $\delta \in (0,1)$ if $m=2$. Building on the analysis of the previous section, we are going to introduce the following {\em deciciency spaces}. Given a triple of vectors $\boldsymbol\alpha \in \RR^K$ and $\bbbb, \cccc \in \RR^N$, we set, for $m \geq 3$, $l=1,\ldots,K$ and $j=1, \ldots, N$,
\begin{eqnarray*}
W^l_{\aaaa} & = &
 -\frac{\alpha_l}{4(m-2)} \,\, \bigg[   \, \frac{|\Gamma_{N+l}|}{ 2 (m-1) |\Sp^{2m-1}|} \,\, G_{\Delta\Delta}(q_l,\cdot)  \, \bigg] \, , \\
W^j_{\bbbb,\cccc}  & = &   \beta_j \,\, \bigg[   \, \frac{|\Gamma_{j}|}{ 2 (m-1) |\Sp^{2m-1}|} \,\, G_{\Delta}(p_j,\cdot)  \, \bigg] \\
& & - \,  \bigg(     \frac{\gamma_j}{4(m-2)}  \, - \, \frac{s_\omega \, (m^2-m+2) \, \beta_j}{(m-2)m(m+1)}  \,  \bigg) \,\, \bigg[   \, \frac{|\Gamma_{j}|}{ 2 (m-1) |\Sp^{2m-1}|} \,\,  G_{\Delta\Delta}(p_j,\cdot)  \, \bigg]  \, ,
\end{eqnarray*}
whereas, for $m=2$, $l=1,\ldots,K$ and $j=1, \ldots, N$, we set
\begin{eqnarray*}
W^{l}_{\aaaa} &  = &  \, \alpha_{l} \,\, \bigg[   \, \frac{|\Gamma_{N+l}|}{ 4|\Sp^{3}|} \,\, G_{\Delta\Delta}(q_l,\cdot)  \, \bigg]  \, , \\
W^j_{\bbbb,\cccc}& = &  \beta_j \,\, \bigg[   \, \frac{|\Gamma_{j}|}{ |\Sp^{3}|} \,\,  G_{\Delta}(p_j,\cdot) \, \bigg] \, + \, \bigg(     \frac{\gamma_j}{4}  \, - \, \frac{s_\omega  \, \beta_j}{6}  \,   \bigg) \,\, \bigg[   \, \frac{|\Gamma_{j}|}{ |\Sp^{3}|} \,\, G_{\Delta\Delta}(p_j,\cdot)  \, \bigg]  \, .
\end{eqnarray*}
We are now in the position to define the {\em deficiency spaces}
\begin{eqnarray*}
\mathcal{D}_{\q}(\aaaa) \,\, = \,\, \mbox{\em span} \,\Big\{ \, W^l_{\aaaa}\,  : \, {l=1,\ldots , K}  \, \Big\}  \quad & \hbox{and} & \quad
\mathcal{D}_{\p}(\bbbb, \cccc)  \,\,= \,\,  \mbox{\em span}\, \Big\{ \,
W^j_{\bbbb,\cccc} \,  :  \, {j=1, \ldots , N} \, \Big\} \, .
\end{eqnarray*}
These are finite dimensional vector spaces and they can be endowed with the following norm. If $V = \sum_{l=1}^K V^l \, W_{\aaaa}^l \in \mathcal{D}_\q(\aaaa)$ and $U = \sum_{j=1}^N U^j W_{\bbbb,\cccc}^j \in \mathcal{D}_\p(\bbbb, \cccc)$, we set
$$
\left\| V \right\|_{\mathcal{D}_\q(\aaaa)} \,\, = \,\, \sum_{l=1}^K\,  | V^l | \qquad \hbox{and} \qquad \left\| U \right\|_{\mathcal{D}_\p(\bbbb,\cccc)} \,\, = \,\, \sum_{j=1}^N\,  | U^j |\, .
$$
We will also make use of the shorthand notation $\mathcal{D}_{\p,\q}(\aaaa,\bbbb,\cccc)$ to indicate the direct sum $\mathcal{D}_{\q}(\aaaa) \oplus \mathcal{D}_{\p}(\bbbb,\cccc)$ of the {\em deficiency spaces} introduced above, endowed with the obvious norm $\left\| \, \cdot \, \right\|_{\mathcal{D}_\q(\aaaa)} +  \left\| \, \cdot \, \right\|_{\mathcal{D}_\p(\bbbb,\cccc)}$.

To treat the case $m=2$, it is convenient to introduce further finite dimensional extensions of the domain $C^{4,\alpha}_\delta(M_{\p,\q})$, with $\delta \in (0,1)$. These will be called {\em extra deficiency spaces} and they are defined as
\begin{eqnarray*}
\mathcal{E}_{\q} \,\, = \,\, \mbox{\em span} \,\big\{ \, \chi_{q_l}\,  : \, {l=1,\ldots , K}  \, \big\}  \quad & \hbox{and} & \quad
\mathcal{E}_{\p}  \,\, = \,\, \mbox{\em span}\, \big\{ \,
\chi_{p_j} \,  :  \, {j=1, \ldots , N} \, \big\} \, ,
\end{eqnarray*}
where the functions $\chi_{p_1}, \ldots, \chi_{p_N},\chi_{q_1}, \ldots, \chi_{q_K}$ are smooth cutoff functions supported on small balls centered at the points $p_{1}, \ldots, p_N, q_1, \ldots, q_K$ and identically equal to $1$ in a neighborhood of these points. Given two functions $X = \sum_{j=1}^N X^j \chi_{p_j}\in \mathcal{E}_{\p}$ and $Y= \sum_{l=1}^K Y^l \chi_{q_l} \in \mathcal{E}_{\q}$, we set
$$
\left\| Y \right\|_{\mathcal{E}_\q} \,\, = \,\, \sum_{l=1}^K\, | Y^l | \qquad \hbox{and} \qquad \left\| X \right\|_{\mathcal{E}_\p} \,\, = \,\, \sum_{j=1}^N\,  | X^j |\, .
$$
We will also make use of the shorthand notation $\mathcal{E}_{\p,\q}$ to indicate the direct sum $\mathcal{E}_{\q} \oplus \mathcal{E}_{\p}$ of the {\em extra deficiency spaces} introduced above, endowed with the obvious norm $\left\| \, \cdot \, \right\|_{\mathcal{E}_\q} +  \left\| \, \cdot \, \right\|_{\mathcal{E}_\p}$. Notice that, with these notation, the estimate~\eqref{eq:stimapesatabase2} in Theorem~\ref{invertibilitapesatobase} reads
$$
|| \, \widetilde{u} \,||_{C^{4,\alpha}_\delta(M_{\p,\q}) }  \, + \, ||  \stackrel{\circ}{u} ||_{\mathcal{E}_{\p,\q}} \,\, \leq \,\,  C \, || \, f \, ||_{C^{0,\alpha}_{\delta-4}} \, ,
$$
where $u \, = \, \widetilde{u} \,+ \!\stackrel{\circ}{u} \,\, \in{{C}^{4,\alpha}_{\delta}(M_{\p ,\q})      \, \oplus \,     \mathcal{E}_{\p,\q} }$ and $f \in {C}^{0,\alpha}_{\delta-4}(M_{\p ,\q})$ are functions satisfying the equation $\Lg [u] = f$ as well as the orthogonality condition~\eqref{eq:orto} and $\delta \in (0,1)$.

{
\begin{remark} We notice {\em en passant} that a function $\GGG_{\aaa,\bbb,\ccc}$ constructed as in Remark~\ref{gabc} behaves like $W_\aaa^l$ near the point $q_l$, for $l=1, \ldots, K$ and like $W_{\bbb,\ccc}^j$, near the point $p_j$, for $j=1,\ldots, N$. In fact, it satisfies
\begin{equation*}
\Lg \Big[\, \GGG_{\aaa,\bbb,\ccc}  \, - \, \sum_{l=1}^K W_\aaa^l     \, - \, \sum_{j=1}^N W_{\bbb,\ccc}^j  \,\Big] \,\,\, \in \,\,\, C^{0,\alpha}(M) \, .
\end{equation*}
\end{remark}
}

We recall that we have assumed that the bounded kernel of $\Lg$ is $(d+1)$-dimensional and that it is spanned by $\{\varphi_0, \varphi_1, \ldots , \varphi_d \}$, where $\varphi_0 \equiv 1$ and $\varphi_1, \ldots, \varphi_d$, with $d\geq 1$, is a collection of mutually $L^2(M)$-orthogonal smooth functions with zero mean and $L^2(M)$-norm equal to $1$. Given a triple of vectors $\boldsymbol\alpha \in \RR^K$ and $\bbbb, \cccc \in \RR^N$, it is convenient to introduce the following matrices
\begin{align}
\Xi_{il}(\boldsymbol{\alpha} ) \,\, := & \,\,\,\, \alpha_{l} \, \varphi_{i}(q_{l}) \,,   & \hbox{for} \quad  i=1 \dots, d \quad \hbox{and} \quad  l=1, \dots, K \, ,
\label{eq:nondeggen2}\\
\Theta_{ij}(\boldsymbol{\beta},\boldsymbol{\gamma} ) \,:=&\,\,\,  \beta_{j} \, \Delta\varphi_{i}(p_{j})  \, + \,
\gamma_{j} \, \varphi_{i}(p_{j}) \, , &  \,\hbox{for} \quad i=1 \dots, d \quad \hbox{and} \quad  j=1, \dots, N \, .
\label{eq:nondeggen}
\end{align}
These will help us in formulating our {\em nondegeneracy assumption}. We are now in the position to state the main results of our linear analysis on the base obifold.

\begin{teo}
\label{invertibilitapesatodef}
Let $(M,g,\omega)$ be a compact Kcsc orbifold of complex dimension $m \geq 2$ and let ${\rm Ker}{(\Lg)} = {span}\{\varphi_0, \varphi_1, \dots, \varphi_d\}$. Assume that the following {\em nondegeneracy condition} is satisfied: a triple of vectors
$\boldsymbol\alpha \in \RR^K$ and $\bbbb, \cccc \in \RR^N$ is given such that the $d \times (N+K)$ matrix
\begin{equation*}
\st\left.  \st\Xi_{il}(\aaaa)\dt_{\substack{ 1 \leq i\leq d\\ 1 \leq l \leq K }}   \,\, \right|\,\, \st\Theta_{ij}(\bbbb,\cccc)\dt_{\substack{ 1 \leq i\leq d\\ 1 \leq j \leq N }}  \dt
\end{equation*}
has full rank. Then, the following holds.
\begin{itemize}
\item If $m \geq 3$, then for every $f\in{C}^{0,\alpha}_{\delta-4}(M_{\p ,\q})$ with $\delta\in (4-2m,0)$, there exist real number $\nu$ and a function
$$
u\, = \, \widetilde{u} \, + \, \widehat{u}\, \, \in \, {C}^{4,\alpha}_{\delta}(M_{\p ,\q}) \, \oplus \, \mathcal{D}_{\p,\q}(\aaaa,\bbbb, \cccc)
$$
such that
\begin{equation}
\label{eq:linpro}
\Lg u \, + \, \nu \,\,  = \,\,  f\, , \qquad \hbox{in} \quad {M_{\p,\q}} \,.
\end{equation}
Moreover, there exists a positive constant $ C = C(\aaaa,\bbbb,\cccc,\delta )>0$ such that
\begin{equation}
|\,\nu\,| \,\, + \,\, ||\,\widetilde{u} \,||_{{C}^{4,\alpha}_{\delta}(M_{\p ,\q})  }
\, + \, ||\,\widehat{u} \,||_{  \mathcal{D}_{\p,\q}(\aaaa,\bbbb, \cccc) }
 \,\,\, \leq \,\,\, C \, || \, f \,   ||_{ \mathcal{C}^{0,\alpha}_{\delta-4}(M_{\p ,\q}) } \, .
\end{equation}
\item
If $m =2$, then for every $f\in {C}^{0,\alpha}_{\delta-4}(M_{\p ,\q})$ with $\delta\in (0,1)$, there exist real number $\nu$ and a function
$$
u \, = \, \widetilde{u} \,+ \,
\stackrel{\circ}{u}
 \,+ \, \widehat{u}\, \,\in \,\, {C}^{4,\alpha}_{\delta}(M_{\p ,\q}) \, \oplus \mathcal{E}_{\p,\q} \, \oplus \, \mathcal{D}_{\p,\q}(\aaaa,\bbbb, \cccc)
$$
such that
\begin{equation}
\label{eq:linpro2}
\Lg u \, + \, \nu \,\,  = \,\,  f\, , \qquad \hbox{in} \quad {M_{\p,\q}} \,.
\end{equation}
Moreover, there exists a positive constant $ C = C(\aaaa,\bbbb,\cccc,\delta )>0$ such that
\begin{equation}
|\,\nu\,| \, \, + \,\, ||\,\widetilde{u} \,||_{ {C}^{4,\alpha}_{\delta}(M_{\p ,\q}) }  \,+ \,
||\stackrel{\circ}{u} ||_{ \mathcal{E}_{\p,\q} }   \, + \,
||\,\widehat{u} \,||_{ \mathcal{D}_{\p,\q}(\aaaa,\bbbb, \cccc) }
\,\,\, \leq \,\,\, C \, || \, f \,   ||_{ {C}^{0,\alpha}_{\delta-4}(M_{\p ,\q}) }
\end{equation}
\end{itemize}
\end{teo}
\begin{proof} We only prove the statement in the case $m\geq 3$, since it is completely analogous in the other case. For sake of simplicity we assume $\aaaa= \mathbf{0} \in \RR^K$, so that the nondegeneracy condition becomes equivalent to the requirement that the matrix
$$
\st\Theta_{ij}(\bbbb,\cccc)\dt_{\substack{ 1 \leq i\leq d\\ 1 \leq j \leq N }}
$$
has full rank. Under these assumptions, the {\em deficiency space} $\mathcal{D}_{\p,\q}(\aaaa,\bbbb,\cccc)$ reduces to $\mathcal{D}_\p(\bbbb,\cccc)$. In order to split our problem, it is convenient to set
\begin{equation}
f^{\perp} \,\, = \,\, f \, - \, \frac{1}{\vol\st M \dt}\int_{M} f \,d\mu_{\omega} \, - \, \sum_{i=1}^{d}\varphi_{i}\int_{M} f\varphi_{i}\,d\mu_{\omega}\,,
\end{equation}
so that $f^\perp$ satisfies the orthogonality conditions~\eqref{eq:orto}. By Theorem~\ref{invertibilitapesatobase}, we obtain the existence of a function $u^\perp \in C^{4,\alpha}_\delta(M_{\p,\q})$, which satisfies the equation
\begin{equation*}
\Lg \, [ \, u^\perp ] \,\, = \,\, f^\perp \, ,
\end{equation*}
together with the orthogonality conditions~\eqref{eq:orto} and the desired estimate~\eqref{eq:stimapesatabase}. To complete the resolution of equation~\eqref{eq:linpro}, we set
$$
f_0 \,\, = \,\, \frac{1}{\vol\st M \dt}\int_{M}\!\! f \,d\mu_{\omega} \qquad\quad \hbox{and} \quad\qquad f_i \,\, = \,\, \int_{M}\!\! f\varphi_{i}\,d\mu_{\omega} \, , \qquad \hbox{for $i=1,\ldots, d$} \, .
$$
Recalling the definition of $\Theta_{ij}(\bbbb,\cccc)$ and using the {\em nondegeneracy condition}, we select a solution $(\nu, U_1, \ldots, U_N) \in \RR^{N+1}$ to the following system of {\em linear balancing conditions}
\begin{eqnarray}
f_i \,\,  +\,\,  \sum_{j=1}^N \, U^j \, \big[ \, \beta_j \, (\Delta\varphi_i)(p_j) \, + \,  \gamma_j \, \varphi_i(p_j)  \, \big]  &  = &  0 \,, \quad\quad\qquad\qquad\qquad \hbox{$i = 1, \dots, d$}  , \\
f_0 \, {\rm Vol}_\omega(M)   \,\,  + \,\, \sum_{j=1}^N \, U^j \,\gamma_j  & = & \nu  \, {\rm Vol}_\omega(M) \, .
\end{eqnarray}
It is worth pointing out that in general this choice is not unique, since it depends in the choice of a right inverse for the matrix $\Theta_{ij}(\bbbb,\cccc)$. Theorem~\ref{balancrough} implies then the existence of a distribution $U \in \mathscr{D}'(M)$ which satisfies
\begin{equation*}
\Lg   [{U}]  \,  + \, \nu  \,\,\,  = \,\,\,   \sum_{i=0}^d f_i \, \varphi_i \, + \, \sum_{j=1}^N \, U^j \beta_j \, \Delta\delta_{p_{j}} \, + \, \sum_{j=1}^N U^j  \gamma_j \, \delta_{p_{j}} \, , \qquad \hbox{in \,\,\,$M$}\, .
\end{equation*}
Arguing as in Proposition~\ref{loc_structure}, it is not hard to show that $U \in C^{\infty}_{loc}(M_\p)$. In particular the function $u^\perp + U \in C^{4,\alpha}_{loc}(M_\p)$ satisfies the equation
\begin{equation*}
\Lg   [u^\perp + U]  \,  + \, \nu  \,\,\,  = \,\,\, f  \, , \qquad \hbox{in \,\,\,$M_{\p}$}\, .
\end{equation*}
To complete the proof of our statement, we need to describe the local structure of $U$ in more details. First, we observe that, by the very definition of the deficiency spaces, one has
$$
\Lg \, \big[  W^j_{\bbbb,\cccc} \big] \,\, = \,\, \beta_j \, \Delta \delta_{p_j} \, + \, \gamma_j \, \delta_{p_j} \, + \, V^j_{\bbbb,\cccc} \, ,
$$
where, for every $j=1, \ldots, N$, the function $V^j_{\bbbb,\cccc}$ is in $C^{\infty}(M)$. Combining this fact with the linear balancing conditions, we deduce that
\begin{eqnarray*}
\Lg \, \Big[ \, U \, - \, \sum_{j=1}^N \, U^j  \, W^j_{\bbbb,\cccc}\Big] & = &  f_0 \, - \,  \nu  \, + \, \sum_{i=1}^d f_i \, \phi_i \, - \, \sum_{j=1}^N \, U^j \, V_{\bbbb,\cccc}^j \\
& = &   \frac{1}{\vol(M)} \, \sum_{j=1}^N \, U^j \, \gamma_j \, - \, \sum_{i=1}^d \, \sum_{j=1}^N  \, U^j \, \Theta_{ij}(\bbbb,\cccc) \, \phi_i \, - \, \sum_{j=1}^N \, U^j \, V_{\bbbb,\cccc}^j  \, .
\end{eqnarray*}
By the definition of $V^j_{\bbbb,\cccc}$ it follows that
$$
\int_M V^j_{\bbbb,\cccc} \, \phi_0 \,\, d\mu_\omega \,\, = \,\, -\, \gamma_j \qquad \quad \hbox{and} \qquad \quad \int_M V^j_{\bbbb,\cccc} \, \phi_i\,\, d\mu_\omega \,\, = \,\, - \, \Theta_{ij}(\bbbb,\cccc)
$$
and thus, it is easy to check the right hand side of the equation above is orthogonal to $\ker(\Lg)$. Hence, using Theorem~\ref{invertibilitapesatobase} and by the elliptic regularity, we deduce the existence of a smooth function $\overline u \in C^{\infty}(M)$ which satisfies
$$
\Lg \, [\, \overline{u} \,] \,\,\, = \,\,\, \frac{1}{\vol(M)} \, \sum_{j=1}^N \, U^j \, \gamma_j \, - \, \sum_{i=1}^d \, \sum_{j=1}^N  \, U^j \, \Theta_{ij}(\bbbb,\cccc) \, \phi_i \, - \, \sum_{j=1}^N \, U^j \, V_{\bbbb,\cccc}^j \,, \qquad \hbox{in $M$.}
$$
Setting $\widehat{u} = \sum_{j=1}^N \, U^j  \, W^j_{\bbbb,\cccc}$, we have obtained that $\Lg\, [\,U\,] \, = \, \Lg \, [\,\widehat{u} \,+ \, \overline{u}\, ]$, hence
\begin{equation*}
\Lg   [\, u^\perp \! + \,\overline{u}  + \, \widehat{u}\,]  \, + \, \nu  \,\,\,  = \,\,\, f  \, , \qquad \hbox{in \,\,\,$M_{\p}$}\, ,
\end{equation*}
with $\widetilde{u} = (u^\perp +\, \overline{u}) \, \in \, C^{4,\alpha}_\delta(M_\p)$ and $\widehat{u} \in \mathcal{D}_\p(\bbbb,\cccc)$. Moreover, combining the estimate~\eqref{eq:stimapesatabase} with our construction, it is clear that, for suitable positive constants $C_0,\ldots, C_3$, possibly depending on $\bbbb,\cccc$ and $\delta$, it holds
\begin{eqnarray*}
|| \,u \, ||_{C^{4,\alpha}_{\delta}(M_\p) \oplus \mathcal{D}_\p(\bbbb,\cccc)} \!\!\!\! &= & \!\!\!\! || \,\widetilde{u} \, ||_{C^{4,\alpha}_{\delta}(M_\p)}    + ||\,\widehat{u} \, ||_{\mathcal{D}_\p(\bbbb,\cccc)}  \, \leq \,   || \, u^\perp ||_{C^{4,\alpha}_{\delta}(M_\p) }   +  || \, \overline{u} \,||_{C^{4,\alpha}_{\delta}(M_\p)}  +  || \, \widehat{u} \, ||_{ \mathcal{D}_\p(\bbbb,\cccc)} \phantom{\sum_j=1^N}\\
& \leq & \!\!\!\! C_0 \, || \, f^\perp||_{C^{0,\alpha}_{\delta-4}(M_\p)}  +  C_1  \, \sum_{j=1}^N \, |U^j| \,  \leq \, C_2 \, \Big( \, ||\, f^\perp \, ||_{C^{0,\alpha}_{\delta-4}(M_\p)}  \, + \,    \, \sum_{i=1}^d \, |f_i| \, \,  \Big) \\
& \leq & \!\!\!\!  C_3\, || \, f\,||_{C^{0,\alpha}_{\delta-4}(M_\p)}\,, \phantom{\bigg(\sum_j=1^N\bigg)}
\end{eqnarray*}
which is the desired estimate. Finally, we observe that the constant $\nu$ as well can be easily estimated in terms of the norm of $f$. This concludes the proof of the theorem.
\end{proof}
\begin{remark}\label{inversadef}
In other words, with the notations introduced in the proof of the previous theorem, we have proven that, for $m\geq 3$ and $\delta\in(4-2m,0)$, the operator
\begin{eqnarray*}
\mathbb{L}^{(\delta)}_{\aaaa,\bbbb,\cccc} \,\, : \, \, {C}^{4,\alpha}_{\delta}(M_{\p ,\q}) \, \oplus \, \mathcal{D}_{\p,\q}(\aaaa,\bbbb, \cccc)   \, \times \, \RR \!\!& \longrightarrow &\!\! {C}^{0,\alpha}_{\delta-4}(M_{\p ,\q}) \\
(\, \widetilde{u}\, +\, \widehat{u} \,\,,\,\, \nu\, ) \!\!& \longmapsto &\!\! \Lg \, [\,\widetilde{u}\, +\, \widehat{u} \,] \, + \, \nu \, ,
\end{eqnarray*}
with $\bbbb,\cccc$ and $\aaaa$ satisfying the {\em nondegeneracy condition}, admits a (in general not unique) bounded right inverse
\begin{eqnarray*}
\mathbb{J}^{(\delta)}_{\aaaa,\bbbb,\cccc} \,\, : \, \, {C}^{0,\alpha}_{\delta-4}(M_{\p ,\q})  \!\!& \longrightarrow &\!\!
{C}^{4,\alpha}_{\delta}(M_{\p ,\q}) \, \oplus \, \mathcal{D}_{\p,\q}(\aaaa,\bbbb, \cccc)   \, \times \, \RR \, ,
\end{eqnarray*}
so that $ \big( \,\mathbb{L}^{(\delta)}_{\aaaa,\bbbb,\cccc} \circ  \mathbb{J}^{(\delta)}_{\aaaa,\bbbb,\cccc}\, \big) \, [\, f \,] \, = \, f$, for every $f \in {C}^{0,\alpha}_{\delta-4}(M_{\p ,\q})$ and
$$
\big\|\,      \mathbb{J}^{(\delta)}_{\aaaa,\bbbb,\cccc}  \, [\, f \,]   \,  \big\|_{{C}^{4,\alpha}_{\delta}(M_{\p ,\q}) \, \oplus \, \mathcal{D}_{\p,\q}(\aaaa,\bbbb, \cccc)   \, \times \, \RR } \,\,\, \leq \,\,\, C \,\, || \, f \, ||_{{C}^{0,\alpha}_{\delta-4}(M_{\p ,\q}) } \, .
$$
Of course, the analogous conclusion holds in the case $m=2$.
\end{remark}

\section{On scalar flat ALE spaces}\label{lineareALE} 
We now reproduce an analysis similar to the one just completed on the base base manifold on our 
model ALE spaces. We define in this setting too weighted H\"older spaces. Since we will use duality arguments we introduce   also  weighted Sobolev spaces. Let $\st X_{\Gamma},h,\eta \dt$ be an $ALE$ K\"ahler space and set 
\begin{equation}X_{\Gamma,R_{0}}=\pi^{-1}\st B_{R_{0}} \dt\,.\end{equation}
This can be thought as the counterpart in $X_{\Gamma}$ of  $M_{r_{0}}$ in $M$. Let $\delta\in \RR$, $\alpha\in (0,1)$, the weighted H\"older space $C_{\delta}^{k,\alpha}\st X_{\Gamma} \dt$ is the set of functions $f\in C_{loc}^{k,\alpha}(X_{\Gamma})$ s.t.
\begin{equation}
\left\|f\right\|_{C_{\delta}^{k,\alpha}\st X_{\Gamma} \dt}:=\left\|f\right\|_{C^{k,\alpha}\st X_{\Gamma,R_{0}} \dt}+\sup_{R\geq R_{0}}R^{-\delta}\left\|f\st R\cdot \dt\right\|_{C^{k,\alpha}\st B_{1}\setminus B_{1/2} \dt}<+\infty\,.
\end{equation}  
In order to define weighted Sobolev spaces we have to introduce a distance-like function $\gamma\in C_{loc}^{\infty}\st X_{\Gamma} \dt$ defined as 
\begin{equation}
\gamma\st p \dt:=\chi\st p \dt + \st 1-\chi\st p \dt \dt |x\st p \dt|\qquad p\in X
\end{equation}
with $\chi$ a smooth cutoff function identically $1$ on $X_{\Gamma,R_{0}}$ and identically $0$ on $X_{\Gamma}\setminus X_{\Gamma,2R_{0}}$.
Let moreover $\delta\in \RR$, the weighted Sobolev space $W_{\delta}^{k,2}\st X_{\Gamma} \dt$ is the set of functions $ f\in L_{loc}^{1}(X_{\Gamma})$ such that
\begin{equation}\left\|f\right\|_{W_{\delta}^{k,2}\st X_{\Gamma} \dt}:=\sqrt{\sum_{j=0}^{k}\int_{X}\left|\gamma^{-\delta+j}\nabla^{j}f\right|_{\eta}^{2}\,d\mu_{\eta}}<+\infty\,.
\end{equation}

\begin{remark}
A function $f\in C^{k\,\alpha}\st X_{\Gamma}\dt$ on the set $X\setminus X_{\Gamma,R_{0}}$ beheaves like
\begin{equation}
f|_{X_{\Gamma}\setminus X_{\Gamma,R_{0}}}\st p \dt=\mathcal{O}\st |x\st p \dt|^{\delta} \dt
\end{equation}
We also note that we have the inclusion
\begin{equation}
C_{\delta}^{k,\alpha}\st X_{\Gamma} \dt\subseteq W_{\delta'+m}^{k,2}\st X_{\Gamma} \dt\qquad \delta'>\delta\,.
\end{equation}
\end{remark}

We recall now the natural duality between weighted spaces 
\begin{equation}\left<\cdot|\cdot\right>_{\eta}\,:\, L_{\delta}^{2}\st X_{\Gamma} \dt\times L_{-\delta}^{2}\st X_{\Gamma} \dt\rightarrow \RR\end{equation} 
defined as
\begin{equation}\left<f|g\right>_{\eta}:=\int_{X}f\,g\,d\,\mu_{\eta}\label{eq: dualita}\end{equation}

The main task of this section is to solve the linearized constant scalar curvature equation
\begin{equation}
\Le u=f \label{eq: Leta}
\end{equation}
and we recall that by \eqref{eq:defLg}
\begin{equation}
\Le u=\Delta_{\eta}^{2}u+\left< \rho_{\eta} | i\dd u\right>
\end{equation}
and we notice that if $\st X_{\Gamma}, h,\eta \dt$ is Ricci-flat the operator $\Le$ reduces to the $\eta$-biharmonic operator. Since we want to study the operator $\Le$ on weighted spaces we have to be careful on the choice of weights. Indeed to have Fredholm properties we must avoid the indicial roots at infinity of $\Le$ that thanks to the decay of the metric coincide with those of euclidean Laplace operator $\Delta$ . We recall that the set of indicial roots at infinity for $\Delta$ on $\CC^{m}$ is $\ZZ\setminus \sg 5-2m, \ldots ,-1 \dg$ for $m\geq 3$ and $\ZZ$ for $m=2$.   

\noindent For $ALE$ K\"ahler spaces  a result analogous to Proposition \ref{invertibilitapesatobase} holds true.   

\begin{prop}\label{isomorfismopesati}
Let $(X_{\Gamma},h,\eta)$ a scalar flat $ALE$ K\"ahler space. Let $\delta\in \RR$ with
\begin{equation}
\delta\neq l+m, 4-m-l\qquad l\in \NN\,.
\end{equation}
then the operator
\begin{equation}
\Le:W_{\delta}^{4,2}\st X_{\Gamma} \dt\rightarrow L_{\delta-4}^{2}\st X_{\Gamma} \dt\,.
\end{equation}
is Fredholm and its cokernel is the kernel of its adjoint under duality \eqref{eq: dualita}
\begin{equation}
\Le:W_{-\delta}^{4,2}\st X_{\Gamma} \dt\rightarrow L_{-\delta-4}^{2}\st X_{\Gamma} \dt\,.
\end{equation}
There is, therefore, a continuous right inverse  
\begin{equation}
\mathbb{J}^{(\delta)}: \Le\sq W_{\delta}^{4,2}\st X_{\Gamma} \dt\dq\rightarrow W_{\delta}^{4,2}\st X_{\Gamma} \dt
\end{equation}
for operator
\begin{equation}
\Le:W_{\delta}^{4,2}\st X_{\Gamma} \dt\rightarrow L_{\delta-4}^{2}\st X_{\Gamma} \dt\,.
\end{equation}
If $m\geq 3$  and $\delta\in (4-2m,0)$, then 
\begin{equation}\Le : C_{\delta}^{4,\alpha}\st X_{\Gamma} \dt\longrightarrow C_{\delta-4}^{0,\alpha}\st X_{\Gamma} \dt\end{equation}
 is invertible. If $m=2$ and $\delta\in (0,1)$, then
\begin{equation}\Le : C_{\delta}^{4,\alpha}\st X_{\Gamma} \dt\longrightarrow C_{\delta-4}^{0,\alpha}\st X_{\Gamma} \dt\end{equation}
 is surjective with one dimensional kernel spanned by the constant function.
\end{prop}

\noindent The proof of the above result follows standard lines (see e.g. Theorem 10.2.1 and Proposition 11.1.1 in \cite{Pacard-notes}). 
We focus now on asymptotic expansions of various operators on $ALE$ spaces.

\begin{lemma}\label{espansioniALE}
Let $\st X_{\Gamma},h,\eta \dt$ be a Ricci flat $ALE$-K\"ahler space. Then on the coordinate chart at infinity we have the following expansions  
\begin{itemize}
\item for the inverse of the metric $\eta^{i\bar{\jmath}}$
\begin{equation}\eta^{i\bar{\jmath}}=2\sq \delta^{i\bar{\jmath}}-\frac{2 \cga \st m-1 \dt}{\left|x\right|^{2m}} \st \delta_{i\bar{\jmath}}-m\frac{\overline{x^{i}}x^{j}}{\left|x\right|^{2}}  \dt+\mathcal{O}\st \left|x\right|^{-2-2m} \dt \dq\,;\label{eq:inversaeta} \end{equation}

\item the unit normal vector to the sphere $\left|x\right|=\rho$ 
\begin{equation}\nu= \frac{1}{|x|}\st  x^{i}\frac{\partial}{\partial x^{i}}+\overline{x^{i}}\frac{\partial }{\partial \overline{x^{i}}}\dt    \sq 1+\frac{ \cga \st m-1 \dt^{2}}{|x|^{2m}} \dq+\mathcal{O}\st \left|x\right|^{-2-2m} \dt\,;\label{eq:espnormaleALE}\end{equation}

\item the laplacian $\Delta_{\eta}$
\begin{equation}\Delta_{\eta}=\Delta-\frac{2 \cga \st m-1 \dt}{\left|x\right|^{2m}}\Delta+\frac{8 \cga \st m-1 \dt m}{\left|x\right|^{2m+2}}\overline{x^{i}}x^{j}\partial_{j}{\partial}_{\bar{\imath}}+\mathcal{O}\st \left|x\right|^{-2-2m} \dt\,.\label{eq:esplapALE}\end{equation}
\end{itemize}
\end{lemma}

\noindent We conclude this section with an observation regarding fine mapping properties of

\begin{equation}
\Le: W_{\delta}^{4,2}\st X_{\Gamma} \dt\rightarrow L_{\delta-4}^{2}\st X_{\Gamma} \dt
\end{equation}
that will  be useful in the sequel.

\begin{prop}\label{GAP}
Let $\st X_{\Gamma},h,\eta \dt$ be a Ricci-flat non flat ALE K\"ahler manifold, $\delta\in\st  -m-2 , -m-1 \dt$, then we have the following characterization:
\begin{equation}\label{eq:caratt}
\Le\sq W_{\delta}^{4,2}\st X_{\Gamma} \dt\dq=\sg f\in L_{\delta}^{2}\st X_{\Gamma} \dt\,| \, \int_{X_{\Gamma}}f\,d\mu_{\eta}=0 \dg\,.
\end{equation}
Therefore the equation
\begin{equation}
\Le u= f
\end{equation} 
with $f\in L_{\delta-4}^{2}\st X_{\Gamma} \dt$ (respectively $f\in C_{\delta-4}^{0,\alpha}\st X_{\Gamma} \dt$)   and $\delta\in  \st -m-2,-m-1 \dt$ (respectively $\delta\in  \st 2-2m,4-2m \dt$) is solvable for $u\in W_{\delta}^{4,2}\st X_{\Gamma} \dt$ (respectively $u\in C_{\delta}^{4,\alpha}\st X_{\Gamma} \dt$) if and only if 
\begin{equation}
\int_{X_{\Gamma}}f\,d\mu_{\eta}=0\,.
\end{equation}
\end{prop}
\begin{proof}

Since $\Le$ is formally selfadjoint  we can identify, via duality \eqref{eq: dualita}, the cokernel of 
\begin{equation}
\Le: W_{\delta}^{4,2}\st X_{\Gamma} \dt\rightarrow L_{\delta}^{2}\st X_{\Gamma} \dt \qquad \delta\in \st -m-2,-m-1 \dt
\end{equation} 
with the kernel of 
\begin{equation}
\Le: W_{-\delta}^{4,2}\st X_{\Gamma} \dt\rightarrow L_{-\delta}^{2}\st X_{\Gamma} \dt\,.
\end{equation}
to get the characterization \eqref{eq:caratt} stated in the proposition we have to identify generators of this kernel. Let then $u\in W_{\delta'}^{4,2}\st X_{\Gamma} \dt$ such that
\begin{equation}
\Le u=0
\end{equation}

with $\delta'\in (m+1,m+2)$, by standard elliptic regularity we have that $u\in C_{loc}^{\omega}\st X_{\Gamma} \dt$. On $X_{\Gamma}\setminus  X_{\Gamma,R}$ we consider the Fourier expansion of $u$  
\begin{equation}u=\sum_{k=0}^{+\infty} u^{(k)}\st |x| \dt \phi_{k}\end{equation}     
with $u^{(k)}\in C_{\delta'-m}^{n,\alpha}\st [R,+\infty) \dt$ for any $n\in \NN$ and this sum is $C^{n,\alpha}$-convergent on compact sets. Then, using expansions\eqref{eq:inversaeta}, \eqref{eq:espnormaleALE},\eqref{eq:esplapALE}, we have on $X_{\Gamma}\setminus X_{\Gamma,R}$
\begin{align}
0=&\Delta_{\eta}^{2}u\\
=& \sum_{k=0}^{+\infty}\Delta^{2}\st u^{(k)}\st |x| \dt  \phi_{k}\dt +|x|^{-2m}L_{4}\st u \dt+|x|^{-1-2m}L_{3}\st u \dt+|x|^{-2-2m}L_{2}\st u \dt
\end{align}

\noindent and $L_{k}$ differential operators of order $k$ and uniformly bounded coefficients.  The equation 
\begin{equation}\sum_{k=0}^{+\infty} \Delta^{2}\st u^{(k)}\st |x| \dt \phi_{k} \dt =-|x|^{-2m}L_{4}\st u \dt-|x|^{-1-2m}L_{3}\st u \dt-|x|^{-2-2m}L_{2}\st u \dt\end{equation}

\noindent implies  
\begin{equation}\Delta^{2}\st u^{(k)} \phi_{k} \dt \in C_{\delta'-3m-4}^{n,\alpha}\st X_{\Gamma}\setminus X_{\Gamma,R} \dt\qquad  k\geq 0\,.\end{equation}
\noindent Suppose   
\begin{equation}\limsup_{|x|\rightarrow +\infty}|u|>0\,,\end{equation}
\noindent since $u^{(k)}\phi_{k}\in C_{\delta'-m}^{n,\alpha}\st X_{\Gamma}\setminus X_{\Gamma,R} \dt$ the only possibilities are
\begin{equation} u^{(0)}\st |x| \dt= c_{0}+\upsilon_{0}\st |x| \dt\end{equation}
\begin{equation} u^{(1)}\st |x| \dt= \st |x| +\upsilon_{1}\st |x| \dt\dt \phi_{1}\end{equation}
\noindent with $\upsilon_{0},\upsilon_{1}\in C_{\delta'-3m}^{n,\alpha}\st [R,+\infty) \dt$ and $c_{0}\in \RR$. But there aren't $\phi_{1}$ that are $\Gamma$-invariant (see Remark \ref{nolinear}), so the only possibility is that
\begin{equation}u^{(0)}\st |x| \dt= c_{0}+\upsilon_{0}\st |x| \dt\,.\end{equation}
\noindent We now show that $u$ is actually constant, indeed $u-c_{0}\in C_{\delta'-3m}^{n,\alpha}\st X \dt$ and
\begin{equation}\Le\st u-c_{0} \dt=\Delta_{\eta}^{2}\st u-c_{0} \dt=0\end{equation}
\noindent so by Proposition \ref{isomorfismopesati} we can conclude
\begin{equation}u-c_{0}\equiv 0\,.\end{equation}
The proposition now follows immediately.

\end{proof}

\begin{remark}

As we will see in the sequel, it will be necessary to solve this kind of equations  for constructing  refined families of metrics on model spaces $X_{\Gamma}$.
\end{remark}


\section{Nonlinear analysis}
\noindent In this section we collect all the critical estimates needed in the proof of Theorems \ref{belliebrutti} and \ref{maintheorem}. As in \cite{ap1} and \cite{ap2} we produce Kcsc metrics on manifolds wiht boundary  which we believe could be of independent interest (Propositions \ref{crucialbase}, \ref{crucialmodello}) We will give details only of the steps where the absence of the $4-2m$
 asymptotic in the model creates a genuine difference from the blow-up case.

Given $\varepsilon$ sufficiently small we look at the truncated manifolds $M_{\rep}$  and $X_{\Rep}$  where we impose the following relations:
\begin{equation}
\rep=\varepsilon^{\frac{2m-1}{2m+1}}=\varepsilon \Rep.
\end{equation}

As in \cite{ap1}, \cite{ap2} we introduce some functional spaces we will need in the sequel that will naturally work as ``space of parameters" for our construction:

\begin{equation}
\mathcal{B}_{j}:=C^{4,\alpha}\st \Sp^{2m-1}/\Gamma_{j} \dt\times C^{2,\alpha}\st \Sp^{2m-1}/\Gamma_{j} \dt
\end{equation}

\begin{equation}
\mathcal{B}:=\prod_{j=1}^{N}\mathcal{B}_{j}
\end{equation}

\begin{equation}
\mathfrak{B}\st\kappa, \sigma, \beta\dt:=\sg \st \hg,\kg \dt\in \mathcal{B}\,\left|\, \left\|h^{(0)}_{j},k^{(0)}_{j}\right\|_{\mathcal{B}_{j}}\leq \kappa \rep^{\beta},  \left\|h^{(\dagger)}_{j},k^{(\dagger)}_{j}\right\|_{\mathcal{B}_{j}}\leq \kappa \rep^{\sigma} \right. \dg
\end{equation}

\noindent Functions in $\mathcal{B}^{\alpha}$ will be used to parametrize solution of the Kcsc problem near a given
``skeleton" solution built by hand to match some of the first orders of the metrics coming on the two sides of the gluing. A key technical tool to implement such a 
strategy is given by using  outer (which will be transplanted on the base manifold) and inner (transplanted on the model) biharmonic extensions of functions on the unit sphere.

\begin{definition}
Let $(h,k)\in C^{4,\alpha}\st\Sp^{2m-1}\dt\times C^{4,\alpha}\st\Sp^{2m-1}\dt$ the outer biharmonic extension of $(h,k)$ is the function $H_{h,k}^{o}\in C^{4,\alpha}\st \CC^{m}\setminus B_{1} \dt$ solution fo the boundary value problem
\begin{equation}\begin{cases}
\Delta^{2} H_{h,k}^{o}=0 & \textrm{ on } \CC^{m}\setminus B_{1}\\
H_{h,k}^{o}=h &\textrm{ on } \partial B_{1}\\
\Delta H_{h,k}^{o}=k& \textrm{ on } \partial B_{1}
\end{cases}\end{equation}
Moreover $H_{h,k}^{o}$ has the following expansion in Fourier series for $m\geq 3$
\begin{equation}H^{o}_{h,k}:=\sum_{\gamma=0}^{+\infty}\st  \st  h^{(\gamma)}+\frac{k^{(\gamma)}}{4(m+\gamma-2)}  \dt |w|^{2-2m-\gamma}-\frac{k^{(\gamma)}}{4(m+\gamma-2)}|w|^{4-2m-\gamma} \dt \phi_\gamma\,,\label{eq:bihout3}\end{equation}
and for $m=2$
\begin{equation} H^{o}_{h,k}:=   h^{(0)}|w|^{-2}+\frac{k^{(0)}}{2}\log\st|w|\dt+  \sum_{\gamma=1}^{+\infty}\st  \st  h^{(\gamma)}+\frac{k^{(\gamma)}}{4\gamma}  \dt |w|^{-2-\gamma}-\frac{k^{(\gamma)}}{4\gamma}|w|^{-\gamma} \dt \phi_\gamma\,.\label{eq:bihout2}\end{equation}
\end{definition}

\begin{remark}
In the sequel we will take $\Gamma$-invariant $(h,k)\in C^{4,\alpha}\st\Sp^{2m-1}\dt\times C^{4,\alpha}\st\Sp^{2m-1}\dt$ and by the Remark \ref{nolinear} we will 
have no terms with $\phi_{1}$ in the formulas  \eqref{eq:bihout3} and \eqref{eq:bihout2} for nontrivial $\Gamma$.
\end{remark}

\begin{definition}
Let $\st \tilde{h}, \tilde{k}\dt\in C^{4,\alpha}\st \Sp^{2m-1} \dt\times C^{2,\alpha}\st \Sp^{2m-1} \dt$, the biharmonic extension $\hkjj$ on $B_1$ of $\st \tilde{h}, \tilde{k}\dt$ is the function $\hkjj\in C^{4,\alpha}\st \overline{B_{1}} \dt$ given by the solution of the boundary value problem
\begin{equation}\begin{cases}
\Delta^{2}H_{\tilde{h}\tilde{k}}^{I}=0 & w\in  B_{1} \\
H_{\tilde{h}\tilde{k}}^{I}=\tilde{h}& w\in \partial B_{1}\\
\Delta H_{\tilde{h}\tilde{k}}^{I}=\tilde{k}& w\in \partial B_{1}\\
\end{cases}\,.\end{equation}

\noindent The function $\hkjj$  has moreover the expansion

\begin{equation}\hkjj\st w\dt=\sum_{\gamma=0}^{+\infty}\left(\left(\tilde{h}^{(\gamma)}-\frac{\tilde{k}^{(\gamma)}}{4(m+\gamma)}  \dt |w|^\gamma+\frac{\tilde{k}^{(\gamma)}}{4(m+\gamma)}|w|^{\gamma+2} \dt \phi_\gamma\,.\end{equation}
\end{definition}

\begin{remark}
Again, if the group $\Gamma$ is non trivial, for $\Gamma$-invariant $(h,k)$ will be no $\phi_1$-term in the above summations and so we will have
\begin{equation}H_{h,k}^{I}=\left(\tilde{h}^{(0)}-\frac{\tilde{k}^{(0)}}{4m}  \dt  +\frac{\tilde{k}^{(0)}}{4m}|w|^{2}+  \sum_{\gamma=2}^{+\infty}\left(\left(\tilde{h}^{(\gamma)}-\frac{\tilde{k}^{(\gamma)}}{4(m+\gamma)}  \dt |w|^{\gamma}+\frac{\tilde{k}^{(\gamma)}}{4(m+\gamma)}|w|^{\gamma+2} \dt \phi_{\gamma}\,.\end{equation}
\end{remark}

\subsection{Kcsc metrics on the truncated base manifold:}\label{nonlinbase}

\noindent we want to find $F_{\ag,\bg,\cg,\hg,\kg}^{o}\in C^{4,\alpha}\st M_{\rep} \dt$ such that 

\begin{equation}\omega_{\ag,\bg,\cg,\hg,\kg}:=\omega +i\dd F_{\ag,\bg,\cg,\hg,\kg}^{o}\end{equation} 

\noindent is a metric on $M_{\rep}$ and 
\begin{equation}
\s_{\omega}\st F_{\ag,\bg,\cg,\hg,\kg}^{o} \dt=s_{\omega}+\mathfrak{s}_{\ag,\bg,\cg,\hg,\kg}\qquad \mathfrak{s}_{\ag,\bg,\cg,\hg,\kg} \in \RR
\end{equation}

\noindent The function $F_{\ag,\bg,\cg,\hg,\kg}^{o}$, as in \cite{ap2}, will be made of three blocks: a skeleton that keeps the metric near a Kcsc metric, modifications of biharmonic extensions of functions on the unit sphere that will allow us to glue the metric on the base with those on  the models and a small perturbation $f_{\ag,\bg,\cg,\hg,\kg}^{o}$ that ensures the constancy of the scalar curvature. The first observation is that the skeleton, in this new setting, has to take in account the different asymptotics of the model spaces and so it must contain a function ${\bf{G}}_{\ag,\bg,\cg}$ coming from Proposition \ref{balancrough}. If we apply the construction used in \cite{ap2} as it stands we will get exactly the same estimates on the term $F_{\ag,\bg,\cg,\hg,\kg}^{o}$ indeed we will obtain the following proposition.

\begin{prop}\label{famigliabaseap2}
Let $(M,g,\omega )$ a Kcsc orbifold with isolated singularities and $\st \hg,\kg \dt\in \mathcal{B}$  with
\begin{equation}
\left\|\hg,\kg\right\|_{\ambbda}\leq\kappa\rep^{4}
\end{equation}

\begin{enumerate}
\item
\label{sipuntibuoni}

If the set of singular points $\bf{q}$ with a scalar flat ALE K\"ahler resolution with $e\st \Gamma_{N+l} \dt\neq 0$ is nonempty and there exist $\ag\in \st\RR^{+}\dt^{K}$ such that, having set $\bar{\ag}:=\st \frac{a_{1}e\st \Gamma_{N+1} \dt}{|e\st \Gamma_{N+1} \dt|},\ldots, \frac{a_{K}e\st \Gamma_{N+K} \dt}{|e\st \Gamma_{N+K} \dt|} \dt$,

	\begin{displaymath}
	\left\{\begin{array}{lcl}
	\sum_{l=1}^{K}\bar{a}_{l}\varphi_{i}\st q_{l} \dt=0 && i=1,\ldots, d\\
	&&\\
	\st \bar{a}_{l} \varphi_{i}\st q_{l} \dt \dt_{\substack{1\leq i\leq d\\1\leq l\leq K}}&& \textrm{has full rank}
	\end{array}\right.
	\end{displaymath}
 then, for $\delta\in (4-2m,5-2m)$, for all $\cg\in \RR^{N}$  with $|\cg|\leq C\st \omega \dt|\ag|$ for some $C(\omega)>0$,  there exist $\tilde{\ag}\in\st\RR^{+}\dt^{K}$ and  $f_{\ag,\bg,\cg,\hg,\kg}^{o}\in C_{\delta}^{4,\alpha}\st M_{\p , \q} \dt\oplus\mathcal{D}_{\p,\q}\st \bar{\ag}, \bf{0},\bf{0} \dt$ 
 such that 
\begin{equation}\label{eq:errore}
\left\|f_{\ag,\bf{0},\cg,\hg,\kg}^{o}\right\|_{C_{\delta}^{4,\alpha}\st M_{\p , \q} \dt \oplus\mathcal{D}_{\p,\q}\st \bar{\ag}, \bf{0},\bf{0} \dt}\leq C\st  \kappa\dt \rep^{2m+1}
\end{equation}
and on $M_{\rep}$
\begin{equation}
\s_{\omega}\st \varepsilon^{2m-2}{\bf{G}}_{\tilde{\ag},\bf{0},\cg}+H_{\hg,\kg^{(\dagger)}}^{o}+ f_{\ag,\bf{0},\cg,\hg,\kg}^{o}\dt\equiv s_{\omega}+\mathfrak{s}_{\ag,\bf{0},\cg,\hg,\kg}.
\end{equation}
with $\bf{G}_{\tilde{\ag},\bf{0},\cg}$ constructed using Remark \ref{gabc}.
\item	
\label{solocattivi}
	If the set of singular points with a scalar flat non Ricci-flat ALE K\"ahler resolution is empty ($K=\emptyset$) and there exists $ \bg \in \st\mathbb{R}^{+}\dt^{N}$ and $\cg\in \RR^{N}$ such that
		\begin{displaymath}
	\left\{\begin{array}{lcl}
	\sum_{j=1}^{N}b_{j}\Delta_{\omega}\varphi_{i}\st p_{j} \dt+c_{j}\varphi_{i}\st p_{j} \dt=0 && i=1,\ldots, d\\
	&&\\
	\st b_{j}\Delta_{\omega}\varphi_{i}\st p_{j} \dt+c_{j}\varphi_{i}\st p_{j} \dt \dt_{\substack{1\leq i\leq d\\1\leq j\leq N}}&& \textrm{has full rank}
	\end{array}\right.
	\end{displaymath}
\noindent then, for $\delta\in (4-2m,5-2m)$, there exist $f_{\bg,\cg,\hg,\kg}^{o}\in C_{\delta}^{4,\alpha}\st M_{\p}\dt\oplus\mathcal{D}_{\p}\st \bg,\cg  \dt $ such that 
\begin{equation}\label{eq:errorebis}
\left\|f_{\bg,\cg,\hg,\kg}^{o}\right\|_{C_{\delta}^{4,\alpha}\st M_{\p } \dt\oplus \mathcal{D}_{\p}\st \bg,\cg  \dt}\leq C\st  \kappa\dt \rep^{2m+1}
\end{equation}
and on $M_{\rep}$
\begin{equation}
\s_{\omega}\st -\varepsilon^{2m}{\bf{G}}_{\bg,\cg}+H_{\hg,\kg^{(\dagger)}}^{o}+ f_{\bg,\cg,\hg,\kg}^{o}\dt\equiv s_{\omega}+\mathfrak{s}_{\bg,\cg,\hg,\kg}
\end{equation}
\noindent with ${\bf{G}}_{\bg,\cg}$ constructed with Remark \ref{gabc}. If $K=0$ it is assumed that the index $\ag$ does not appear in the above expressions.

\end{enumerate}

\end{prop}

\noindent As we will see in the last Section, the above estimate is enough to conclude the existence of Kcsc metrics on the glued manifold
only in the first case above (\ref{sipuntibuoni}). Unfortunately if $K=0$ the above dependence of $C$ on $\kappa$ forbids the use of the final
fixed point argument and a much more refined construction is needed.

\centerline{ {\em{For this reason, from now on we assume $K=0$ and $\ag$ does not appear.}}}
			The first attempt to improve the above result is to include in $F_{\bg,\cg,\hg,\kg}^{o}$, suitably adapted, potentials $\psi_{\eta_{j}}$'s.
Even if we do this, things won't change indeed we would get exactly the same result. We have to do a much more refined construction to really get what we need for gluing the metrics. To improve the estimates we need to modify $g$ by the combination of ${\bf{G}}_{\bg,\cg}$ and $\psi_{\eta_{j}}$ which corresponds to the Ricci-flat model we want to insert since, by Proposition \ref{balancrough} and Lemma \ref{Gbilapl}, on a neighborhood of $p_{j}$ 
\begin{equation}
{\bf{G}}_{\bg,\cg}\sim \frac{b_{j} |\Gamma_{j}|}{2\st m-1 \dt |\Sp^{2m-1}|}G_{\Delta}\st p_{j},z \dt.
\end{equation} 
It is clear that the form
\begin{equation}
\omega +i\dd \sq - \varepsilon^{2m}{\bf{G}}_{\bg,\cg}+\st \frac{b_{j}|\Gamma_{j}|}{ 2\cgaj  \st m-1 \dt|\Sp^{2m-1}|}  \dt^{\frac{1}{m}}\varepsilon^{2}\chi_{j}\psi_{\eta_{j}}\st  \st \frac{2 \cgaj  \st m-1 \dt|\Sp^{2m-1}|}{b_{j}|\Gamma_{j}|}  \dt^{\frac{1}{m}} \frac{z}{\varepsilon}  \dt   \dq
\end{equation}
matches exactly at the highest order the form  $\st \frac{b_{j}|\Gamma_{j}|}{2 \cgaj  \st m-1 \dt|\Sp^{2m-1}|}  \dt^{\frac{1}{m}}\eta_{j}$; once we rescale (as we will in the final gluing) the model using the map 
\begin{equation}
x= \st \frac{2 \cgaj  \st m-1 \dt|\Sp^{2m-1}|}{b_{j}|\Gamma_{j}|}  \dt^{\frac{1}{m}} \frac{z}{\varepsilon},
\end{equation} 
where $ \cgaj  $ is given by Proposition \ref{asintpsieta}. It is then convenient, from now on, to set the following notation
\begin{equation}\label{eq:Bj}
B_{j}=\st \frac{b_{j}|\Gamma_{j}|}{ 2\cgaj  \st m-1 \dt|\Sp^{2m-1}|}  \dt^{\frac{1}{2m}}.
\end{equation}
It will also be convenient to identify the right constants $C_{j}$ such that 
\begin{equation}
\Lg\st {\bf{G}}_{\bg,\cg}-\sum_{j=1}^{N} \cgaj  B_{j}^{2m}G_{\Delta}\st p_{j},z \dt+ C_{j}G_{\Delta\Delta}\st p_{j},z \dt \dt\in C^{0,\alpha}\st M \dt.
\end{equation}
By Lemma \ref{Glapl} one gets
\begin{equation}\label{eq:Cj}
C_{j}=\frac{|\Gamma_{j}|}{8\st m-2 \dt\st m-1 \dt}\sq  2\cgaj  B_{j}^{2m}\frac{\st m-1 \dt|\Sp^{2m-1}|}{m|\Gamma_{j}|}s_{\omega}\st 1+\frac{\st m-1 \dt^{2}}{\st m+1 \dt} \dt-c_{j} \dq.
\end{equation}

\noindent  The last ``ad hoc" term we will insert in the potential of the seeked metric is the one depending on $\st \hg,\kg \dt\in \mathfrak{B}\st \kappa,\beta,\sigma \dt$

\begin{equation}
\textbf{H}_{\hg,\kg}^{o}:=\sum_{j=1}^{N}\chi_{j}H_{h_{j}^{(\dagger)},k_{j}^{(\dagger)}}^{o}\st \frac{z}{\rep} \dt
\end{equation}

\noindent with $\chi_{j}$ smooth cutoff functions identically $1$ on $B_{r_{0}}\st p_{j} \dt$ and identically $0$ outside $B_{2r_{0}}\st p_{j} \dt$.
\newline
\phantom{hhhhh}
\newline
{\em For the rest of the subsection $\chi_{j}$ will denote a smooth cutoff function identically $1$ on $B_{r_{0}}\st p_{j} \dt$ and identically $0$ outside $B_{2r_{0}}\st p_{j} \dt$.}
\newline
\phantom{hhhhh}
\newline
 We can now state the main proposition for the base space, whose proof will fill the rest of this subsection:

\begin{prop}
\label{crucialbase}
Let $(M,g,\omega )$ a Kcsc orbifold with isolated singularities and let $\p$ be the set of singular points with non trivial orbifold group that admit a Ricci flat resolution. Let $\st \hg,\kg \dt\in \mathcal{B}^{\alpha}$ such that 
\begin{equation}
\left\|h^{(0)}_{j},k^{(0)}_{j}\right\|_{\mathcal{B}_{j}}\leq \kappa \rep^{\beta}=\varepsilon^{4m+2}\rep^{-6m+4-\delta}\,,
\end{equation}

\begin{equation}
\left\|h^{(\dagger)}_{j},k^{(\dagger)}_{j}\right\|_{\mathcal{B}_{j}}\leq \kappa \rep^{\sigma}=\varepsilon^{2m+4}\rep^{2-4m-\delta}\,.
\end{equation}

\noindent If   there exist $ \bg \in \st\mathbb{R}^{+}\dt^{N}$, $\cg\in \RR^{N}$ such that
	\begin{displaymath}
\left\{\begin{array}{lcl}
\sum_{j=1}^{N}b_{j}\Delta_{\omega}\varphi_{i}\st p_{j} \dt+c_{j}\varphi_{i}\st p_{j} \dt=0 && i=1,\ldots, d\\
&&\\
\st b_{j}\Delta_{\omega}\varphi_{i}\st p_{j} \dt+c_{j}\varphi_{i}\st p_{j} \dt \dt_{\substack{1\leq i\leq d\\1\leq j\leq N}}&& \textrm{has full rank}
\end{array}\right.
\end{displaymath} 
then for $\delta\in (4-2m,5-2m)$ there exist $f_{\bg,\cg,\hg,\kg}^{o}\in C_{\delta}^{4,\alpha}\st M_{\p} \dt\oplus\mathcal{D}_{\p}\st \bg,\cg \dt$ such that 
\begin{equation}
\left\|f_{\bg,\cg,\hg,\kg}^{o}\right\|_{C_{\delta}^{4,\alpha}\st M_{\p} \dt\oplus\mathcal{D}_{\p}\st \bg,\cg \dt}\leq C\st g \dt \varepsilon^{2m+2}\rep^{2-2m-\delta}=\rep^{6+\frac{6}{2m-1}-\delta}=\rep^{2m+1+ \alpha}\quad \alpha>0
\end{equation}
and on $M_{\rep}$
\begin{equation}
\s_{\omega}\st - \varepsilon^{2m}{\bf{G}}_{\bg,\cg}+\sum_{j=1}^{N}B_{j}^{2}\varepsilon^{2}\chi_{j}\psi_{\eta_{j}}\st \frac{z}{B_{j}	\varepsilon} \dt+\textbf{H}_{\hg,\kg}^{o}+ f_{\bg,\cg,\hg,\kg}^{o}\dt\equiv s_{\omega}+\mathfrak{s}_{\bg,\cg,\hg,\kg}\,,
\end{equation}
with ${\bf{G}}_{\bg,\cg}$ constructed with Remark \ref{gabc}.
\end{prop}


\noindent First of all we point out that the assumption that there exist $ \bg \in \st\mathbb{R}^{+}\dt^{N}$, $\cg\in \RR^{N}$ such that
\begin{equation}
\sum_{j=1}^{N}b_{j}\Delta_{\omega}\varphi_{i}\st p_{j} \dt+c_{j}\varphi_{i}\st p_{j} \dt=0  \qquad i=1,\ldots, d
\end{equation}
we can construct the function ${\bf{G}}_{\bg,\cg}$ using Remark \ref{gabc}. In order to find the correction $f_{\bg,\cg,\hg,\kg}^{o}$ we set up a fixed point problem that we will solve using Banach-Caccioppoli Theorem. First of all we want to find a PDE that  $f_{\bg,\cg,\hg,\kg}^{o}$ has to satisfy. We are looking for small perturbations of $\omega $ and this allows us to use on $M_{\rep}$ the expansion \eqref{eq:espsg} of operator $\s_{\omega}$ that is stated in Proposition \ref{espsg}. 
 
\begin{align}
s_{\omega}+\mathfrak{s}_{\bg,\cg,\hg,\kg}=&\s_{\omega}\st -\varepsilon^{2m}{\bf{G}}_{\bg,\cg}+\sum_{j=1}^{N}B_{j}^{2}\varepsilon^{2}\chi_{j}\psi_{\eta_{j}}\st \frac{z}{B_{j}\varepsilon} \dt+ \textbf{H}_{\hg,\kg}^{o}+ f_{\bg,\cg,\hg,\kg}^{o}\dt\\
=&s_{\omega}-\frac{1}{2} \varepsilon^{2m}\nu_{\bg,\cg}-\frac{1}{2}\Lg\st\sum_{j=1}^{N}B_{j}^{2}\varepsilon^{2}\chi_{j}\psi_{\eta_{j}}\st \frac{z}{B_{j}\varepsilon} \dt\dt -\frac{1}{2}\Lg \textbf{H}_{\hg,\kg}^{o} -\frac{1}{2}\Lg  f_{\bg,\cg,\hg,\kg}^{o}\\
&+\NN _{\omega}\st -\varepsilon^{2m}{\bf{G}}_{\bg,\cg}+\sum_{j=1}^{N}B_{j}^{2}\varepsilon^{2}\chi_{j}\psi_{\eta_{j}}\st \frac{z}{B_{j}\varepsilon} \dt+ \textbf{H}_{\hg,\kg}^{o}+ f_{\bg,\cg,\hg,\kg}^{o}\dt
\end{align}

On $M_{\rep}$ we have to solve the following PDE:

\begin{align}\label{eq: PDEbase}
\Lg f_{\bg,\cg,\hg,\kg}^{o}+\mathfrak{s}_{\bg,\cg,\hg,\kg}+ \varepsilon^{2m}\nu_{\bg,\cg}=&-\Lg\st\sum_{j=1}^{N}B_{j}^{2}\varepsilon^{2}\chi_{j}\psi_{\eta_{j}}\st \frac{z}{B_{j}\varepsilon} \dt\dt -\Lg \textbf{H}_{\hg,\kg}^{o}\\
&+2\NN _{\omega}\st -\varepsilon^{2m}{\bf{G}}_{\bg,\cg}+\sum_{j=1}^{N}B_{j}^{2}\varepsilon^{2}\chi_{j}\psi_{\eta_{j}}\st \frac{z}{B_{j}\varepsilon} \dt+ \textbf{H}_{\hg,\kg}^{o}+ f_{\bg,\cg,\hg,\kg}^{o}\dt\,.
\end{align}

\noindent The assumption of Proposition \ref{crucialbase} that there exist $\bg\in \st\RR^{+}\dt^{N}$ and $\cg\in \RR^{N}$
\begin{equation}
\st b_{j}\Delta_{\omega}\varphi_{i}\st p_{j} \dt+c_{j}\varphi_{i}\st p_{j} \dt \dt_{\substack{1\leq i\leq d\\1\leq j\leq N}}
\end{equation}  
enables us, making use of Remark \ref{inversadef},  to invert the operator $\Lg$ on $M_{\p}$ so we look for a PDE defined on the whole $M_{\p}$ and such that on $M_{\rep}$ restricts to the one written above. To this aim we introduce a truncation-extension operator on weighted H\"older spaces.  

\begin{definition}
Let $f\in C_{\delta}^{0,\alpha}\st  M \dt $ we define $\mathcal{E}_{\rep}:C_{\delta}^{0,\alpha}\st  M \dt \rightarrow C_{\delta}^{0,\alpha}\st  M \dt $
\begin{displaymath}
\mathcal{E}_{\rep}\st  f \dt :\left\{ \begin{array}{ll}
f\st  z \dt  & z\in B_{2\rep}\setminus B_{\rep}   \\
f\st  \rep\frac{z}{|z|} \dt \chi\st  \frac{|z|}{\rep} \dt  & z\in B_{ \rep }\setminus B_{\frac{\rep}{2}}\\
0  & z\in B_{\frac{\rep}{2}}  
\end{array} \right.
\end{displaymath}
with $\chi\in C^{\infty}\st [0,+\infty) \dt$ a cutoff function identically $1$ on $[1,+\infty)$ and identically $0$ on $\sq 0,\frac{1}{2} \dq$. 
\end{definition}

Now if we set

\begin{equation}
\mathfrak{n}_{\bg,\cg,\hg,\kg}:= 2\mathfrak{s}_{\bg,\cg,\hg,\kg}+ \varepsilon^{2m}\nu_{\bg,\cg}
\end{equation}

and use the truncation operator we find our differential equation. 

\begin{align}
\Lg f_{\bg,\cg,\hg,\kg}^{o}+\mathfrak{n}_{\bg,\cg,\hg,\kg}=&-\mathcal{E}_{\rep}\Lg\st\sum_{j=1}^{N}B_{j}^{2}\varepsilon^{2}\chi_{j}\psi_{\eta_{j}}\st \frac{z}{B_{j}\varepsilon} \dt\dt -\mathcal{E}_{\rep}\Lg \textbf{H}_{\hg,\kg}^{o}\\
&+2\mathcal{E}_{\rep}\NN _{\omega} \st \varepsilon^{2m}{\bf{G}}_{\bg,\cg}+\sum_{j=1}^{N}B_{j}^{2}\varepsilon^{2}\chi_{j}\psi_{\eta_{j}}\st \frac{z}{B_{j}\varepsilon} \dt+ \textbf{H}_{\hg,\kg}^{o}+ f_{\bg,\cg,\hg,\kg}^{o}\dt
\end{align}

Now using the inverse of $\Lg$ of Remark	 \ref{inversadef} we construct the following operator 
\begin{equation}
\mathcal{T}:\Cqdd \rightarrow \Cqdd\,.
\end{equation}

\begin{align}
\mathcal{T}\st f,\hg,\kg \dt=&\mathbb{J}^{(\delta)}_{\bf{1},\cccc}\sq-\mathcal{E}_{\rep}\Lg\st\sum_{j=1}^{N}B_{j}^{2}\varepsilon^{2}\chi_{j}\psi_{\eta_{j}}\st \frac{z}{B_{j}\varepsilon} \dt\dt -\mathcal{E}_{\rep}\Lg \textbf{H}_{\hg,\kg}^{o}\right.\\
&\left.\phantom{aaaaa}+2\mathcal{E}_{\rep}\NN _{\omega}\st -\varepsilon^{2m}{\bf{G}}_{\bg,\cg}+\sum_{j=1}^{N}B_{j}^{2}\varepsilon^{2}\chi_{j}\psi_{\eta_{j}}\st \frac{z}{B_{j}\varepsilon} \dt+ \textbf{H}_{\hg,\kg}^{o}+ f_{\bg,\cg,\hg,\kg}^{o}\dt-\mathfrak{n}\dq
\end{align}

with

\begin{align}
\mathfrak{n} \vol\st M \dt=&\int_{M}\sq -\mathcal{E}_{\rep}\Lg\st\sum_{j=1}^{N}B_{j}^{2}\varepsilon^{2}\chi_{j}\psi_{\eta_{j}}\st \frac{z}{B_{j}\varepsilon} \dt\dt -\mathcal{E}_{\rep}\Lg \textbf{H}_{\hg,\kg}^{o}  \dq\, d\mu_{\omega}\\
&+2\int_{M} \mathcal{E}_{\rep}\NN _{\omega}\st -\varepsilon^{2m}{\bf{G}}_{\bg,\cg}+\sum_{j=1}^{N}B_{j}^{2}\varepsilon^{2}\chi_{j}\psi_{\eta_{j}}\st \frac{z}{B_{j}\varepsilon} \dt+ \textbf{H}_{\hg,\kg}^{o}+ f\dt  \,d\mu_{\omega}\,.
\end{align}

\noindent We prove the existence of a solution of equation \eqref{eq: PDEbase} by finding a fixed point of the operator $\mathcal{T}$ hence showing it satisfies the assumptions of  contraction Theorem. More precisely we want to prove that there exist a domain $\Omega\subset \Cqdd$ such that $\mathcal{T}\st \Omega \dt\subseteq \Omega$ and $\mathcal{T}$ is a contraction on $\Omega$. The first step  is looking at $\mathcal{T}\st 0,\textbf{0},\textbf{0} \dt$ that heuristically  tells us ``how far'' is  the metric
\begin{equation}
\omega +i\dd\st  - \varepsilon^{2m}{\bf{G}}_{\bg,\cg}+\sum_{j=1}^{N}B_{j}^{2}\varepsilon^{2}\chi_{j}\psi_{\eta_{j}}\st \frac{z}{B_{j}	\varepsilon} \dt  \dt
\end{equation}   
from being Kcsc on $M_{\rep}$.

\begin{lemma}\label{stimascheletrobase}
Under the assumptions of Proposition \ref{crucialbase} the following estimate holds
\begin{equation}
\left\|\mathcal{T}\st 0,\textbf{0},\textbf{0} \dt\right\|_{\Cqdd}\leq C\varepsilon^{2m+2}\rep^{2-\delta-2m}\,.
\end{equation}

\end{lemma}

\begin{proof}
For the sake of notation throughout this proof we set 
\begin{equation}
\mathfrak{p}_{j}\st z \dt:=B_{j}^{2}\varepsilon^{2}\chi_{j}\psi_{\eta_{j}}\st \frac{z}{B_{j}\varepsilon} \dt
\end{equation}

\noindent  We note that, on $M_{r_0}$, using estimates of Proposition \ref{asintpsieta} we have 
\begin{align}
-\mathcal{E}_{\rep}\Lg\st\sum_{j=1}^{N}\mathfrak{p}_{j}\dt +2\mathcal{E}_{\rep}\NN _{\omega}\st -\varepsilon^{2m}{\bf{G}}_{\bg,\cg}+\sum_{j=1}^{N}\mathfrak{p}_{j}\dt=& \varepsilon^{2m+2}\mathcal{O}_{C^{0,\alpha}\st M_{r_0} \dt}\st 1\dt
\end{align}

\noindent so we have 
\begin{equation}\left\|-\mathcal{E}_{\rep}\Lg\st\sum_{j=1}^{N}\mathfrak{p}_{j}\dt +2\mathcal{E}_{\rep}\NN _{\omega}\st -\varepsilon^{2m}{\bf{G}}_{\bg,\cg}+\sum_{j=1}^{N}\mathfrak{p}_{j}\dt\right\|_{C^{0,\alpha}\st M_{r_0}\dt}\leq C\st g, \boldsymbol{\eta} \dt \varepsilon^{2m+2}\,.\end{equation}

\noindent Now we estimate on $B_{2r_{0}}\st p_{j} \dt$ the quantity 
\begin{equation}\sup_{\rho\in[\rep ,r_0]} \rho^{4-\delta}\left\| -\mathcal{E}_{\rep}\Lg\st\sum_{j=1}^{N}\mathfrak{p}_{j}\dt +2\mathcal{E}_{\rep}\NN _{\omega}\st -\varepsilon^{2m}{\bf{G}}_{\bg,\cg}+\sum_{j=1}^{N}\mathfrak{p}_{j}\dt  \right\|_{C^{0,\alpha}\st  B_{2} \setminus B_{1} \dt}\,. \end{equation}

\noindent On $B_{2r_0}$, we have 
\begin{align}
\Lg\mathfrak{p}_{j} -2\NN _{\omega}\st -\varepsilon^{2m}{\bf{G}}_{\bg,\cg}+\mathfrak{p}_{j}\dt
=& \Delta^{2}\mathfrak{p}_{j}+ \sq\Lg-\Delta^{2}\dq\st\mathfrak{p}_{j}\dt\\
&-2\NN _{eucl}\st - \cgaj  \varepsilon^{2m}B_{j}^{2m}|z|^{2m}+\mathfrak{p}_{j}\dt\\
&+\sq 2\NN _{eucl}\st - \cgaj  \varepsilon^{2m}B_{j}^{2m}|z|^{2m}+\mathfrak{p}_{j}\dt\right.\\
&\phantom{..}\left.\phantom{-}-2\NN _{\omega}\st - \cgaj  \varepsilon^{2m}B_{j}^{2m}|z|^{2m}+\mathfrak{p}_{j}\dt\dq\\
&+\sq 2\NN _{\omega}\st - \cgaj  \varepsilon^{2m}B_{j}^{2m}|z|^{2m}+\mathfrak{p}_{j}\dt\right.\\
&\phantom{-}\left.\phantom{..}-2\NN _{\omega}\st -\varepsilon^{2m}{\bf{G}}_{\bg,\cg}+\mathfrak{p}_{j}\dt\dq\,.
\end{align}

\noindent The metric $\eta_{j}$ is scalar-flat and this fact gives us the algebraic identity
\begin{equation}
-\Delta^{2}\mathfrak{p}_{j}+2\NN _{eucl}\st - \cgaj  \varepsilon^{2m}B_{j}^{2m}|z|^{2m}+\mathfrak{p}_{j}\dt=0
\end{equation}

\noindent This cancellation let us to have better crucial estimates on $f_{\bg,\cg,\hg,\kg}^{o}$:

\begin{equation}\sq  \Lg-\Delta^{2} \dq \st \mathfrak{p}_{j} \dt= \varepsilon^{2m+2}\rho^{-2m-2} \mathcal{O}_{C^{0,\alpha}\st  B_{2} \setminus B_{1} \dt}\st 1 \dt\end{equation}

\begin{equation}
\NN _{eucl}\st - \cgaj  \varepsilon^{2m}B_{j}^{2m}|z|^{2m}+\mathfrak{p}_{j}\dt-\NN _{\omega}\st - \cgaj  \varepsilon^{2m}B_{j}^{2m}|z|^{2m}+\mathfrak{p}_{j}\dt=\varepsilon^{4m}\rho^{-4m}\mathcal{O}_{C^{0,\alpha}\st  B_{2} \setminus B_{1} \dt}\st 1 \dt
\end{equation}

\begin{equation}
\NN _{\omega}\st - \cgaj  \varepsilon^{2m}B_{j}^{2m}|z|^{2m}+\mathfrak{p}_{j}\dt-\NN _{\omega}\st -\varepsilon^{2m}{\bf{G}}_{\bg,\cg}+\mathfrak{p}_{j}\dt=\varepsilon^{4m}\rho^{-4m}\mathcal{O}_{C^{0,\alpha}\st  B_{2} \setminus B_{1} \dt}\st 1 \dt
\end{equation}

\noindent so we can conclude that

\begin{equation}\sup_{\substack{  1\leq j\leq N \\ \rho\in [\rep, r_0]}}\rho^{-\delta+4}\left\| \Lg\mathfrak{p}_{j} -2\NN _{\omega}\st -\varepsilon^{2m}{\bf{G}}_{\bg,\cg}+\mathfrak{p}_{j}\dt\right\|_{C^{0,\alpha}\st  B_{2} \setminus B_{1} \dt} \leq C\st g,\etag \dt\varepsilon^{2m+2}\rep^{2-2m-\delta }\,. \end{equation}

\end{proof}

\noindent  In light of  Lemma \ref{stimascheletrobase} we can take  $\left\|\mathcal{T}\st 0,\textbf{0},\textbf{0} \dt\right\|_{\Cqdd}$ as a reference for the magnitude of the diameter of the $\Omega$ we are looking for.  Indeed if we consider the set of $f\in \Cqdd$ such that 
\begin{equation}
\left\|f\right\|_{\Cqdd}\leq 2 \left\|\mathcal{T}\st 0,\textbf{0},\textbf{0} \dt\right\|_{\Cqdd}=2C\st g,\etag \dt\varepsilon^{2m+2}\rep^{2-2m-\delta }.
\end{equation}
we find our $\Omega$. The fact that 
\begin{equation}
\mathcal{T}:\Omega\rightarrow \Omega
\end{equation}

\noindent and it is a contraction follows from the following two Lemmas.

\begin{lemma}\label{stimabiarmonichebase}
Let $\delta\in (4-2m,5-2m)$ and $(\hg,\kg)\in \dombd$, then 
\begin{equation}\left\|\mathcal{E}_{\rep}\Lg\textbf{H}^{o}_{\hg,\kg}\right\|_{C_{\delta-4}^{0,\alpha}\st M_{\p} \dt}\leq C\st g \dt\left\|\hg^{\dagger},\kg^{\dagger}\right\|_{\mathcal{B}^{\alpha}} \rep ^{2-\delta}\,.\end{equation}
\end{lemma}
\begin{proof} 

This is a straightforward computation using Remark \ref{nolinear}.
\end{proof}

\begin{lemma}
Let  $ \st \hg',\kg' \dt \in \dombd$ and 
$f,f'\in \Cqdd$ such that 
\begin{equation}
\left\|f\right\|_{\Cqdd},\left\|f'\right\|_{\Cqdd}\leq 2 \left\|\mathcal{T}\st 0,\textbf{0},\textbf{0} \dt\right\|_{\Cqdd}.
\end{equation}
\noindent If assumptions of Proposition \ref{crucialbase} are satisfied then then the following estimates hold:

\begin{align}
\left\|\mathcal{T}\st f,\hg,\kg \dt-\mathcal{T}\st 0,\hg,\kg \dt\right\|_{\Cqdd}\,\leq&\, \frac{1}{2}\left\|\mathcal{T}\st 0,\textbf{0},\textbf{0} \dt\right\|_{\Cqdd}\\
\left\|\mathcal{T}\st f,\hg,\kg \dt-\mathcal{T}\st f',\hg,\kg \dt\right\|_{\Cqdd}\,\leq&\, \frac{1}{2}\left\|f-f'\right\|_{\Cqdd}\\
\left\|\mathcal{T}\st f,\hg,\kg \dt-\mathcal{T}\st f,\hg',\kg' \dt\right\|_{\Cqdd}\,\leq&\, \frac{1}{2}\left\|\hg-\hg',\kg-\kg'\right\|_{\mathcal{B}^{\alpha}}\,.
\end{align}

\end{lemma}
\begin{proof}
Follows by direct computation as \cite[Lemma 5.2]{ap2}.
\end{proof}

The proof of Proposition \ref{crucialbase} is now complete.

\subsection{Kcsc metrics on the truncated model spaces:} we now want to perform a similar analysis as done above  on  the model spaces $X_{j}$'s. 
To keep notations as short as possible we drop the subscript $j$. We then want to find $F_{B,\tilde{h},\tilde{k}}^{I}\in C^{4,\alpha}\st X_{\frac{\Rep}{B}} \dt$ such that 

\begin{equation}\eta_{B,\tilde{h},\tilde{k}}:=B^{2}\eta+i\dd F_{B,\tilde{h},\tilde{k}}^{I}\end{equation} 

\noindent is a metric on $X_{\frac{\Rep}{B}}$ and 
\begin{equation}
\s_{\eta}\st F_{B,\tilde{h},\tilde{k}}^{I} \dt=\varepsilon^{2}\st s_{\omega}+\mathfrak{s}_{\bg,\cg,\hg,\kg}\dt\,.
\end{equation}

\noindent The function $F_{B,\tilde{h},\tilde{k}}^{I}$, as in the base case, will be made of three blocks: a skeleton that keeps the metric near to a Kcsc metric, modifications of biharmonic extensions of functions on the unit sphere that will allow us to glue the metric on the base with those on  the models and a small perturbation $f_{B,\tilde{h},\tilde{k}}^{I}$ that ensures the constancy of the scalar curvature. Also in the case of model spaces we can't use the theory of \cite{ap2} as it stands, we would get too rough estimates that won't be enough to conclude the gluing process. Indeed we would get the following result.

\begin{prop}\label{famigliamodelloap2}
Let $\st\tilde{h},\tilde{k}\dt \in C^{4,\alpha}\st \Sp^{2m-1}/\Gamma \dt\times  C^{4,\alpha}\st \Sp^{2m-1}/\Gamma \dt$ such that 
\begin{equation}
\left\|\tilde{h}\right\|_{C^{4,\alpha}\st \Sp^{2m-1}/\Gamma \dt},\left\|\tilde{k}\right\|_{C^{2,\alpha}\st \Sp^{2m-1}/\Gamma \dt}\leq \kappa \rep^{4}\varepsilon^{-2}.
\end{equation}
Let $\delta\in (4-2m,5-2m)$ then there exist $f_{\tilde{h}\tilde{k}}^{I}\in C_{\delta}^{4,\alpha}\st X \dt$ such that
\begin{equation}\left\|f_{\tilde{h}\tilde{k}}^{I}\right\|_{C_{\delta}^{4,\alpha}\st X \dt}\leq C\st \etag  \dt \varepsilon^{2}\Rep^{4-\delta}\end{equation}
and on $X_{\Rep}$
\begin{equation}\eta_{\tilde{h},\tilde{k}}=\eta+i\dd\st \textbf{H}_{\tilde{h},\tilde{k}}^{I}+f_{\tilde{h}\tilde{k}}^{I} \dt\end{equation}
is a K\"ahler metric of constant scalar curvature  
\begin{equation}\s_{\eta}\st  \textbf{H}_{\tilde{h},\tilde{k}}^{I}+ f_{\tilde{h},\tilde{k}}^{I}  \dt=\varepsilon^{2}\st s_{\omega}+\nu \dt\,.\end{equation}
Moreover the metric $\eta_{\tilde{h},\tilde{k}}$ depends continuously on $\st \tilde{h},\tilde{k} \dt$.
\end{prop}

\noindent This  kind of result isn't useful even if we use Proposition \ref{crucialbase} on the base. Also in this case we need to refine the construction, in particular, we need to add terms coming from the potentials of the metric on the base and to refine the construction of biharmonic extensions. To improve the estimates we need to modify $\eta$ by  $\psi_{\omega}$ that is the Kahler potential of $\omega $ at the point where we want to insert the Ricci-flat model. The naive attempt of adding the term $\psi_{\omega}$ only rescaling and multiplying it by a cutoff  function won't work since it won't lead us to the estimates we need. We recall that by Lemma \ref{proprietacsck} 
\begin{eqnarray}
\psi_{\omega}  &=& \sum_{k=0}^{+\infty}\Psi_{4+k}\,,\\
-\Delta^{2}\Psi_{4}&=& 2s_{\omega}\,,\\
-\Delta^{2}\Psi_{5}&=&0\,.
\end{eqnarray}

\noindent What we want to do to improve the estimates is to correct $\Psi_{4},\Psi_{5}$ with decaying functions in such a way they are in $\ker\st \Le \dt$ as much as possible. We could correct $\Psi_{4},\Psi_{5}$ with functions in the space $C_{\delta}^{4,\alpha}\st X \dt$ with $\delta\in (4-2m,5-2m)$ but again we would get too rough estimates. We will correct $\Psi_{4},\Psi_{5}$ with functions that we know almost explicitly and functions decaying much faster than those in $C_{\delta}^{4,\alpha}\st X \dt$ with $\delta\in (4-2m,5-2m)$. 
\newline
\phantom{hhhhh}
\newline
 {\em For the rest of the subsection $\chi$ will denote a smooth cutoff function identically $0$ on $X_{\frac{R_{0}}{2B}}$ and identically $1$ outside $X_{\frac{R_{0}}{B}}$.}
\newline
\phantom{hhhhh}
\newline
Using Lemmas \ref{espansioniALE} and \ref{proprietacsck} it is easy to see that 
\begin{equation}
\Delta_{\eta}^{2} \chi \Psi_{4}=-2s_{\omega}+\st \Phi_{2}+\Phi_{4} \dt\chi|x|^{-2m}+ \mathcal{O}\st |x|^{-2-2m} \dt
\end{equation}  
and
\begin{equation}
\Delta_{\eta}^{2} \chi \Psi_{5}=\st \Phi_{3}+\Phi_{5} \dt\chi|x|^{1-2m}+ \mathcal{O}\st |x|^{-1-2m} \dt\,.
\end{equation}
\noindent If we set 
\begin{align}
u_{4}:= \st \Phi_{2}+\Phi_{4} \dt\chi|x|^{-2m}&\qquad\textrm{for }m\geq3\\
u_{4}:= \st \Phi_{2}+\Phi_{4} \dt\chi\log\st|x|\dt&\qquad\textrm{for }m=2\\ 
 \end{align}

and
\begin{equation}
u_{5}:= \st \Phi_{3}+\Phi_{5} \dt\chi|x|^{5-2m}
\end{equation}

\noindent with suitable $\Phi_{2},\Phi_{4},\Phi_{3},\Phi_{5}\in C^{\infty}\st \Sp^{2m-1} \dt$ spherical harmonics then 
\begin{equation}
\Delta_{\eta}^{2}\st \chi \Psi_{4}+u_{4}\dt=-2s_{\omega}+ \mathcal{O}\st |x|^{-2-2m} \dt
\end{equation}  
and
\begin{equation}
\Delta_{\eta}^{2}\st \chi \Psi_{5}+u_{5}\dt= \mathcal{O}\st |x|^{-1-2m} \dt\,.
\end{equation}

\noindent Now we want to find $v_{4}\in C_{\delta}^{4,\alpha}\st X \dt$ with $\delta\in (2-2m,3-2m)$ and $v_{5}\in C_{\delta}^{4,\alpha}\st X \dt$ with $\delta\in (3-2m,4-2m)$ such that 
\begin{equation}
\Delta_{\eta}^{2}\st \chi \Psi_{4}+u_{4}+v_{4}\dt=-2s_{\omega}
\end{equation}  
and
\begin{equation}
\Delta_{\eta}^{2}\st \chi \Psi_{5}+u_{5}+v_{5}\dt= 0\,.
\end{equation}

\noindent Proposition \ref{isomorfismopesati} and Lemma \ref{GAP} tell us that we can find such $v_{4},v_{5}$ if and only if the integrals  
\begin{equation}\label{eq:int4}
\int_{X}\sq\Delta_{\eta}^{2}\st \chi \Psi_{4}+u_{4}\dt +2s_{\omega} \dq \, d\mu_{\eta}
\end{equation}
\noindent and
\begin{equation}\label{eq:int5}
\int_{X}\Delta_{\eta}^{2}\st \chi \Psi_{5}+u_{5}\dt\,d\mu_{\eta} 
\end{equation}
\noindent vanish identically. We have to compute the two integrals above and the crucial tool for the calculations will be the following Lemma.

\begin{lemma}\label{misuraeuclidea}
Let $\st X, h, \eta \dt$ be a Ricci flat ALE K\"ahler resolution of $\CC^{m}/\Gamma$ 
\begin{equation}\pi:X\rightarrow \CC^{m}/\Gamma\,,\end{equation}
 then on $X\setminus \pi^{-1}\st 0 \dt$ we have
\begin{equation}d\mu_{\eta}= d\mu_{0}\,,\end{equation}
and for $\rho>0$
\begin{equation}\text{Vol}_{\eta}\st X_{\rho}\dt=\frac{|\Sp^{2m-1}|}{2m\left| \Gamma\right|}\rho^{2m}\,.\end{equation}
\end{lemma}
\begin{proof}
Let $\pi_{\Gamma}:\CC^{m}\rightarrow \CC^{m}/\Gamma$ the canonical holomorphic quotient map, since 
\begin{equation}\Ric\st \eta \dt=0\,,\end{equation}
on $\st\CC^{m}\setminus B_{\rho}\dt/\Gamma$  we have 
\begin{equation}i\dd\sq\log\st\det\st \st\pi_{\Gamma}\dt^{*}\st\pi^{-1}\dt^{*}\eta\dt\dt\dq=0\,.\end{equation} 
We want to prove that on \begin{equation}\CC^{m}\setminus \sg 0 \dg\end{equation}
\begin{equation}\det\st\st\pi_{\Gamma}\dt^{*}\st\pi^{-1}\dt^{*}\eta\dt\equiv C\qquad C\in \RR\,.\end{equation}
By Proposition \ref{asintpsieta} we have  on $\CC^{m}\setminus B_{\rho}$ 
\begin{equation}\st\pi_{\Gamma}\dt^{*}\st\pi^{-1}\dt^{*}\eta=\frac{\delta_{i\bar{\jmath}}}{2}- \cga \dd \left|x\right|^{2-2m} +	\dd \mathcal{O}\st \left|x\right|^{2-2m}  \dt \end{equation}

\noindent that implies immediately

\begin{align}
\log\st \det\st \st\pi_{\Gamma}\dt^{*}\st\pi^{-1}\dt^{*}\eta \dt \dt=&-m\log\st 2 \dt+\mathcal{O}\st \left|x\right|^{-2-2m} \dt\,.
\end{align}

\noindent On $\CC^{m}\setminus B_{\rho}$ we have
\begin{align}
i\dd \log\st \det\st \st\pi_{\Gamma}\dt^{*}\st\pi^{-1}\dt^{*}\eta \dt \dt=&-id\st \partial \log\st \det\st  \st\pi_{\Gamma}\dt^{*}\st\pi^{-1}\dt^{*}\eta \dt\dt\dt\,, 
\end{align}

so 
\begin{equation}\partial \log\st \det\st \st\pi_{\Gamma}\dt^{*}\st\pi^{-1}\dt^{*}\eta \dt \dt\in H^{1}\st \CC^{m}\setminus B_{\rho},\CC \dt\end{equation} but $H^{1}\st \CC^{m}\setminus B_{\rho},\CC \dt=0$  
and there exists $h_{1}\in C^{1}\st \CC^{m}\setminus B_{\rho},\CC  \dt$ such that
\begin{equation}\partial \log\st \det\st \st\pi_{\Gamma}\dt^{*}\st\pi^{-1}\dt^{*}\eta \dt \dt=dh_{1}=\partial h_{1} +\overline{\partial}h_{1}\qquad\Rightarrow\qquad \overline{\partial}h_{1}=0\,.\end{equation}

Analogously, there is $h_{2}\in C^{1}\st \CC^{m}\setminus B_{\rho},\CC  \dt$ such that
\begin{equation}\overline{\partial} \sq\log\st \det\st \st\pi_{\Gamma}\dt^{*}\st\pi^{-1}\dt^{*}\eta \dt \dt -h_{1} \dq=dh_{2}=\partial h_{2}+\overline{\partial}h_{2}\qquad \Rightarrow\qquad \partial h_{2}=0\,.\end{equation}

It is now clear that 
\begin{align}
d\sq \log\st \det\st \st\pi_{\Gamma}\dt^{*}\st\pi^{-1}\dt^{*}\eta \dt \dt-h_{1}-h_{2} \dq=&0\,.
\end{align}
We conclude that on $\CC^{m}\setminus B_{\rho}$
\begin{equation}\log\st \det\st \st\pi_{\Gamma}\dt^{*}\st\pi^{-1}\dt^{*}\eta \dt \dt=h_{1}+h_{2}+K\qquad K\in \RR \qquad \mathfrak{Im}h_{2}=-\mathfrak{Im}h_{1} \end{equation}
moreover  $h_{1},\overline{h_{2}}$ are holomorphic on $\CC^{m}\setminus B_{\rho}$ and by Hartogs extension theorem they are extendable to functions $H_{1},H_{2}$ holomorphic on $\CC^{m}$. Since $H_{1},H_{2}$ are holomorphic, their real and imaginary parts are harmonic with respect to the euclidean metric on $\CC^{m}$ and by assumptions on $\eta$  we have on  $\CC^{m}\setminus B_{\rho}$
\begin{equation}\mathfrak{Re}H_{1}+\mathfrak{Re}H_{2}+K=-m\log\st 2 \dt+\mathcal{O}\st \left|x\right|^{-2-2m} \dt\,.\end{equation}  
Since $\mathfrak{Re}H_{1}+\mathfrak{Re}H_{2}+K$ is harmonic and bounded, Liouville theorem implies it is constant, so 
\begin{equation}\log\st \det\st \st\pi_{\Gamma}\dt^{*}\st\pi^{-1}\dt^{*}\eta \dt \dt=C\Rightarrow \det\st \st\pi_{\Gamma}\dt^{*}\st\pi^{-1}\dt^{*}\eta \dt=\frac{1}{2^{m}}\end{equation}

\noindent We can now see that

\begin{align}
\frac{1}{m!}\st\pi_{\Gamma}\dt^{*}\sq \st\pi^{-1}\dt^{*}\eta \dq^{\wedge m}=&d\mu_{0}\,.
\end{align}

\noindent and then

\begin{equation}
\text{Vol}_{\eta}\st X_{\rho}\dt=\int_{B_{\rho}/\Gamma\setminus \sg 0 \dg}d\mu_{\st\pi^{-1}\dt^{*}\eta}=\frac{\mu\st S^{2m-1} \dt}{2m\left| \Gamma\right|}\rho^{2m}
\end{equation}

\noindent 	so the lemma follows.
\end{proof}

\noindent We start computing integral \eqref{eq:int4}. By means of divergence Theorem and Lemma \ref{misuraeuclidea} we can write 
\begin{equation}
\int_{X}\sq\Delta_{\eta}^{2}\st \chi \Psi_{4}+u_{4}\dt +2s_{\omega} \dq \, d\mu_{\eta}=\lim_{\rho\rightarrow+\infty}\sq \int_{\partial X_{\rho}}\partial_{\nu}\Delta_{\eta}\st  \chi \Psi_{4}\dt d\mu_{\eta}+\frac{s_{\omega}| \Sp^{2m-1} | }{m\left| \Gamma\right|}\rho^{2m}\dq\,,
\end{equation} 
\noindent with $\nu$ outward unit normal to the boundary. Then using Proposition \ref{proprietacsck} and Lemma \ref{espansioniALE}
\begin{equation}
\partial_{\nu}\Delta_{\eta}\Psi_{4}\left.d\mu_{\eta}\right|_{\partial X_{\rho}} = \sq -\frac{s_{\omega}}{m} \rho^{2m}- \frac{4  \cga \st m-1 \dt^{2}s_{\omega}}{m\st m+1 \dt}+\mathcal{O}\st 1\dt\st \Phi_{2}+\Phi_{4}\dt +\mathcal{O}\st\rho^{-1}\dt\dq \left.d\mu_{0}\right|_{\Sp^{2m-1} / \Gamma }\,,
\end{equation}
and integrating we obtain
\begin{equation}
\int_{X}\sq\Delta_{\eta}^{2}\st \chi \Psi_{4}+u_{4}\dt +2s_{\omega} \dq \, d\mu_{\eta} = - \frac{4  \cga \st m-1 \dt^{2}| \Sp^{2m-1} | s_{\omega}}{m\st m+1 \dt |\Gamma|}\,.
\end{equation}
\noindent Hence  can't solve the equation
\begin{equation}
\Delta_{\eta}^{2}\st \chi \Psi_{4}+u_{4}+v_{4}\dt=-2s_{\omega}
\end{equation}
for $v_{4}\in C_{\delta}^{4,\alpha}\st X \dt$ with $\delta\in (2-2m,3-2m)$, but adding a function that is ``almost in $\ker\st \Delta_{\eta} \dt$'' we reach our goal, precisely we can solve the equation
\begin{displaymath}
\begin{array}{lcl}
\Delta_{\eta}^{2}\st \chi \Psi_{4}+u_{4}+\frac{  \cga \st m-1 \dt s_{\omega} }{2\st m-2 \dt   m\st m+1 \dt }\chi |x|^{4-2m}+v_{4}\dt&=&-2s_{\omega}\qquad \textrm{ for }m\geq3\\
&&\\
\Delta_{\eta}^{2}\st \chi \Psi_{4}+u_{4}-\frac{  \cga  s_{\omega} }{6 }\chi \log\st|x|\dt+v_{4}\dt&=&-2s_{\omega}\qquad \textrm{ for }m=2
\end{array}
\end{displaymath}
\noindent for $v_{4}\in C_{\delta}^{4,\alpha}\st X \dt$ with $\delta\in (2-2m,3-2m)$. In a completely analogous way we can compute integral \eqref{eq:int5} that vanishes identically  and so we can solve the equation 
\begin{equation}
\Delta_{\eta}^{2}\st \chi \Psi_{5}+u_{5}+v_{5}\dt= 0\,.
\end{equation}
\noindent for $v_{5}\in C_{\delta}^{4,\alpha}\st X \dt$ with $\delta\in (3-2m,4-2m)$. If we define  
\begin{equation}
V:=\varepsilon^{2}B^{4}\st u_{4}+ v_{4}\dt+\varepsilon^{3}B^{5}\st u_{5}+v_{5}\dt \,.
\end{equation}
\noindent then we can define as the skeleton the function in $C^{4,\alpha}\st X_{\frac{\Rep}{B}} \dt$
\begin{align}
\label{eq:scheletromod1}\frac{1}{\varepsilon^{2}}\chi\psi_{\omega}\st \frac{B x}{\varepsilon} \dt+\frac{  \cga \st m-1 \dt s_{\omega}B^{4}\varepsilon^{2} }{2\st m-2 \dt   m\st m+1 \dt }\chi |x|^{4-2m}+V&\qquad \textrm{ for }m\geq3\,,\\
\\
\label{eq:scheletromod2} \frac{1}{\varepsilon^{2}}\chi\psi_{\omega}\st \frac{B x}{\varepsilon} \dt-\frac{  \cga  s_{\omega}B^{4}\varepsilon^{2} }{6 }\chi \log\st|x|\dt+V&
\qquad \textrm{ for }m=2\,.
\end{align}

\noindent We then construct a function  $\chi H_{\tilde{h},\tilde{k}}^{I}\st \frac{ B x}{\Rep} \dt \in C^{4,\alpha}\st X_{\frac{\Rep}{B}} \dt$ that resembles inner biharmonic extensions. We want to build a function on $X$ that is ``almost'' in the kernel of $\Delta_{\eta}^{2}$ and, for our purposes, we need a more refined construction than \cite{ap2}. We note that 

\begin{eqnarray}
\Delta_{\eta}^{2}\st\chi|x|^{2}\dt&=& \mathcal{O}\st |x|^{-2-2m} \dt\,,\\
\Delta_{\eta}^{2}\st \chi|x|^{2}\Phi_{2}\dt&=&\mathcal{O}\st |x|^{-2m-2} \dt\,,\\
\Delta_{\eta}^{2}\st\chi|x|^{3}\Phi_{3}\dt&=& \chi|x|^{-1-2m} \Phi_{3} +\mathcal{O}\st|x|^{-3-2m}\dt\,.
\end{eqnarray}

\noindent As for the skeleton we want to correct the functions on the left hand sides of equations in such a way they are in $\ker\st \Delta_{\eta} \dt$. Precisely we want to solve the equations 
\begin{eqnarray}
\Delta_{\eta}^{2}\st\chi|x|^{2}+v^{(0)}\dt&=& 0\,, \\
\Delta_{\eta}^{2}\st \chi|x|^{2}\Phi_{2}+v^{(2)}\dt&=&0\,,\\
\Delta_{\eta}^{2}\st\chi|x|^{3}\Phi_{3}+ u^{(3)} +v^{(3)}\dt&=& 0\,.
\end{eqnarray}

\noindent with $v^{(0)},v^{(2)},v^{(3)}\in C_{\delta}^{4,\alpha}\st X \dt$ for $\delta\in (2-2m,3-2m)$ and  
\begin{equation}
u^{(3)}:=\chi|x|^{3-2m} \Phi_{3}
\end{equation}
\noindent for a suitable  spherical harmonic $\Phi_{3}$. The existence of  $v^{(0)},v^{(2)},v^{(3)}$ follows from Proposition \ref{isomorfismopesati}, Lemma \ref{GAP} and 
\begin{equation}
\int_{X}\Delta_{\eta}^{2}\st\chi|x|^{2}\dt\,d\mu_{\eta}=\int_{X}\Delta_{\eta}^{2}\st\chi|x|^{2}\Phi_{2}\dt\,d\mu_{\eta}=\int_{X}\Delta_{\eta}^{2}\st\chi|x|^{3}\Phi_{3}\dt\,d\mu_{\eta}=0  
\end{equation}
\noindent as one can easily check using exactly the same ideas exposed for the skeleton. We are ready to define the function $\hkii \in C^{4,\alpha}(X_{\frac{\Rep}{B}})$ 
\begin{align}
\hkii :=& H_{\tilde{h},\tilde{k}}^{I}\st 0\dt+ \chi\st H_{\tilde{h},\tilde{k}}^{I}\st \frac{ B x}{\Rep} \dt-  H_{\tilde{h},\tilde{k}}^{I}\st 0\dt\dt+\frac{\tilde{k}^{(0)}B^2}{4m \Rep^2}v^{(0)}  \\
&+ \st \tilde{h}^{(2)}-\frac{\tilde{k}^{(2)}}{4(m+2)} \dt \frac{B^{2}v^{(2)}}{\Rep^2}+ \st \tilde{h}^{(3)}-\frac{\tilde{k}^{(3)}B^{3}}{4(m+3)}  \dt  \frac{u^{(3)}+v^{(3)}}{\Rep^3}\,.\label{eq:thki}
\end{align}

\noindent We are now ready to state the main result on the model spaces.

\begin{prop}\label{crucialmodello}
Let $(X,h, \eta)$ an ALE  Ricci-Flat K\"ahler resolution of $\CC^{m}/\Gamma$. Let $\st \tilde{h},\tilde{k} \dt\in C^{4,\alpha}\st \Sp^{2m-1}/\Gamma \dt\times C^{2,\alpha}\st \Sp^{2m-1}/\Gamma \dt$ such that
\begin{equation}
\left\|\tilde{h}^{(0)},\tilde{k}^{(0)}\right\|_{\mathcal{B}^{\alpha}}\leq \kappa \rep^{\beta}\varepsilon^{-2}\,,
\end{equation}

\begin{equation}
\left\|\tilde{h}^{(\dagger)},\tilde{k}^{(\dagger)}\right\|_{\mathcal{B}^{\alpha}}\leq \kappa \rep^{\sigma}\varepsilon^{-2}\,.
\end{equation}

\noindent Then there exists $f_{B,\tilde{h},\tilde{k}}^{I}\in C_{\delta}^{4,\alpha}\st X \dt$ for $\delta\in (4-2m,5-2m)$ with 
\begin{equation}
\left\|f_{B,\tilde{h},\tilde{k}}^{I}\right\|_{C_{\delta}^{4,\alpha}\st X \dt}\leq C\st g,\kappa \dt \rep^{\sigma-2}\Rep^{-2}
\end{equation}
such that  on $X_{\frac{\Rep}{B}}$
\begin{equation}
\s_{B^{2}\eta}\st \frac{1}{\varepsilon^{2}}\chi\psi_{\omega}\st \frac{B x}{\varepsilon} \dt+\frac{  \cga \st m-1 \dt s_{\omega}B^{4}\varepsilon^{2} }{2\st m-2 \dt   m\st m+1 \dt }\chi |x|^{4-2m}+V+\textbf{H}_{\tilde{h},\tilde{k}}^{I}+ f_{B,\tilde{h},\tilde{k}}^{I}\dt\equiv \varepsilon^{2}\st s_{\omega}+\mathfrak{s}_{\bg,\cg,\hg,\kg}\dt.
\end{equation}
for $m\geq 3$, and
\begin{equation}
\s_{B^{2}\eta}\st \frac{1}{\varepsilon^{2}}\chi\psi_{\omega}\st \frac{B x}{\varepsilon} \dt-\frac{  \cga  s_{\omega}B^{4}\varepsilon^{2} }{6 }\chi \log\st|x|\dt+V+\textbf{H}_{\tilde{h},\tilde{k}}^{I}+ f_{B,\tilde{h},\tilde{k}}^{I} + C\dt\equiv \varepsilon^{2}\st s_{\omega}+\mathfrak{s}_{\bg,\cg,\hg,\kg}\dt.
\end{equation}
for $m=2$, where $C$ is the constant term in the expansion  at $B_{2r_{0}}\st p \dt\setminus B_{\rep}\st p \dt$ of  
\begin{equation}
- \varepsilon^{2m}{{\bf{G}}}_{\bg,\cg}+\sum_{j=1}^{N}B_{j}^{2}\varepsilon^{2}\chi_{j}\psi_{\eta_{j}}\st \frac{z}{B_{j} \varepsilon} \dt+\textbf{H}_{\hg,\kg}^{o}+ f_{\bg,\cg,\hg,\kg}^{o}\,.
\end{equation}
\end{prop}

\noindent Again  we want to find a PDE that  that $\csfii$ has to solve on $X_{\frac{\Rep}{B}}$. Since we are looking for small perturbations of $\eta$ we can use on $X_{\frac{\Rep}{B}}$ the expansion \eqref{eq:espsg} of operator $\s_{\eta}$ given by Proposition \ref{espsg}.  

\begin{align}
\varepsilon^{2}\st s_{\omega}+\mathfrak{s}_{\bg,\cg,\hg,\kg}\dt
=&\varepsilon^{2}s_{\omega}-\frac{1}{2\varepsilon^{2}B^{4}}\Delta_{\eta}^{2}\st \sum_{k=6}^{+\infty}\chi\Psi_{k}\st \frac{B x}{\varepsilon} \dt\dt-\frac{1}{2B^{4}}\Delta_{\eta}^{2}\st \textbf{H}_{\tilde{h},\tilde{k}}^{I}\dt-\frac{1}{2B^{4}}\Delta_{\eta}^{2}\st f_{B,\tilde{h},\tilde{k}}^{I}  \dt\\
&+\NN _{B^{2}\eta}\st \frac{1}{\varepsilon^{2}}\chi\psi_{\omega}\st \frac{B x}{\varepsilon} \dt+\frac{  \cga \st m-1 \dt s_{\omega}B^{4}\varepsilon^{2} }{2\st m-2 \dt   m\st m+1 \dt }\chi |x|^{4-2m}+V+\textbf{H}_{\tilde{h},\tilde{k}}^{I}+ f_{B,\tilde{h},\tilde{k}}^{I}  \dt\\
\end{align} 

\noindent On $X_{\frac{\Rep}{B}}$ we have then to solve the following PDE:

\begin{align}\label{eq: PDEmodello}
\Delta_{\eta}^{2}\st f_{B,\tilde{h},\tilde{k}}^{I}  \dt=&-2\varepsilon^{2}B^{4}\mathfrak{s}_{\bg,\cg,\hg,\kg}-\frac{1}{\varepsilon^{2}}\Delta_{\eta}^{2}\st \sum_{k=6}^{+\infty}\chi\Psi_{k}\st \frac{B x}{\varepsilon} \dt\dt-\Delta_{\eta}^{2}\st \textbf{H}_{\tilde{h},\tilde{k}}^{I}\dt\\
&+2B^{4}\NN_{B^{2}\eta}\st \frac{1}{\varepsilon^{2}}\chi\psi_{\omega}\st \frac{B x}{\varepsilon} \dt+\frac{  \cga \st m-1 \dt s_{\omega}B^{4}\varepsilon^{2} }{2\st m-2 \dt   m\st m+1 \dt }\chi |x|^{4-2m}+V+\textbf{H}_{\tilde{h},\tilde{k}}^{I}+ f_{B,\tilde{h},\tilde{k}}^{I}  \dt\,.
\end{align}

\noindent In analogy with what we did on the base manifold, we  look for a PDE defined on the whole $X$ and such that on $X_{\frac{\Rep}{B}}$ restricts to the one written above. To this aim we introduce a truncation-extension operator on weighted H\"older spaces

\begin{definition}
Let $f\in C_{\delta}^{0,\alpha} \st X \dt $, we define $\mathcal{E}_{R}:C_{\delta}^{0,\alpha} \st X \dt \rightarrow C_{\delta}^{0,\alpha} \st X \dt $
\begin{displaymath}
\mathcal{E}_{\Rep} \st f \dt :\left\{ \begin{array}{ll}
f \st x \dt  & x\in X_{\frac{\Rep}{B}}   \\
f \st R\frac{x}{|x|} \dt \chi \st \frac{|x|B}{\Rep} \dt  & x\in X_{\frac{2\Rep}{B}}\setminus X_{\frac{\Rep}{B}}\\
0  & x\in X\setminus X_{\frac{2\Rep}{B}}  
\end{array} \right.
\end{displaymath}
with $\chi\in C^{\infty}\st [0,+\infty) \dt$ a  cutoff function that is identically $1$ on $[0,1]$ and identically $0$ on $[2,+\infty)$. 
\end{definition}

\noindent Using the right inverse for $\Delta_{\eta}^{2}$ of Proposition \ref{isomorfismopesati}  
\begin{equation}
\mathbb{J}^{(\delta)}:C_{\delta-4}^{0,\alpha}\st X \dt\rightarrow C_{\delta}^{4,\alpha}\st X \dt
\end{equation}
and truncation-extension operator we define the nonlinear operator
\begin{equation}
\mathcal{T}:C_{\delta}^{4,\alpha}\st X \dt\rightarrow C_{\delta}^{4,\alpha}\st X \dt
\end{equation}

\begin{align}
\mathcal{T}\st f,\tilde{h},\tilde{k} \dt:=&-2\varepsilon^{2}B^{4}\mathbb{J}^{(\delta)}\mathcal{E}_{\Rep}\mathfrak{s}_{\bg,\cg,\hg,\kg}-\frac{1}{\varepsilon^{2}}\mathbb{J}^{(\delta)}\mathcal{E}_{\Rep}\Delta_{\eta}^{2}\st \sum_{k=6}^{+\infty}\chi\Psi_{k}\st \frac{B x}{\varepsilon} \dt\dt-\mathbb{J}^{(\delta)}\mathcal{E}_{\Rep}\Delta_{\eta}^{2}\st \textbf{H}_{\tilde{h},\tilde{k}}^{I}\dt\\
&+2B^{4}\mathbb{J}^{(\delta)}\mathcal{E}_{\Rep}\NN _{B^{2}\eta}\st \frac{1}{\varepsilon^{2}}\chi\psi_{\omega}\st \frac{B x}{\varepsilon} \dt+\frac{  \cga \st m-1 \dt s_{\omega}B^{4}\varepsilon^{2} }{2\st m-2 \dt   m\st m+1 \dt }\chi |x|^{4-2m} +V+\textbf{H}_{\tilde{h},\tilde{k}}^{I}+ f_{B,\tilde{h},\tilde{k}}^{I}  \dt
\end{align}

\noindent We prove the existence of a solution of equation \eqref{eq: PDEmodello} finding a fixed point of the operator $\mathcal{T}$ following exactly the same strategy we used on the base. We need to find a domain $\Omega\subseteq C_{\delta}^{4,\alpha}\st X \dt$ such that $\mathcal{T}\st \Omega \dt$ and $\mathcal{T}$ is contractive. To decide what kind of domain will be our $\Omega$ we need some informations on the behavior of $\mathcal{T}$ that we find in the following two lemmas.

\begin{lemma} 
Let $\delta\in (4-2m,5-2m)$ and $\st \varepsilon^{2}\tilde{h}, \varepsilon^{2}\tilde{k}\dt\in \dombd$, then the following estimate holds
\begin{equation}\left\|\mathcal{E}_{\Rep}\Delta_{\eta}^{2}\hkii\right\|_{\Ccx}\leq  \frac{ C\st \eta \dt }{\Rep^4}\left\|\tilde{h}^{(\dagger)},\tilde{k}^{(\dagger)}\right\|_{\ambbda}\,.\end{equation}
\end{lemma}
\begin{proof}
Using formula \eqref{eq:thki} we have 
\begin{align}
\Delta_{\eta}^{2} \hkii =& \Delta_{\eta}^{2}\st H_{\tilde{h},\tilde{k}}^{I}\st 0\dt+ \chi\st H_{\tilde{h},\tilde{k}}^{I}\st \frac{ B x}{\Rep} \dt- \chi H_{\tilde{h},\tilde{k}}^{I}\st 0\dt\dt\dt\\
&+\Delta_{\eta}^{2}\st \frac{\tilde{k}^{(0)}B^2}{4m \Rep^2}v^{(0)}  + \st \tilde{h}^{(2)}-\frac{\tilde{k}^{(2)}}{4(m+2)} \dt \frac{B^{2}v^{(2)}}{\Rep^2}+ \st \tilde{h}^{(3)}-\frac{\tilde{k}^{(3)}B^{3}}{4(m+3)}  \dt  \frac{u^{(3)}+v^{(3)}}{\Rep^3}\dt\\
=&\sq \Delta_{\eta}^{2}-\Delta^{2} \dq\st \frac{\tilde{k}^{(2)}}{4(m+2)}\chi\left|\frac{B  x}{\Rep}\right|^{4}  \phi_{2}  + \frac{\tilde{k}^{(\gamma)}}{4(m+3)}\chi\left|\frac{B  x}{\Rep}\right|^{5}  \phi_{3}  \dt\\
&+\sq \Delta_{\eta}^{2}-\Delta^{2} \dq\st   \chi\sum_{\gamma=4}^{+\infty}\left(\left(\tilde{h}^{(\gamma)}-\frac{\tilde{k}^{(\gamma)}}{4(m+\gamma)}  \dt \left|\frac{B x}{\Rep}\right|^{\gamma}+\frac{\tilde{k}^{(\gamma)}}{4(m+\gamma)}\left|\frac{B  x}{\Rep}\right|^{\gamma+2} \dt \phi_{\gamma}   \dt
\end{align}
and so we deduce that
\begin{equation}\left\|\Delta_{\eta}^{2} \hkii\right\|_{C^{0,\alpha}\st X_{\frac{R_0}{B}} \dt}\leq C\st \eta\dt \frac{\left\|\tilde{h}^{(\dagger)},\tilde{k}^{(\dagger)}\right\|_{\ambbda}}{\Rep^4}\,.\end{equation}
Now we estimate the quantity
\begin{equation}\sup_{\rho\in [R_0,\Rep]}\rho^{-\delta+4}\left\| \Le \hkii\right\|_{C^{0,\alpha}\st  B_{1}\setminus B_{\frac{1}{2}}  \dt}\,.\end{equation}
Using again formula \eqref{eq:thki}, we have  

\begin{align}
\Delta_{\eta}^{2} \hkii\st\rho w\dt =&  \sq \Delta_{\eta}^{2}-\Delta^{2} \dq\st \frac{\tilde{k}^{(2)}}{4(m+2)}\chi\left|\frac{B  x}{\Rep}\right|^{4}  \phi_{2}  + \frac{\tilde{k}^{(\gamma)}}{4(m+3)}\chi\left|\frac{B  x}{\Rep}\right|^{5}  \phi_{3}  \dt\\
&+\sq \Delta_{\eta}^{2}-\Delta^{2} \dq\st   \chi\sum_{\gamma=4}^{+\infty}\left(\left(\tilde{h}^{(\gamma)}-\frac{\tilde{k}^{(\gamma)}}{4(m+\gamma)}  \dt \left|\frac{B x}{\Rep}\right|^{\gamma}+\frac{\tilde{k}^{(\gamma)}}{4(m+\gamma)}\left|\frac{B  x}{\Rep}\right|^{\gamma+2} \dt \phi_{\gamma}   \dt\\
=& \frac{\left\|\tilde{h}^{(\dagger)},\tilde{k}^{(\dagger)}\right\|}{\Rep^4}\rho^{-2m}\mathcal{O}_{C^{0,\alpha}\st   B_{1}\setminus B_{\frac{1}{2}}  \dt}\st 1+\frac{\rho}{\Rep} \dt\,.
\end{align}

\noindent So we have 

\begin{equation}\sup_{\rho\in [R_0,\Rep]}\rho^{-\delta+4}\left\| \Delta_{\eta}^{2} \hkii  \right\|_{C^{0,\alpha}\st   B_{1}\setminus B_{\frac{1}{2}}   \dt}\leq C\st \eta \dt \frac{\left\|\tilde{h}^{(\dagger)},\tilde{k}^{(\dagger)}\right\|_{\ambbda}}{\Rep^{\delta+2m}}\,.\end{equation}

\end{proof}

\begin{lemma}
Let $\delta\in (4-2m,5-2m)$, then the following estimate holds: 
\begin{equation}\left\| \mathcal{T}\st 0, 0,0 \dt \right\|_{\Ccx}\leq C\st g,\eta \dt\varepsilon^{4}\Rep^{6-2m-\delta}\,.\end{equation}
\end{lemma}
\begin{proof}
We will prove the lemma for the case $m\geq 3$, for the case $m=2$ the proof is identical. For the sake of notation, throughout this proof we will use the following convention
\begin{equation}
\mathfrak{P}:=\frac{1}{\varepsilon^{2}}\chi\psi_{\omega}\st \frac{B x}{\varepsilon} \dt+\frac{  \cga \st m-1 \dt s_{\omega}B^{4}\varepsilon^{2} }{2\st m-2 \dt   m\st m+1 \dt }\chi |x|^{4-2m}+V.
\end{equation} 
By formulas \eqref{eq:scheletromod1} and \eqref{eq:scheletromod2}, on $X_{ \frac{R_{0}}{B}  }$, we have    
\begin{align}
-\varepsilon^{2}B^{4}s_{\omega}-\frac{1}{2}\Delta_{\eta}^{2}\mathfrak{P}+\mathbb{N}_{B^{2}\eta}\st \mathfrak{P} \dt =&-\frac{1}{2}\Delta_{\eta}^{2}\st \sum_{k=6}^{+\infty}\frac{B^{k}}{\varepsilon^{k-2}}\chi\Psi_{k} \dt+\mathbb{N}_{B^{2}\eta}\st \mathfrak{P} \dt\\
=&\varepsilon^{4}\mathcal{O}_{C^{0,\alpha}\st X_{\frac{R_0}{B}} \dt}\st 1\dt
\end{align}
and so 
\begin{equation}\left\|-\varepsilon^{2}B^{4}s_{\omega}-\frac{1}{2}\Delta_{\eta}^{2}\mathfrak{P}+\mathbb{N}_{B^{2}\eta}\st \mathfrak{P} \dt\right\|_{C^{0,\alpha}\st X_{\frac{R_0}{B}} \dt} \leq C\st g, \eta \dt \varepsilon^4\,.\end{equation}

Now we estimate the weighted part of the norm. On $X_{\frac{\Rep}{B}}\setminus X_{\frac{R_0}{ 2 B}}$, using Proposition \ref{proprietacsck}, precisely the algebraic identity
\begin{equation}
-\frac{1}{2}\Delta \psi_{\omega}+\NN_{eucl}\st \psi_{\omega} \dt=s_{\omega}\,,
\end{equation}
we have
\begin{align}
-\varepsilon^{2}B^{4}s_{\omega}-\frac{1}{2}\Delta_{\eta}^{2}\mathfrak{P}+\NN _{B^{2}\eta}\st \mathfrak{P} \dt
=&-\frac{1}{2}\sq \Delta_{\eta}^{2}-\Delta^{2} \dq\st \sum_{k=6}^{+\infty}\frac{B^{k}}{\varepsilon^{k-2}}\chi\Psi_{k} \dt\\
&+\sq \NN_{B^{2}\eta}\st \mathfrak{P} \dt-\mathbb{N}_{B^{2}\eta}\st  \frac{1}{\varepsilon^{2}}\chi\psi_{\omega}\st \frac{B x}{\varepsilon} \dt  \dt\dq\\
&+\sq \NN_{B^{2}\eta}\st  \frac{1}{\varepsilon^{2}}\chi\psi_{\omega}\st \frac{B x}{\varepsilon} \dt  \dt-\NN _{B^{2}\,eucl}\st  \frac{1}{\varepsilon^{2}}\chi\psi_{\omega}\st \frac{B x}{\varepsilon} \dt  \dt\dq
\end{align}
 
And so, on $X_{\frac{\rho}{B}}\setminus X_{\frac{\rho}{ 2 B}}$ we have
\begin{align}		
-\varepsilon^{2}B^{4}s_{\omega}-\frac{1}{2}\Delta_{\eta}^{2}\mathfrak{P}+\NN_{B^{2}\eta}\st \mathfrak{P} \dt &=\varepsilon^4\rho^{2-2m}\mathcal{O}_{C^{0,\alpha}\st   B_{1}\setminus B_{\frac{1}{2}}   \dt}\st 1+\varepsilon\rho\dt 
\end{align}

\noindent and we can conclude 

\begin{equation}\sup_{\rho\in[R_0,\Rep]}\rho^{-\delta+4}\left\|-\varepsilon^{2}B^{4}s_{\omega}-\frac{1}{2}\Delta_{\eta}^{2}\mathfrak{P}+\NN_{B^{2}\eta}\st \mathfrak{P} \dt\right\|_{C^{0,\alpha}\st   B_{1}\setminus B_{\frac{1}{2}}   \dt} \leq C\st g, \eta \dt \varepsilon^4 \Rep^{6-2m-\delta}\,. \end{equation}

\end{proof}

\noindent  

We consider the subset of $C_{\delta}^{4,\alpha}\st X \dt$ with $\delta\in (4-2m,5-2m)$ 
\begin{equation}\left\| f \right\|_{C_{\delta}^{4,\alpha}\st X \dt}\leq 2\left\|\mathbb{J}^{(\delta)}\mathcal{E}_{\Rep}\Delta_{\eta}^{2}\hkii\right\|_{\Ccx}\end{equation}

and we study continuity properties of $\mathcal{T}$ on this domain.

\begin{lemma}
If  $\st\varepsilon^{2}\tilde{h}',\varepsilon^{2}\tilde{k}'\dt\in \dombd$, $f,f'\in  C_{\delta}^{4,\alpha}\st X \dt$ 

\begin{equation}\left\|f\right\|_{C_{\delta}^{4,\alpha}\st X \dt},\left\|f'\right\|_{C_{\delta}^{4,\alpha}\st X \dt}\leq  2\left\|\mathbb{J}^{(\delta)}\mathcal{E}_{\Rep}\Delta_{\eta}^{2}\hkii\right\|_{\Ccx} \end{equation}
and assumptions of Proposition \ref{crucialmodello} are satisfied, then the following estimates hold:

\begin{align}\left\|\mathcal{T}\st f,\tilde{h},\tilde{k} \dt - \mathcal{T}\st 0,0,0 \dt\right\|_{\Ccx}\,\leq&\, \frac{3}{2} \left\|\mathbb{J}^{(\delta)}\mathcal{E}_{\Rep}\Delta_{\eta}^{2}\hkii\right\|_{\Ccx}\\
\left\|\mathcal{T}\st f,\tilde{h},\tilde{k} \dt - \mathcal{T}\st f',\tilde{h},\tilde{k} \dt\right\|_{\Ccx}\,\leq&\, \frac{1}{2} \left\|f-f'\right\|_{C_{\delta}^{4,\alpha}\st X \dt}\\
\left\|\mathcal{T}\st f,\tilde{h},\tilde{k} \dt - \mathcal{T}\st f,\tilde{h}',\tilde{k}' \dt\right\|_{C_{\delta}^{4,\alpha}\st X \dt}\,\leq&\, \frac{1}{2} \left\|\tilde{h}-\tilde{h}',\tilde{k}-\tilde{k}'\right\|_{\ambbda}\,.
\end{align}

\end{lemma}
\begin{proof}
Follows by direct computations as \cite[Lemma 5.3]{ap2}
\end{proof}

\section{Data matching}

Now that we have the families of metrics on the base orbifold and on model spaces we want to glue them. To perform the data matching construction we will rescale all functions involved in such a way that functions on $X$ are  functions on the annulus $\overline{B_{1}}\setminus B_{\frac{1}{2}}$ and functions on $M$ are functions on the annulus $\overline{B_{2}}\setminus B_{1}$. The main technical tool we will use in this section is the ``Dirichet to Neumann'' map for biharmonic extensions that we introduce with the following Theorem whose proof can be found in \cite[Lemma 6.3]{ap1}.

\begin{teo} \label{dirneu}
The map
\begin{equation}\mathcal{P}:C^{4,\alpha} \st \Sp^{2m-1} \dt \times C^{2,\alpha} \st \Sp^{2m-1} \dt \rightarrow C^{3,\alpha} \st \Sp^{2m-1} \dt \times C^{1,\alpha} \st \Sp^{2m-1} \dt \end{equation} 
\begin{equation}\mathcal{P} \st h,k \dt = \st \partial_{|w|} \st H_{h,k}^{o}-H_{h,k}^{I} \dt ,\partial_{|w|}\Delta \st H_{h,k}^{o}-H_{h,k}^{I} \dt  \dt \end{equation}
is an isomorphism of Banach spaces with inverse $\mathcal{Q}$.
\end{teo}

\begin{proof}[Proof of Theorem \ref{belliebrutti}]: 
we give details for the case $m\geq 3$ and for $m=2$ the proof is exactly the same.   We take the families of metrics constructed in Proposition \ref{famigliabaseap2} (\ref{sipuntibuoni}) and Proposition \ref{famigliamodelloap2}.  We consider K\"ahler potentials $\mathcal{V}^{o}_{j,\ag,\bf{0},\cg,\hg,\kg}$,$\mathcal{V}^{o}_{l,\ag,\bf{0},\cg,\hg,\kg}$  of $\omega_{\ag,\bf{0},\cg,\hg,\kg}$ at the annulus $\overline{B_{2\rep}\st p_{j} \dt}\setminus B_{\rep}\st p_{j} \dt$ respectively $\overline{B_{2\rep}\st q_{l} \dt}\setminus B_{\rep}\st q_{l} \dt$ under the homothety
\begin{equation}
z=\rep w\,.
\end{equation} 

We consider the  K\"ahler potential $\mathcal{V}^{I}_{j,\mathfrak{b}_{j},\tilde{h}_{j},\tilde{k}_{j}}$  of $\varepsilon^{2}\mathfrak{b}^{2}\eta_{j,\mathfrak{b}_{j},\tilde{h}_{j},\tilde{k}_{j}}$ (with $\mathfrak{b}_{j}\in \RR^{+}$ to be determined later) at the annulus $\overline{X_{\Gamma_{j},\frac{\Rep}{\mathfrak{b}_{j}}}}\setminus X_{\Gamma_{j},\frac{\Rep}{2\mathfrak{b}_{j}}}$ under the homothety

\begin{equation}
x=\frac{\Rep w}{\mathfrak{b}_{j}}\,.
\end{equation} 
and the  K\"ahler potential $\mathcal{V}^{I}_{l,\mathfrak{a}_{l},\tilde{h}_{N+l},\tilde{k}_{N+l}}$  of $\varepsilon^{2}\mathfrak{a}^{2}\theta_{l,\mathfrak{a}_{l},\tilde{h}_{N+l},\tilde{k}_{N+l}}$ (with $\mathfrak{a}_{l}\in \RR^{+}$ to be determined later) at the annulus $\overline{Y_{\Gamma_{N+l},\frac{\Rep}{\mathfrak{a}_{l}}}}\setminus Y_{\Gamma_{N+l},\frac{\Rep}{2\mathfrak{a}_{l}}}$ under the homothety
\begin{equation}
x=\frac{\Rep w}{\mathfrak{a}_{l}}\,.
\end{equation}

For the sake of clearness we set
\begin{equation}
A_{l}:=\frac{|\Gamma_{N+l}|a_{l}}{8\st m-2 \dt\st m-1 \dt|\Sp^{2m-1}|}
\end{equation}
and we point out in table \ref{tab2bis} the various  growth in $\varepsilon$ of various expansions.  Since $\varepsilon^{2m-2}\rep^{4-2m}=\rep^{\gamma}$ and $\gamma<4$ and the size of $\st \hg,\kg \dt$ is $\rep^{4}$ one cannot recover terms of size $\varepsilon^{2m-2}\rep^{4-2m}$ as boundary conditions for the biharmonic extensions, hence we are forced to define   
\begin{align}
\egal \mathfrak{a}_{l}^{2m-2}\varepsilon^{2}\Rep^{4-2m}|w|^{2-2m}=&  \frac{\egal A_{l}\varepsilon^{2m-2}\rep^{4-2m}}{|\egal|}|w|^{4-2m}\\
&{+\frac{k_{N+l}^{(0)}}{4m-8}|w|^{4-2m}}\\
&{  +   \frac{\egal\hat{f}_{l,\ag,\bf{0},\cg,\hg,\kg}^{o}A_{l}\rep^{4-2m}}{|\egal|}    |w|^{4-2m} }
\end{align}

(where $ \hat{f}_{l,\ag,\bf{0},\cg,\hg,\kg}^{o}$ is the coefficient of the leading asymptotic of the expansion of $ \hat{f}_{\ag,\bf{0},\cg,\hg,\kg}^{o}$ near the point $q_l$)  that implies 
\begin{equation}\label{eq:sceltaa}
\mathfrak{a}_{l}^{2m-2}=\frac{A_{l}\st 1 + \hat{f}_{l,\ag,\bf{0},\cg,\hg,\kg}^{o}\varepsilon^{2-2m} \dt}{|\egal|}+ \frac{k_{N+l}^{(0)}}{4\st m-2\dt \egal\varepsilon^{2}\Rep^{4-2m}} 
\end{equation}
and moreover we have to define
\begin{equation}\label{eq:sceltac}
c_{j}=\frac{2\st m-1 \dt|\Gamma_{j}|\rep^{2m-4}}{|\Sp^{2m-1}|\varepsilon^{2m-2}}k_{j}^{(0)}\,.
\end{equation} 
Note that with this choice we regained the mean of the functions $\kg$.  After this choice of $\mathfrak{a}_{l}$'s and $c_{j}$'s we can summarize the orders of dependence on $\varepsilon$ of various terms on table \ref{tab2bis}. 

The data matching procedure works exactly as performed in \cite{ap1} and \cite{ap2} and the proof of Theorem \ref{belliebrutti} is complete.  
\end{proof}

\noindent If there are no points $\q$  the situation is more delicate, indeed if we try to use Propositions \ref{famigliabaseap2}, Proposition \ref{famigliamodelloap2} and then  apply the matching procedure as \cite{ap2} we can't regain the degree of freedom lost taking  $\kg^{(\dagger)}$ (functions on the unit sphere with zero mean) instead of  $\kg$. This difficulty occurs because at  points $\p$ there is no asymptotic $|z|^{4-2m}$ to perturb as we did above and the final argument of matching doesn't work.  If we don't require the vanishing of means for the boundary conditions then the estimate \eqref{eq:errore} of Proposition \ref{famigliabaseap2} worsen and become
\begin{equation}\label{eq:errore2}
\left\|f_{\bg,\cg,\hg,\kg}^{o}\right\|_{C_{\delta}^{4,\alpha}\st M_{\p } \dt}\leq C\st  \kappa\dt \rep^{2m}\,.
\end{equation}
This translates to the fact that biharmonic extensions don't control all the terms that aren't matched perfectly (e.g. last row in table \ref{tab2}) because of the dependence on $\kappa$ of constant of estimate \eqref{eq:errore2} and again the final argument of matching doesn't work. Our refined construction for families of metrics on the base manifold (Proposition \ref{crucialbase}) and on model spaces (Proposition \ref{crucialmodello}) let us to overcome these difficulties and now we ilustrate the matching procedure in the case $\q=\emptyset$ in order to complete the proof of Theorem \ref{maintheorem}.

\newpage
\thispagestyle{empty}
\newgeometry{left=1cm,right=1cm,bottom=1cm, top=1cm}

\landscape

\begin{table}
\centering
\caption{\textbf{Sizes in $\varepsilon$ before ``by hand'' identifications}}\label{tab2bis}
\begin{tabular}{p{1.6	cm}||p{5cm}|p{5cm}||p{7.8cm}|p{5cm}}
\hline

{\scriptsize}&{\scriptsize}&{\scriptsize}&{\scriptsize}&{\scriptsize}\\

{\scriptsize {Size in $\varepsilon$ } }
&
{\scriptsize  $\mathcal{V}^{o}_{j,\ag,\bg,\cg,\hg,\kg}\st \rep w \dt$    }
&
{\scriptsize  $\mathcal{V}^{I}_{j,\mathfrak{b}_{j},\tilde{h}_{j},\tilde{k}_{j}}\st \frac{\Rep w}{\mathfrak{b}_{j}} \dt$   }
&
{\scriptsize    $\mathcal{V}^{o}_{l,\ag,\bf{0},\cg,\hg,\kg}\st \rep w \dt$   }
&
{\scriptsize  $\mathcal{V}^{I}_{l,\mathfrak{a}_{l},\tilde{h}_{N+l},\tilde{k}_{N+l}}\st \frac{\Rep w}{\mathfrak{a}_{l}} \dt$  } \\ 

\hline
\hline
{\scriptsize}&{\scriptsize}&{\scriptsize}&{\scriptsize}&{\scriptsize}\\

{\scriptsize \textbf{ $\mathcal{O}\st \rep^{2} \dt$} }&{\scriptsize  $\frac{\rep^{2}|w|^{2}}{2}$ }&{\scriptsize  $\frac{\varepsilon^{2}\Rep^{2}|w|^{2}}{2}$ }&{\scriptsize $\frac{\rep^{2}|w|^{2}}{2}$ }&{\scriptsize $\frac{\varepsilon^{2}\Rep^{2}|w|^{2}}{2}$}\\

{\scriptsize}&{\scriptsize}&{\scriptsize}&{\scriptsize}&{\scriptsize}\\
\hline
{\scriptsize}&{\scriptsize}&{\scriptsize}&{\scriptsize}&{\scriptsize}\\

{\scriptsize  \textbf{  $\mathcal{O}\st\frac{\varepsilon^{2m-2}}{\rep^{2m-4}}\dt$ }}
&
{\scriptsize  $\frac{c_{j}\varepsilon^{2m-2}}{\rep^{2m-4}}|w|^{4-2m}$ }
&{\scriptsize}
&
{\scriptsize  $\frac{\egal A_{l}\varepsilon^{2m}\rep^{4-2m}}{|\egal|}|w|^{4-2m}$ $+\frac{\egal\hat{f}_{l,\ag,\bf{0},\cg,\hg,\kg}^{o}A_{l}\rep^{4-2m}}{|\egal|}    |w|^{4-2m}$  }  
&
{\scriptsize $\frac{\egal \mathfrak{a}_{l}^{2m-2}\varepsilon^{2}}{\Rep^{2m-4}}|w|^{2-2m}$}\\

{\scriptsize}&{\scriptsize}&{\scriptsize}&{\scriptsize}&{\scriptsize}\\
\hline
{\scriptsize}&{\scriptsize}&{\scriptsize}&{\scriptsize}&{\scriptsize}\\

{\scriptsize \textbf{$\mathcal{O}\st \rep^{4} \dt$}  }
&
{\scriptsize  $H_{h_{j},k_{j}}^{o}\st w \dt$       $+\frac{k_{j}^{(0)}}{4m-8}|w|^{4-2m} $ $+\psi_{\omega}\st \rep w \dt$        }
&
{\scriptsize  $\varepsilon^{2}H_{\tilde{h}_{j},\tilde{k}_{j}}^{I}\st w \dt$}
&
{\scriptsize $H_{h_{N+l},k_{N+l}}^{o}\st w \dt$ $+\frac{k_{N+l}^{(0)}}{4m-8}|w|^{4-2m} $ $+\psi_{\omega}\st \rep w \dt$ }&{\scriptsize $\varepsilon^{2}H_{\tilde{h}_{N+l},\tilde{k}_{N+l}}^{I}\st w \dt$ }\\

{\scriptsize}&{\scriptsize}&{\scriptsize}&{\scriptsize}&{\scriptsize}\\
\hline
{\scriptsize}&{\scriptsize}&{\scriptsize}&{\scriptsize}&{\scriptsize}\\

{\scriptsize \textbf{$o\st\rep^{4}\dt$ }  }
&
{\scriptsize    $\sq\varepsilon^{2m-2}{\bf{G}}_{\ag,\bf{0},\cg}\st \rep w \dt+\frac{k_{j}^{(0)}}{4m-8}|w|^{4-2m} \dq$ 	$ + f_{\ag,\bf{0},\cg,\hg,\kg}^{o}$}
&
{\scriptsize  $ - \frac{\cgaj  \mathfrak{b}_{j}^{2m}\varepsilon^{2}}{\Rep^{2m-2}}|w|^{2-2m}$$+\mathfrak{b}_{j}^{2}\varepsilon^{2}\psi_{\eta_{j}}\st \frac{\Rep w}{\mathfrak{b}_{j}} \dt$  $+\varepsilon^{2} f_{j,\mathfrak{b}_{j},\tilde{h}_{j},\tilde{k}_{j}}^{I}\st \frac{\Rep w}{\mathfrak{b}_{j}} \dt$  }
&
{\scriptsize   $\sq \varepsilon^{2m-2}{\bf{G}}_{\ag,\bf{0},\cg}\st \rep w \dt - \frac{\egal A_{l}\varepsilon^{2m}\rep^{4-2m}}{|\egal|}|w|^{4-2m} \dq$                        $+\sq f_{\ag,\bf{0},\cg,\hg,\kg}^{o}-\frac{\egal\hat{f}_{l,\ag,\bf{0},\cg,\hg,\kg}^{o}A_{l}\rep^{4-2m}}{|\egal|}    |w|^{4-2m} \dq$                    }
&
{\scriptsize  $\frac{\cgal  \mathfrak{a}_{l}^{2m}\varepsilon^{2}}{\Rep^{2m-2}}|w|^{2-2m}$$+\mathfrak{a}_{l}^{2}\varepsilon^{2}\psi_{\theta_{l}}\st \frac{\Rep w}{\mathfrak{a}_{l}} \dt$ $+\varepsilon^{2} f_{l,\mathfrak{a}_{l},\tilde{h}_{N+l},\tilde{k}_{N+l}}^{I}\st \frac{\Rep w}{\mathfrak{a}_{l}} \dt$ }\\

{\scriptsize}&{\scriptsize}&{\scriptsize}&{\scriptsize}&{\scriptsize}\\

\hline 
\end{tabular}
\end{table}

\begin{table}
\centering
\caption{\textbf{Sizes in $\varepsilon$ after ``by hand'' identifications}}\label{tab2}
\begin{tabular}{p{1.6	cm}||p{5cm}|p{5cm}||p{7.8cm}|p{5cm}} 
\hline

{\scriptsize}&{\scriptsize}&{\scriptsize}&{\scriptsize}&{\scriptsize}\\

{\scriptsize  {Size in $\varepsilon$ } }
&
{\scriptsize  $\mathcal{V}^{o}_{j,\ag,\bg,\cg,\hg,\kg}\st \rep w \dt$    }
	&
{\scriptsize  $\mathcal{V}^{I}_{j,\mathfrak{b}_{j},\tilde{h}_{j},\tilde{k}_{j}}\st \frac{\Rep w}{\mathfrak{b}_{j}} \dt$   }
&
{\scriptsize    $\mathcal{V}^{o}_{l,\ag,\bf{0},\cg,\hg,\kg}\st \rep w \dt$   }
&
{\scriptsize  $\mathcal{V}^{I}_{l,\mathfrak{a}_{l},\tilde{h}_{N+l},\tilde{k}_{N+l}}\st \frac{\Rep w}{\mathfrak{a}_{l}} \dt$  } \\

\hline
\hline
{\scriptsize}&{\scriptsize}&{\scriptsize}&{\scriptsize}&{\scriptsize}\\

{\scriptsize \textbf{ $\mathcal{O}\st \rep^{2} \dt$} }
&
{\scriptsize  $\frac{\rep^{2}|w|^{2}}{2}$ }
&
{\scriptsize  $\frac{\varepsilon^{2}\Rep^{2}|w|^{2}}{2}$ }
&
{\scriptsize $\frac{\rep^{2}|w|^{2}}{2}$ }
&
{\scriptsize $\frac{\varepsilon^{2}\Rep^{2}|w|^{2}}{2}$}\\

{\scriptsize}&{\scriptsize}&{\scriptsize}&{\scriptsize}&{\scriptsize}\\
\hline
{\scriptsize}&{\scriptsize}&{\scriptsize}&{\scriptsize}&{\scriptsize}\\

{\scriptsize  \textbf{  $\mathcal{O}\st\frac{\varepsilon^{2m-2}}{\rep^{2m-4}}\dt$ }}
&
{\scriptsize }
&
{\scriptsize}
&
{\scriptsize $\frac{\egal A_{l}\varepsilon^{2m}\rep^{4-2m}}{|\egal|}|w|^{4-2m}$$+\frac{k_{N+l}^{(0)}}{4m-8}|w|^{4-2m}$  $+\frac{\egal\hat{f}_{l,\ag,\bf{0},\cg,\hg,\kg}^{o}A_{l}\rep^{4-2m}}{|\egal|}    |w|^{4-2m}$  }
&
{\scriptsize $\frac{\egal \mathfrak{a}_{l}^{2m-2}\varepsilon^{2}}{\Rep^{2m-4}}|w|^{2-2m}$}\\

{\scriptsize}&{\scriptsize}&{\scriptsize}&{\scriptsize}&{\scriptsize}\\
\hline
{\scriptsize}&{\scriptsize}&{\scriptsize}&{\scriptsize}&{\scriptsize}\\

{\scriptsize \textbf{$\mathcal{O}\st \rep^{4} \dt$}  }
&
{\scriptsize  $H_{h_{j},k_{j}}^{o}\st w \dt$ $+\psi_{\omega}\st \rep w \dt$ }
&
{\scriptsize  $\varepsilon^{2}H_{\tilde{h}_{j},\tilde{k}_{j}}^{I}\st w \dt$}
&
{\scriptsize $H_{h_{N+l},k_{N+l}}^{o}\st w \dt$ $+\psi_{\omega}\st \rep w \dt$ }
&
{\scriptsize $\varepsilon^{2}H_{\tilde{h}_{N+l},\tilde{k}_{N+l}}^{I}\st w \dt$ }\\

{\scriptsize}&{\scriptsize}&{\scriptsize}&{\scriptsize}&{\scriptsize}\\
\hline
{\scriptsize}&{\scriptsize}&{\scriptsize}&{\scriptsize}&{\scriptsize}\\

{\scriptsize \textbf{$o\st\rep^{4}\dt$ }  }
&
{\scriptsize    $\sq\varepsilon^{2m-2}{\bf{G}}_{\ag,\bf{0},\cg}\st \rep w \dt+\frac{k_{j}^{(0)}}{4m-8}|w|^{4-2m} \dq$ 	$ + f_{\ag,\bf{0},\cg,\hg,\kg}^{o}$}
&
{\scriptsize  $ - \frac{\cgaj  \mathfrak{b}_{j}^{2m}\varepsilon^{2}}{\Rep^{2m-2}}|w|^{2-2m}$$+\mathfrak{b}_{j}^{2}\varepsilon^{2}\psi_{\eta_{j}}\st \frac{\Rep w}{\mathfrak{b}_{j}} \dt$  $+\varepsilon^{2} f_{j,\mathfrak{b}_{j},\tilde{h}_{j},\tilde{k}_{j}}^{I}\st \frac{\Rep w}{\mathfrak{b}_{j}} \dt$  }
&
{\scriptsize   $\sq \varepsilon^{2m-2}{\bf{G}}_{\ag,\bf{0},\cg}\st \rep w \dt - \frac{\egal A_{l}\varepsilon^{2m}\rep^{4-2m}}{|\egal|}|w|^{4-2m} \dq$                        $+\sq f_{\ag,\bf{0},\cg,\hg,\kg}^{o}-\frac{\egal\hat{f}_{l,\ag,\bf{0},\cg,\hg,\kg}^{o}A_{l}\rep^{4-2m}}{|\egal|}    |w|^{4-2m} \dq$                    }
&
{\scriptsize  $\frac{\cgal  \mathfrak{a}_{l}^{2m}\varepsilon^{2}}{\Rep^{2m-2}}|w|^{2-2m}$$+\mathfrak{a}_{l}^{2}\varepsilon^{2}\psi_{\theta_{l}}\st \frac{\Rep w}{\mathfrak{a}_{l}} \dt$ $+\varepsilon^{2} f_{l,\mathfrak{a}_{l},\tilde{h}_{N+l},\tilde{k}_{N+l}}^{I}\st \frac{\Rep w}{\mathfrak{a}_{l}} \dt$ }\\

{\scriptsize}&{\scriptsize}&{\scriptsize}&{\scriptsize}&{\scriptsize}\\

\hline 
\end{tabular}
\end{table}

\endlandscape

\restoregeometry

\begin{proof}[Proof of Theorem \ref{maintheorem}]:  We consider the radial and non radial part of K\"ahler potential $\mathcal{V}^{o}_{j,\bg,\cg,\hg,\kg}$  of $\omega_{\bg,\cg,\hg,\kg}$ at the annulus $\overline{B_{2\rep}\st p_{j} \dt}\setminus B_{\rep}\st p_{j} \dt$ under the homothety
\begin{equation}
z=\rep w\,.
\end{equation}

We consider also the radial and non radial part of K\"ahler potential $\mathcal{V}^{I}_{j,\mathfrak{b}_{j},\tilde{h}_{j},\tilde{k}_{j}}$  of $\varepsilon^{2}\mathfrak{b}^{2}\eta_{j,\mathfrak{b}_{j},\tilde{h}_{j},\tilde{k}_{j}}$ at the annulus $\overline{X_{j,\frac{\Rep}{\mathfrak{b}_{j}}}}\setminus X_{j,\frac{\Rep}{2\mathfrak{b}_{j}}}$ under the homothety
\begin{equation}
x=\frac{\Rep w}{\mathfrak{b}_{j}}\,.
\end{equation}
As in the previous case we cannot recover the terms in the second row of table \ref{tab1bis} as boundary conditions for the biharmonic extensions and we are forced to have
\begin{align}
\cgaj  \mathfrak{b}_{j}^{2m}\varepsilon^{2}\Rep^{2-2m}|w|^{2-2m}=&\cgaj  B_{j}^{2m}\varepsilon^{2m}\rep^{2-2m}|w|^{2-2m}\\
&+\st h_{j}^{(0)}+\frac{k_{j}^{(0)}}{4m-8}\dt|w|^{2-2m}\\
&-\hat{f}_{j,\bg,\cg,\hg,\kg}^{o}\cgaj B_{j}^{2m} \rep^{2-2m}|w|^{2-2m}\,,
\end{align}

\begin{align}
\frac{  \cga \st m-1 \dt s_{\omega}\mathfrak{b}^{2m}\varepsilon^{4}\Rep^{4-2m} }{2\st m-2 \dt   m\st m+1 \dt } |w|^{4-2m}=& -  C_{j}\varepsilon^{2m}\rep^{4-2m}|w|^{4-2m}\\
&-\frac{k_{j}^{(0)}}{4m-8}|w|^{4-2m}\\
&+\hat{f}_{j,\bg,\cg,\hg,\kg}^{o}C_{j}\rep^{4-2m}|w|^{4-2m}\,,
\end{align}



where $\hat{f}_{j,\bg,\cg,\hg,\kg}^{o}$ is the coefficient of the leading asymptotic of the expansion of $\hat{f}_{\bg,\cg,\hg,\kg}^{o}$near the point $p_j$. We recall that coefficients $B_{j}$ and $C_{j}$ are defined in Section \ref{nonlinbase} respectively by equations \eqref{eq:Bj} and \eqref{eq:Cj}. Conditions above force us to set: 

\begin{equation}\label{eq:sceltaB}
\mathfrak{b}_{j}^{2m}=B_{j}^{2m}\st 1-	\frac{\hat{f}_{j,\bg,\cg,\hg,\kg}^{o} }{  \varepsilon^{2m}} \dt+\frac{1}{ \cgaj  }\st h_{j}^{(0)}+\frac{k_{j}^{(0)}}{4m-8}\dt\frac{\rep^{2m-2}}{\varepsilon^{2m}}
\end{equation}


\begin{align}\label{eq:sceltaC}
C_{j}=&-\frac{1}{2\st m-2 \dt\st  \varepsilon^{2m}-\hat{f}_{j,\bg,\cg,\hg,\kg}^{o} \dt}\st \frac{  \cga \st m-1 \dt s_{\omega}\mathfrak{b}^{2m}\varepsilon^{4}\Rep^{4-2m} }{  m\st m+1 \dt }+k_{j}^{(0)} \dt
\end{align}

\begin{equation}\label{eq:sceltagamma}
c_{j}=s_{\omega}b_{j}
\end{equation}

\noindent This explains the importance of assumption \eqref{eq:tuning} in Theorem \ref{maintheorem}.

\newpage
\thispagestyle{empty}
\newgeometry{left=2.5cm,right=1cm,bottom=1cm, top=1cm}

\landscape

\begin{table}
\centering
\caption{\textbf{Sizes in $\varepsilon$ before ``by hand'' identifications}}\label{tab1bis}
\begin{tabular}{p{4cm}|| p{6.3cm}| p{4.5cm}|| p{4.2cm}| p{4.5cm} } 
\hline

{\scriptsize}&{\scriptsize}&{\scriptsize}&{\scriptsize}&{\scriptsize}\\

{\scriptsize {Size in $\varepsilon$ } }
&
{\scriptsize   $\sq\mathcal{V}^{o}_{j,\bg,\cg,\hg,\kg}\st \rep w \dt\dq^{(0)}$   }
&
{\scriptsize   $\sq\mathcal{V}^{I}_{j,\mathfrak{b}_{j},\tilde{h}_{j},\tilde{k}_{j}}\st \frac{\Rep w}{\mathfrak{b}_{j}} \dt\dq^{(0)}$  }
&
{\scriptsize   $\sq\mathcal{V}^{o}_{j,\bg,\cg,\hg,\kg}\st \rep w \dt\dq^{(\dagger)}$    }
&
{\scriptsize   $\sq\mathcal{V}^{I}_{j,\mathfrak{b}_{j},\tilde{h}_{j},\tilde{k}_{j}}\st \frac{\Rep w}{\mathfrak{b}_{j}} \dt\dq^{(\dagger)}$  } \\ 

{\scriptsize}&{\scriptsize}&{\scriptsize}&{\scriptsize}&{\scriptsize}\\
\hline
\hline
{\scriptsize}&{\scriptsize}&{\scriptsize}&{\scriptsize}&{\scriptsize}\\

{\scriptsize \textbf{$o\st \rep^{2} \dt$ for first two columns, $o\st \rep^{4} \dt$ for third and fourth columns} }
&
{\scriptsize ${   \frac{\rep^{2}|w|^{2}}{2}+ \psi_{\omega}^{(0)}\st \rep w \dt}$ }
&
{\scriptsize ${  \frac{\varepsilon^{2}\Rep^{2}|w|^{2}}{2}+ \psi_{\omega}^{(0)}\st \varepsilon^{2}\Rep w \dt}$ }
&
{\scriptsize ${  \psi_{\omega}^{(\dagger)}\st \rep w \dt}$}
&
{\scriptsize ${  \psi_{\omega}^{(\dagger)}\st \varepsilon^{2}\Rep w \dt}$}\\

{\scriptsize}&{\scriptsize}&{\scriptsize}&{\scriptsize}&{\scriptsize}\\
\hline
{\scriptsize}&{\scriptsize}&{\scriptsize}&{\scriptsize}&{\scriptsize}\\

{\scriptsize  \textbf{$\mathcal{O}\st\varepsilon^{2m}\rep^{2-2m}\dt$}}
&
{\scriptsize    $-\cgaj  B_{j}^{2m}\varepsilon^{2m}\rep^{2-2m}|w|^{2-2m}$   $+\hat{f}_{j,\bg,\cg,\hg,\kg}^{o}\cgaj B_{j}^{2m} \rep^{2-2m}|w|^{2-2m}$   $- C_{j}\varepsilon^{2m}\rep^{4-2m}|w|^{4-2m}$       $+\hat{f}_{j,\bg,\cg,\hg,\kg}^{o}C_{j}\rep^{4-2m}|w|^{4-2m}$}
&
{\scriptsize {  $- \cgaj  \mathfrak{b}_{j}^{2m}\varepsilon^{2}\Rep^{2-2m}|w|^{2-2m}$ $+\frac{  \cga \st m-1 \dt s_{\omega}\mathfrak{b}^{2m}\varepsilon^{4}\Rep^{4-2m} }{2\st m-2 \dt   m\st m+1 \dt } |w|^{4-2m}$} }&{\scriptsize  }&{\scriptsize } \\

{\scriptsize}&{\scriptsize}&{\scriptsize}&{\scriptsize}&{\scriptsize}\\
\hline
{\scriptsize}&{\scriptsize}&{\scriptsize}&{\scriptsize}&{\scriptsize}\\

{\scriptsize \textbf{$\mathcal{O}\st \rep^{\beta} \dt$ for first two columns, $\mathcal{O}\st \rep^{\sigma} \dt$ for third and fourth columns}  }
&
{\scriptsize {   } }
&
{\scriptsize ${  \varepsilon^{2}H_{\tilde{h}_{j}^{(0)},\tilde{k}_{j}^{(0)}}^{I}\st w \dt}$ }
&
{\scriptsize ${  H_{h_{j}^{(\dagger)},k_{j}^{(\dagger)}}^{o}\st w \dt}$ }&{\scriptsize ${  \varepsilon^{2}H_{\tilde{h}_{j}^{(\dagger)},\tilde{k}_{j}^{(\dagger)}}^{I}\st w \dt}$}\\

{\scriptsize}&{\scriptsize}&{\scriptsize}&{\scriptsize}&{\scriptsize}\\
\hline
{\scriptsize}&{\scriptsize}&{\scriptsize}&{\scriptsize}&{\scriptsize}\\

{\scriptsize \textbf{$o\st \rep^{\beta} \dt$ for first two columns, $o\st \rep^{\sigma} \dt$ for third and fourth columns}    }
&
{\scriptsize $B_{j}^{2}\varepsilon^{2}\psi_{\eta_{j}}^{(0)}\st \frac{\rep w}{B_{j}	\varepsilon} \dt$$-\sq \varepsilon^{2m}{\bf{G}}_{\bg,\cg}\st \rep w \dt\right.$ $- \cgaj  B_{j}^{2m}\varepsilon^{2m}\rep^{2-2m}|w|^{2-2m}$    $\left.- \cgaj  C_{j}\varepsilon^{2m}\rep^{4-2m}|w|^{4-2m} \dq^{(0)}$
$+\sq f_{\bg,\cg,\hg,\kg}^{o}- \hat{f}_{j,\bg,\cg,\hg,\kg}^{o}\cgaj B_{j}^{2m} \rep^{2-2m}|w|^{2-2m}\right.$ $\left.-\hat{f}_{j,\bg,\cg,\hg,\kg}^{o}C_{j}\rep^{4-2m}|w|^{4-2m}\dq^{(0)}$ }
&
{\scriptsize $\mathfrak{b}_{j}^{2}\varepsilon^{2}\psi_{\eta_{j}}^{(0)}\st \frac{\Rep w}{\mathfrak{b}_{j}} \dt$

$+\varepsilon^{2}V^{(0)}\st \frac{\Rep w}{\mathfrak{b}_{j}} \dt$
$+\varepsilon^{2}\sq \hkii\st \frac{\Rep w}{\mathfrak{b}_{j}} \dt-   H_{\tilde{h}_{j},\tilde{k}_{j}}^{I}\st w \dt\dq^{(0)}$
$+\sq  \varepsilon^{2}f_{j,\mathfrak{b}_{j},\tilde{h}_{j},\tilde{k}_{j}}^{I}\st \frac{\Rep w}{\mathfrak{b}_{j}} \dt \dq^{(0)}$}&{\scriptsize $B_{j}^{2}\varepsilon^{2}\psi_{\eta_{j}}^{(\dagger)}\st \frac{\rep w}{B_{j}	\varepsilon} \dt$
$-\sq \varepsilon^{2m}{\bf{G}}_{\bg,\cg}\st \rep w \dt  \dq^{(\dagger)}$
$+\sq f_{\bg,\cg,\hg,\kg}^{o}\dq^{(\dagger)}$}&{\scriptsize $\mathfrak{b}_{j}^{2}\varepsilon^{2}\psi_{\eta_{j}}^{(\dagger)}\st \frac{\Rep w}{\mathfrak{b}_{j}} \dt$
                                               
$+\varepsilon^{2}V^{(\dagger)}\st \frac{\Rep w}{\mathfrak{b}_{j}} \dt$
$+\varepsilon^{2}\sq \hkii\st \frac{\Rep w}{\mathfrak{b}_{j}} \dt-   H_{\tilde{h}_{j},\tilde{k}_{j}}^{I}\st w \dt\dq^{(\dagger)}$
$+\sq \varepsilon^{2}f_{j,\mathfrak{b}_{j},\tilde{h}_{j},\tilde{k}_{j}}^{I}\st \frac{\Rep w}{\mathfrak{b}_{j}} \dt \dq^{(\dagger)}$ }\\

{\scriptsize}&{\scriptsize}&{\scriptsize}&{\scriptsize}&{\scriptsize}\\

\hline 
\end{tabular}
\end{table}
\endlandscape

\restoregeometry



\newpage
\thispagestyle{empty}
\newgeometry{left=2.5cm,right=1cm,bottom=1cm, top=1cm}

\landscape

\begin{table}
\centering
\caption{\textbf{Sizes in $\varepsilon$ after ``by hand'' identifications}}\label{tab1}
\begin{tabular}{p{4cm}|| p{6.3cm}| p{4.5cm}|| p{4.2cm}| p{4.5cm} } 
\hline

{\scriptsize}&{\scriptsize}&{\scriptsize}&{\scriptsize}&{\scriptsize}\\

{\scriptsize {Size in $\varepsilon$ } }
&
{\scriptsize   $\sq\mathcal{V}^{o}_{j,\bg,\cg,\hg,\kg}\st \rep w \dt\dq^{(0)}$   }
&
{\scriptsize   $\sq\mathcal{V}^{I}_{j,\mathfrak{b}_{j},\tilde{h}_{j},\tilde{k}_{j}}\st \frac{\Rep w}{\mathfrak{b}_{j}} \dt\dq^{(0)}$  }
&
{\scriptsize   $\sq\mathcal{V}^{o}_{j,\bg,\cg,\hg,\kg}\st \rep w \dt\dq^{(\dagger)}$    }
&
{\scriptsize   $\sq\mathcal{V}^{I}_{j,\mathfrak{b}_{j},\tilde{h}_{j},\tilde{k}_{j}}\st \frac{\Rep w}{\mathfrak{b}_{j}} \dt\dq^{(\dagger)}$  } \\ 

{\scriptsize}&{\scriptsize}&{\scriptsize}&{\scriptsize}&{\scriptsize}\\
\hline
\hline
{\scriptsize}&{\scriptsize}&{\scriptsize}&{\scriptsize}&{\scriptsize}\\

{\scriptsize \textbf{$o\st \rep^{2} \dt$ for first two columns, $o\st \rep^{4} \dt$ for third and fourth columns} }
&
{\scriptsize ${   \frac{\rep^{2}|w|^{2}}{2}+ \psi_{\omega}^{(0)}\st \rep w \dt}$ }&{\scriptsize ${  \frac{\varepsilon^{2}\Rep^{2}|w|^{2}}{2}+ \psi_{\omega}^{(0)}\st \varepsilon^{2}\Rep w \dt}$ }
&
{\scriptsize ${  \psi_{\omega}^{(\dagger)}\st \rep w \dt}$}&{\scriptsize ${  \psi_{\omega}^{(\dagger)}\st \varepsilon^{2}\Rep w \dt}$}\\

{\scriptsize}&{\scriptsize}&{\scriptsize}&{\scriptsize}&{\scriptsize}\\
\hline
{\scriptsize}&{\scriptsize}&{\scriptsize}&{\scriptsize}&{\scriptsize}\\

{\scriptsize  \textbf{$\mathcal{O}\st\varepsilon^{2m}\rep^{2-2m}\dt$}}
&
{\scriptsize    $-\cgaj  B_{j}^{2m}\varepsilon^{2m}\rep^{2-2m}|w|^{2-2m}$   $-\st h_{j}^{(0)}+\frac{k_{j}^{(0)}}{4m-8}\dt|w|^{2-2m}$     $+\hat{f}_{j,\bg,\cg,\hg,\kg}^{o}\cgaj B_{j}^{2m} \rep^{2-2m}|w|^{2-2m}$       $- C_{j}\varepsilon^{2m}\rep^{4-2m}|w|^{4-2m}$$-\frac{k_{j}^{(0)}}{4m-8}|w|^{4-2m}$         $+\hat{f}_{j,\bg,\cg,\hg,\kg}^{o}C_{j}\rep^{4-2m}|w|^{4-2m}$
  }
&
{\scriptsize {  $- \cgaj  \mathfrak{b}_{j}^{2m}\varepsilon^{2}\Rep^{2-2m}|w|^{2-2m}$ $+\frac{  \cga \st m-1 \dt s_{\omega}\mathfrak{b}^{2m}\varepsilon^{4}\Rep^{4-2m} }{2\st m-2 \dt   m\st m+1 \dt } |w|^{4-2m}$} }&{\scriptsize  }&{\scriptsize } \\

{\scriptsize}&{\scriptsize}&{\scriptsize}&{\scriptsize}&{\scriptsize}\\
\hline
{\scriptsize}&{\scriptsize}&{\scriptsize}&{\scriptsize}&{\scriptsize}\\

{\scriptsize \textbf{$\mathcal{O}\st \rep^{\beta} \dt$ for first two columns, $\mathcal{O}\st \rep^{\sigma} \dt$ for third and fourth columns}  }
&
{\scriptsize {   $H_{h_{j}^{(0)},k_{j}^{(0)}}^{o}\st w \dt$}} 
&
{\scriptsize ${  \varepsilon^{2}H_{\tilde{h}_{j}^{(0)},\tilde{k}_{j}^{(0)}}^{I}\st w \dt}$ }
&
{\scriptsize ${  H_{h_{j}^{(\dagger)},k_{j}^{(\dagger)}}^{o}\st w \dt}$ }&{\scriptsize ${  \varepsilon^{2}H_{\tilde{h}_{j}^{(\dagger)},\tilde{k}_{j}^{(\dagger)}}^{I}\st w \dt}$}\\

{\scriptsize}&{\scriptsize}&{\scriptsize}&{\scriptsize}&{\scriptsize}\\
\hline
{\scriptsize}&{\scriptsize}&{\scriptsize}&{\scriptsize}&{\scriptsize}\\

{\scriptsize \textbf{$o\st \rep^{\beta} \dt$ for first two columns, $o\st \rep^{\sigma} \dt$ for third and fourth columns}    }
&
{\scriptsize $B_{j}^{2}\varepsilon^{2}\psi_{\eta_{j}}^{(0)}\st \frac{\rep w}{B_{j}	\varepsilon} \dt$$-\sq \varepsilon^{2m}{\bf{G}}_{\bg,\cg}\st \rep w \dt\right.$ $- \cgaj  B_{j}^{2m}\varepsilon^{2m}\rep^{2-2m}|w|^{2-2m}$    $\left.- \cgaj  C_{j}\varepsilon^{2m}\rep^{4-2m}|w|^{4-2m} \dq^{(0)}$
$+\sq f_{\bg,\cg,\hg,\kg}^{o}- \hat{f}_{j,\bg,\cg,\hg,\kg}^{o}\cgaj B_{j}^{2m} \rep^{2-2m}|w|^{2-2m}\right.$ $\left.-\hat{f}_{j,\bg,\cg,\hg,\kg}^{o}C_{j}\rep^{4-2m}|w|^{4-2m}\dq^{(0)}$ }
&
{\scriptsize $\mathfrak{b}_{j}^{2}\varepsilon^{2}\psi_{\eta_{j}}^{(0)}\st \frac{\Rep w}{\mathfrak{b}_{j}} \dt$

$+\varepsilon^{2}V^{(0)}\st \frac{\Rep w}{\mathfrak{b}_{j}} \dt$
$+\varepsilon^{2}\sq \hkii\st \frac{\Rep w}{\mathfrak{b}_{j}} \dt-   H_{\tilde{h}_{j},\tilde{k}_{j}}^{I}\st w \dt\dq^{(0)}$
$+\sq  \varepsilon^{2}f_{j,\mathfrak{b}_{j},\tilde{h}_{j},\tilde{k}_{j}}^{I}\st \frac{\Rep w}{\mathfrak{b}_{j}} \dt \dq^{(0)}$}&{\scriptsize $B_{j}^{2}\varepsilon^{2}\psi_{\eta_{j}}^{(\dagger)}\st \frac{\rep w}{B_{j}	\varepsilon} \dt$
$-\sq \varepsilon^{2m}{\bf{G}}_{\bg,\cg}\st \rep w \dt  \dq^{(\dagger)}$
$+\sq f_{\bg,\cg,\hg,\kg}^{o}\dq^{(\dagger)}$}&{\scriptsize $\mathfrak{b}_{j}^{2}\varepsilon^{2}\psi_{\eta_{j}}^{(\dagger)}\st \frac{\Rep w}{\mathfrak{b}_{j}} \dt$
                                               
$+\varepsilon^{2}V^{(\dagger)}\st \frac{\Rep w}{\mathfrak{b}_{j}} \dt$
$+\varepsilon^{2}\sq \hkii\st \frac{\Rep w}{\mathfrak{b}_{j}} \dt-   H_{\tilde{h}_{j},\tilde{k}_{j}}^{I}\st w \dt\dq^{(\dagger)}$
$+\sq \varepsilon^{2}f_{j,\mathfrak{b}_{j},\tilde{h}_{j},\tilde{k}_{j}}^{I}\st \frac{\Rep w}{\mathfrak{b}_{j}} \dt \dq^{(\dagger)}$ }\\

{\scriptsize}&{\scriptsize}&{\scriptsize}&{\scriptsize}&{\scriptsize}\\

\hline 
\end{tabular}
\end{table}
\endlandscape

\restoregeometry				

The problem of gluing metrics $\omega_{\bg,\cg,\hg,\kg}$ and $\varepsilon^{2}\eta_{j,\mathfrak{b}_{j},\tilde{h}_{j},\tilde{k}_{j}}$ is equivalent to the problem of finding the correct parameters such that at $\Sp^{2m-1	}$ there is a $C^{3}$ matching of potentials $\mathcal{V}^{o}_{j,\bg,\cg,\hg,\kg}\st \rep w \dt$ and $\mathcal{V}^{I}_{j,\mathfrak{b}_{j},\tilde{h}_{j},\tilde{k}_{j}}\st \frac{\Rep w}{\mathfrak{b}_{j}} \dt$. As proved in \cite{ap1}   there is the $C^{3}$ matching at the boundaries  if and only if the following system is verified    
\begin{equation}
(\Sigma_{j}):\left\{
\begin{array}{rcl}
\mathcal{V}^{o}_{j,\bg,\cg,\hg,\kg}\st \rep w \dt&=&\mathcal{V}^{I}_{j,\mathfrak{b}_{j},\tilde{h}_{j},\tilde{k}_{j}}\st \frac{\Rep w}{\mathfrak{b}_{j}} \dt\\
\partial_{|w|}\mathcal{V}^{o}_{j,\bg,\cg,\hg,\kg}\st \rep w \dt&=&\partial_{|w|}\mathcal{V}^{I}_{j,\mathfrak{b}_{j},\tilde{h}_{j},\tilde{k}_{j}}\st \frac{\Rep w}{\mathfrak{b}_{j}} \dt\\
\Delta\mathcal{V}^{o}_{j,\bg,\cg,\hg,\kg}\st \rep w \dt&=&\Delta\mathcal{V}^{I}_{j,\mathfrak{b}_{j},\tilde{h}_{j},\tilde{k}_{j}}\st \frac{\Rep w}{\mathfrak{b}_{j}} \dt\\
\partial_{|w|}\Delta\mathcal{V}^{o}_{j,\bg,\cg,\hg,\kg}\st \rep w \dt&=&\partial_{|w|}\Delta\mathcal{V}^{I}_{j,\mathfrak{b}_{j},\tilde{h}_{j},\tilde{k}_{j}}\st \frac{\Rep w}{\mathfrak{b}_{j}} \dt
\end{array}\right.
\end{equation}

After choices \eqref{eq:sceltaB},\eqref{eq:sceltaC}, \eqref{eq:sceltagamma} and some algebraic manipulations, systems $(\Sigma_{j})$ become 


\begin{equation}
(\Sigma_{j}):\left\{
\begin{array}{rcl}
\varepsilon^{2}\tilde{h}_{j}&=&h_{j}- \xi_{j}\\
\varepsilon^{2}\tilde{k}_{j}&=&k_{j}- \Delta\xi_{j}\\
\partial_{|w|}\sq H_{h_{j},k_{j}}^{o}-H_{h_{j},k_	{j}}^{I}\dq&=& \partial_{|w|}\st \xi_{j}-H_{\xi_{j},\Delta\xi_{j}}^{I}\dt\\
\partial_{|w|}\Delta\sq H_{h_{j},k_{j}}^{o}-H_{h_{j},k_{j}}^{I}\dq&=& \partial_{|w|}\Delta\st\xi_{j}-H_{\xi_{j},\Delta\xi_{j}}^{I}\dt
\end{array}\right.
\end{equation}
with  $\xi_{j}$ a function depending linearly on elements of fourth row of table \ref{tab1}.  Using Theorem \ref{dirneu} we define the operators 

\begin{equation}
\mathcal{S}_{j}\st \varepsilon^{2}\tilde{h}_{j},\varepsilon^{2}\tilde{k}_{j} , h_{j},k_{j} \dt:=\st h_{j}- \xi_{j},  k_{j}- \Delta\xi_{j},\mathcal{Q}\st \partial_{|w|}\st \xi_{j}-H_{\xi_{j},\Delta\xi_{j}}^{I}\dt,\partial_{|w|}\Delta\st\xi_{j}-H_{\xi_{j},\Delta\xi_{j}}^{I}\dt   \dt \dt
\end{equation}

and then the operator $\mathcal{S}: \dombd^{2}\rightarrow \ambbda^{2}$

\begin{equation}
\mathcal{S}:=\st \mathcal{S}_{1},\ldots, \mathcal{S}_{N} \dt\,.
\end{equation}

Note also that biharmonic extensions, seen as operators 
\begin{equation}
H_{\cdot,\cdot}^{o},H_{\cdot,\cdot}^{I}:C^{4,\alpha}\st \Sp^{2m-1} \dt\times C^{2,\alpha}\st \Sp^{2m-1} \dt\rightarrow C^{4,\alpha}\st \Sp^{2m-1} \dt
\end{equation}
and the operator
\begin{equation}
\mathcal{Q}:C^{3,\alpha}\st \Sp^{2m-1} \dt\times C^{1,\alpha}\st \Sp^{2m-1} \dt\rightarrow C^{4,\alpha}\st \Sp^{2m-1} \dt\times C^{2,\alpha}\st \Sp^{2m-1} \dt
\end{equation}
defined in Proposition \ref{dirneu}, preserve eigenspaces of $\Delta_{\Sp^{2m-1}}$.

Thanks to the explicit knowledge of the various  terms we constructed ``by hand'' and that don't match perfectly (fourth row of table \ref{tab1}) we can find  $\kappa>0$ such that 
\begin{equation}
\mathcal{S}: \dombd^{2}\rightarrow \dombd^{2}\,.
\end{equation}
Now the conclusion follows immediately applying a Picard iteration scheme and standard regularity theory.
\end{proof}

\newpage

\section{Applications}

\noindent In this Section we list few examples where our results can be applied. We have confined ourselves to the case when $M$ is a {\em{toric}} K\"ahler-Einstein orbifold, but there is no doubt that this is far from a comprehensive list.

\begin{example}
Consider $\st\PP^{1}\times\PP^{1},\pi_{1}^{*}\omega_{FS}+\pi_{2}^{*}\omega_{FS}\dt$ and let $\ZZ_{2}$ act in the following way
\begin{equation}\st[x_{0}:x_{1}],[y_{0}:y_{1}]\dt\longrightarrow  \st[x_{0}:-x_{1}],[y_{0}:-y_{1}]\dt\end{equation} 

\noindent It's immediate to check that this action is in $SU(2)$ with four fixed points
\begin{gather}
p_{1}= \st[1:0],[1:0]\dt \\
p_{2}= \st[1:0],[0:1]\dt \\
p_{3}= \st[0:1],[1:0]\dt \\
p_{4}= \st[0:1],[0:1]\dt 
\end{gather}

\noindent The quotient space $X_{2}:=\PP^{1}\times\PP^{1}/\ZZ_{2}$ is a K\"ahler-Einstein, Fano orbifold. Since it is K\"ahler-Einstein, conditions for applying our construction become exactly the conditions of  \cite{ap2}, so we have to verify that the matrix
\begin{equation}
\Theta\st \textrm{\bf{1}} , s_{\omega} \textrm{\bf{1}}\dt  =\st \frac{ s_{\omega}}{2} \varphi_{j}\st p_{i} \dt \dt_{\substack{ 1\leq i \leq 2 \\1\leq j\leq 4}}\,.
\end{equation}
has full rank and there exist a positive element in $\ker \Theta\st \textrm{\bf{1}} , s_{\omega} \textrm{\bf{1}}\dt$. It is immediate to see that we have 
\begin{equation}H^{0}\st X_{2}, T^{(1,0)}X_{2} \dt=H^{0}\st \PP^{1}/\ZZ_{2}, T^{(1,0)}\st  \PP^{1}/\ZZ_{2} \dt \dt\oplus H^{0}\st \PP^{1}/\ZZ_{2}, T^{(1,0)}\st  \PP^{1}/\ZZ_{2} \dt \dt\,.\end{equation}
Moreover \begin{equation}H^{0}\st \PP^{1}/\ZZ_{2}, T^{(1,0)}\st  \PP^{1}/\ZZ_{2} \dt \dt\end{equation} is generated by holomorphic vector fields on $\PP^{1}$ that vanish on points 
$[0:1],[1:0]$ so 
\begin{equation}\dim_{\CC} H^{0}\st \PP^{1}/\ZZ_{2}, T^{(1,0)}\st  \PP^{1}/\ZZ_{2} \dt \dt=1\end{equation}
and an explicit generator is the vector field
\begin{equation}V=z^{1}\partial_{1}\,.\end{equation} 
We can compute explicitly its potential $\varphi_{V}$ with respect to $\omega_{FS}$ that is
\begin{equation}\varphi_{V}\st [z_{0}:z_{1}] \dt=- \frac{|z_{0}z_{1}|}{|z_{0}|^{2}+|z_{1}|^{2}}+\frac{1}{2}\end{equation}
and it is easy to see that it is a well defined function and 
\begin{equation}\int_{\PP^1}\varphi_{V}\omega_{FS}=0\,.\end{equation}
Summing up everything, we have that the matrix $\Theta\st \textrm{\bf{1}} , s_{\omega} \textrm{\bf{1}}\dt$ for $X_{2}$ is a $2\times 4$ matrix and can be written explicitly
\begin{equation}\Theta\st \textrm{\bf{1}} , s_{\omega} \textrm{\bf{1}}\dt=\frac{s_{\omega}}{2}\begin{pmatrix}
-1&-1 & 1&1\\
-1&1&-1&1
\end{pmatrix}\end{equation}
that has rank $2$ and every vector of type $\st a,b,b,a \dt$ for $a,b>0$ lies in  $\ker\Theta\st \textrm{\bf{1}} , s_{\omega} \textrm{\bf{1}}\dt$.

\end{example}

\begin{example}

Consider $\st\PP^{2},\omega_{FS}\dt$ and let $\ZZ_{3}$ act in the following way
\begin{equation}[z_{0}:z_{1}:z_{2}]\longrightarrow  [x_{0}:\zeta_{3}x_{1}:\zeta_{3}^{2}x_{2}]\qquad \zeta_{3}\neq 1, \zeta_{3}^{3}=1\end{equation} 

\noindent It's immediate to check that this action is in $SU(2)$ with three fixed points
\begin{gather}
p_{1}= [1:0:0]\\
p_{2}= [0:1:0] \\
p_{3}= [0:0:1] 
\end{gather}

\noindent The quotient space $X_{3}:=\PP^{2}/\ZZ_{3}$ is a K\"ahler-Einstein, Fano orbifold and it is isomorphic to the singular cubic surface in $\PP^{3}$
\begin{equation}
x_{1}x_{2}x_{3}-x_{0}^{3}=0\,.
\end{equation}
Again, conditions for applying our construction become exactly the conditions of Theorem \cite[Theorem  ]{ap2}, so we have to verify that the matrix
\begin{equation}
\Theta\st \textrm{\bf{1}} , s_{\omega} \textrm{\bf{1}}\dt  =\st \frac{ 2s_{\omega}}{3} \varphi_{j}\st p_{i} \dt \dt_{\substack{ 1\leq i \leq 2 \\1\leq j\leq 3}}\,.
\end{equation}
has full rank and there exist a positive element in $\ker \Theta\st \textrm{\bf{1}} , s_{\omega} \textrm{\bf{1}}\dt $. It is immediate to see that we have 
\begin{equation}
\dim_{\CC} H^{0}\st X_{3}, T^{(1,0)}X_{3} \dt=2
\end{equation}
because $H^{0}\st X_{3}, T^{(1,0)}X_{3} \dt$ it is generated by holomorphic vector fields on $\PP^{2}$ vanishing at points $p_{1},p_{2},p_{3}$.
Explicit generators are the vector fields
\begin{gather}
V_{1}=z^{1}\partial_{1}+z^{2}\partial_{2}\\
V_{2}=z^{0}\partial_{0}+z^{1}\partial_{1}
\end{gather} 
We can compute explicitly their potentials $\phi_{V_{1}},\phi_{V_{2}}$ with respect to $\omega_{FS}$ that are
\begin{equation}\phi_{V_{1}}\st [z_{0}:z_{1}:z_{2}] \dt=  -\frac{|z^{0}|^{2}}{|z^{0}|^{2}+|z^1|^{2}+|z^{2}|^{2}}+\frac{1}{3}\end{equation}
\begin{equation}\phi_{V_{2}}\st [z_{0}:z_{1}:z_{2}] \dt=-\frac{|z^{2}|^{2}}{|z^{0}|^{2}+|z^1|^{2}+|z^{2}|^{2}}+\frac{1}{3}\end{equation}
and it is easy to see that are well defined functions and 
\begin{equation}\int_{\PP^2}\phi_{V_{1}}\frac{\omega_{FS}^{2}}{2}=\int_{\PP^2}\phi_{V_{1}}\frac{\omega_{FS}^{2}}{2}=0\,\end{equation}
One can check that

\begin{align}
\varphi_{1}=&-3\st \phi_{1}+2\phi_{2}  \dt\\
\varphi_{2}=&-3\st 2\phi_{1}+\phi_{2}  \dt
\end{align}
is a basis of the space of potentials of holomorphic vector fields vanishing somewhere on $X_{3}$. Summing up everything, we have that the matrix $\Theta\st \textrm{\bf{1}} , s_{\omega} \textrm{\bf{1}}\dt$ for $X_{3}$ is a $2\times 3$ matrix and can be written explicitly
\begin{equation}\Theta\st \textrm{\bf{1}} , s_{\omega} \textrm{\bf{1}}\dt=\frac{2s_{\omega}}{3}\begin{pmatrix}
1&-1 & 0\\
0&-1&1
\end{pmatrix}\end{equation}
that has rank $2$ and every vector of type $\st a,a,a \dt$ for $a>0$ lies in  $\ker \Theta\st \textrm{\bf{1}} , s_{\omega} \textrm{\bf{1}}\dt$.

\end{example}

\subsection{Equivariant version and partial desingularizations}

\noindent If the manifold is acted on by a compact group it is immediate to observe that our proof goes through taking at every step of the proof equivariant spaces and averaging on the group with its Haar measure. We can then use the following

\begin{teo}\label{maintheoremequiv}
Let $\st M,\omega, g\dt$ be a compact Kcsc orbifold with isolated singularities and let $G$ be a compact subgroup of holomorphic isometries
such that $\omega $ is invariant under the action of $G$. Let $\p\:=\sg p_1,\ldots,p_{N}\dg\subseteq M$ the set of points  with neighborhoods biholomorphic to  a ball of  $\CC^m/\Gamma_{j}$ with $\Gamma_{j}$ nontrivial such that $\CC^{m}/\Gamma_{j}$ admits an ALE  Kahler Ricci-flat resolution $\st X_{\Gamma_{j}},h,\eta_{j} \dt$
and
\begin{align}
\ker\st   \Lg  \dt^{G} :=&\ker\st   \Lg  \dt\cap \left\{f\in C^{2}\st M \dt| \gamma^{*}\dd f=\dd f \quad \forall\gamma \in G \right\}\\
=&\left<1,\varphi_1,\ldots,\varphi_d \right>\,.
\end{align}

Suppose moreover that there exist $ \bg \in (\mathbb{R}^{+})^{N}$ and $\cg\in\RR^{N}$ such that
			
		\begin{displaymath}\label{eq:matricebalequiv}
		\left\{\begin{array}{lcl}
		\sum_{j=1}^{N}b_{j}\Delta_{\omega}\varphi_{i}\st p_{j} \dt+c_{j}\varphi_{i}\st p_{j} \dt=0 && i=1,\ldots, d\\
		&&\\
		\st b_{j}\Delta_{\omega}\varphi_{i}\st p_{j} \dt+c_{j}\varphi_{i}\st p_{j} \dt \dt_{\substack{1\leq i\leq d\\1\leq j\leq N}}&& \textrm{has full rank}
		\end{array}\right.
		\end{displaymath}
	
	If 
	\begin{equation}\label{eq:tuningequiv}
	c_{j}=s_{\omega}b_{j}\,,
	\end{equation}
then
\[
\tilde{M} : = M \sqcup _{{p_{1}, \varepsilon}} X_{\Gamma_1} \sqcup_{{p_{2},\varepsilon}} \dots
\sqcup _{{p_N, \varepsilon}} X_{\Gamma_N}
\]

\noindent has a Kcsc metric in the class 
\begin{equation}
\pi^{*}\sq\omega \dq+ \sum_{j=1}^{N}\varepsilon^{2m}\mathfrak{b}_{j}^{2m}\sq \tilde{\eta}_{j} \dq
\end{equation}
\noindent where
\begin{itemize}
\item 
$\mathfrak{i}_{j}^{*}\sq \tilde{\eta}_{j} \dq=[\eta_{j}]$
with  $\mathfrak{i}_{j}:X_{\Gamma_{j},\Rep}\hookrightarrow \tilde{M}$,
\item
$\left|\mathfrak{b}_{j}^{2m} - \frac{|\Gamma_{j}|b_{j}}{2\st m-1 \dt}\right| \leq C \varepsilon^{\gamma}$, for some $\gamma>0$.
\end{itemize}
\end{teo}

\noindent If the K\"ahler orbifold $(M,\omega )$ is  a toric variety,  $\omega $ is K\"ahler-Einstein and $G=\st S^{1} \dt^{m}$ then  $\omega $ is $G$-invariant (by Matsushima-Lichnerowicz) and 
\begin{equation}
\ker\st \Lg \dt^{G}=\left\{ 1,\varphi_{1},\ldots, \varphi_{m} \right\}\,.
\end{equation}

By definition, the functions  $\varphi_{j}$ are such that  
\begin{equation}
\partial^{\sharp}\varphi_{j}\st p \dt=\left.\frac{d}{dt}\sq \st e^{t\log\st \lambda_{j}^{1} \dt},\ldots, e^{t\log\st \lambda_{j}^{m} \dt} \dt \cdot p  \dq\right|_{t=0}\qquad (\lambda_{j}^{1},\ldots,\lambda_{j}^{m})\in \st\CC^{*}\dt^{m}
\end{equation}
\noindent and can be chosen in such a way that,  having set
\begin{equation}
\mu:M\rightarrow \RR^{m}\qquad \mu\st p\dt := \st \varphi_{1}\st p \dt,\ldots,\varphi_{d}\st p \dt \dt\,,
\end{equation}
the set $\mu\st M \dt$ is a convex polytope that coincides up to transformations in $SL(2,\ZZ)$ with the polytope associated to the pluri-anticanonical polarization of the toric variety $M$. Moreover 
\begin{equation}
\Lg=\Delta_{\omega}^{2}+\frac{s_{\omega}}{m}\Delta_{\omega}
\end{equation}
and 
\begin{equation}
\Delta\varphi_{j}=-\frac{s_{\omega}}{m}\varphi_{j}
\end{equation}
so 
\begin{equation}
\Theta\st \textrm{\bf{1}} , s_{\omega} \textrm{\bf{1}}\dt=\Theta\st \textrm{\bf{0}} , \frac{\st  m-1 \dt s_{\omega}}{m} \textrm{\bf{1}}\dt  =\st \frac{\st  m-1 \dt s_{\omega}}{m} \varphi_{j}\st p_{i} \dt \dt_{\substack{1\leq j\leq d\\1\leq i \leq N}}\,.
\end{equation}
Moreover the set $\mu\st \p \dt$ is a subset of the vertices of $\mu\st M \dt$,  indeed points of $\p$ are critical points for $\varphi_{j}$ since their gradients vanish at these points (indeed the holomorphic vector fields $\partial^{\sharp}\varphi_{j}$ must vanish at these points since they must preserve the isolated singularities). Assumptions of Theorem \ref{maintheoremequiv} are then satisfied if the barycenter of the set $\mu\st \p \dt$ is the origin of $\RR^{m}$.

\begin{example}
Let  $X^{(1)}$ be the toric K\"ahler-Einstein threefold whose  1-dimensional fan  $\Sigma^{(1)}_{1}$ is generated by points
\begin{equation}\Sigma^{(1)}_{1}=\left\{(1,3,-1), (-1,0,-1), (-1,-3,1), (-1,0,0), (1,0,0), (0,0,1), (0,0,-1), (1,0,1)\right\}
\end{equation}
and its $3$-dimensional fan $\Sigma^{(1)}_{3}$  is generated by $12$ cones 
\begin{align}
C_{1}:=& \left<        (-1,  0, -1),(-1, -3,  1),(-1,  0,  0)\right>\\
C_{2}:=& \left<        ( 1,  3, -1),(-1,  0, -1),(-1,  0,  0)\right>\\
C_{3}:=& \left<        (-1, -3,  1),(-1,  0,  0),( 0,  0,  1)\right>\\
C_{4}:=&\left<        ( 1,  3, -1),(-1,  0,  0),( 0,  0,  1)\right>\\
C_{5}:= &\left<        ( 1,  3, -1),(-1,  0, -1),( 0,  0, -1)\right>\\
C_{6}:=&\left<        (-1,  0, -1),(-1, -3,  1),( 0,  0, -1)\right>\\
C_{7}:= & \left<        (-1, -3,  1),( 1,  0,  0),( 0,  0, -1)\right>\\
C_{8}:=& \left<        (1,  3, -1),(1,  0,  0),(0,  0, -1)\right>\\
C_{9}:= &\left<        (1,  3, -1),(0,  0,  1),(1,  0,  1)\right>\\
C_{10}:=& \left<        (-1, -3,  1),( 1,  0,  0),( 1,  0,  1)\right>\\
C_{11}:=& \left<        (1,  3, -1),(1,  0,  0),(1,  0,  1)\right>\\
C_{12}:= & \left<        (-1, -3,  1),( 0,  0,  1),( 1,  0,  1)\right>
\end{align}

\noindent All these cones are singular and $C_{1},C_{4},C_{5},C_{7},C_{11},C_{12}$ are cones relative to affine open subsets of $X^{(1)}$ containing a $SU(3)$ singularity, while the others  are cones relative to affine open subsets of $X^{(1)}$ containing a $U(3)$ singularity. 


\noindent The  3-anticanonical polytope  $P_{-3K_{X^{(1)}}}$ is the convex hull of vertices
\begin{align}
P_{-3K_{X^{(1)}}}:=&\left<(0,-2,-3), (-3,0,0), (-3,1,3), (0,0,3), (3,-2,0),\right.\\
&\left. (0,2,3), (0,0,-3), (-3,2,0), (-3,3,3), (3,0,0), (3,-1,-3), (3,-3,-3)\right>
\end{align}


\noindent With 2-faces 
\begin{align}
F_{1}:=&\left<        ( 0, -2, -3),( 3, -3, -3),(-3,  0,  0),(-3,  1,  3),( 0,  0,  3),( 3, -2,  0)\right>\\
F_{2}:=&\left<        (-3,  1,  3),( 0,  0,  3),( 0,  2,  3),(-3,  3,  3)\right>\\
F_{3}:=&\left<        (0,  0,  3),(3, -2,  0),(0,  2,  3),(3,  0,  0)\right>\\
F_{4}:=&\left<        ( 0, -2, -3),(-3,  0,  0),( 0,  0, -3),(-3,  2,  0)\right>\\
F_{5}:=&\left<        ( 3, -1, -3),( 0,  2,  3),( 0,  0, -3),(-3,  2,  0),(-3,  3,  3),( 3,  0,  0)\right>\\
F_{6}:=&\left<        (-3,  0,  0),(-3,  1,  3),(-3,  2,  0),(-3,  3,  3)\right>\\
F_{7}:=&\left<        (3, -1, -3),(0, -2, -3),(3, -3, -3),(0,  0, -3)\right>\\
F_{8}:=&\left<        (3, -1, -3),(3, -3, -3),(3, -2,  0),(3,  0,  0)\right>
\end{align}


\noindent We have the following correspondences between cones containing a $SU(3)$-singularity and vertices of $P_{-3K_{X^{(1)}}}$
\begin{align}
C_{1}& \longleftrightarrow F_{3}\cap F_{5}\cap F_{8}=\sg (3,0,0) \dg\\
C_{4}& \longleftrightarrow F_{1}\cap F_{7}\cap F_{8}=\sg (3,-3,-3) \dg\\
C_{5}& \longleftrightarrow F_{1}\cap F_{2}\cap F_{3}=\sg (0,0,3) \dg\\
C_{7}& \longleftrightarrow F_{2}\cap F_{5}\cap F_{7}=\sg (-3,3,3) \dg\\
C_{11}& \longleftrightarrow F_{1}\cap F_{4}\cap F_{6}=\sg (-3,0,0) \dg\\
C_{12}& \longleftrightarrow F_{4}\cap F_{5}\cap F_{7}=\sg (0,0,-3) \dg
\end{align}

\noindent Since in complex dimension $3$ every $SU(3)$-singularity admits a K\"ahler crepant resolution it is then immediate to see that all assumptions of Theorem \ref{maintheoremequiv} are satisfied. 

\end{example}

\begin{example}
Let  $X^{(4)}$ be the toric K\"ahler-Einstein threefold whose  1-dimensional fan  $\Sigma^{(3)}_{1}$ is generated by points
\begin{equation}\Sigma^{(4)}_{1}=\left\{(0,3,1), (1,1,2), (1,0,0), (-1,0,0), (-2,-1,-2), (1,-3,-1)\right\}\end{equation}
and its $3$-dimensional fan $\Sigma^{(4)}_{3}$  is generated by $8$ cones

\begin{align}
C_{1}:=&\left<( 0,  3,  1),( 1,  1,  2),(-1,  0,  0)\right>\\
C_{2}:=&\left<(0, 3, 1),(1, 1, 2),(1, 0, 0)\right>\\
C_{3}:=&\left<( 0,  3,  1),(-1,  0,  0),(-2, -1, -2)\right>\\
C_{4}:=&\left<( 0,  3,  1),( 1,  0,  0),(-2, -1, -2)\right>\\
C_{5}:=&\left<( 1,  0,  0),(-2, -1, -2),( 1, -3, -1)\right>\\
C_{6}:=&\left<( 1,  1,  2),(-1,  0,  0),( 1, -3, -1)\right>\\
C_{7}:=&\left<(-1,  0,  0),(-2, -1, -2),( 1, -3, -1)\right>\\
C_{8}:=&\left<(1,  1,  2),(1,  0,  0),(1, -3, -1)\right>
\end{align}

\noindent The cones $C_{1},C_{4},C_{7},C_{8}$ are relative to affine open subsets of $X^{(4)}$ containing a $SU(3)$ singularity and the  other cones are  relative to affine open subsets of $X^{(4)}$ containing a $U(3)$ singularity.

\noindent The  5-anticanonical polytope  $P_{-5K_{X^{(4)}}}$ is the convex hull of vertices
\begin{align}
P_{-5K_{X^{(4)}}}:=&\left< (5,-1,-2), (5,0,-5), (-5,-2,1), (-5,0,0),\right.\\
&\left. (5,5,-5), (-5,-5,10), (-5,-3,9), (5,6,-8) \right>
\end{align}
With 2-faces 
\begin{align}
F_{1}:=&\left<( 5,  0, -5),(-5, -2,  1),(-5,  0,  0),( 5,  6, -8)\right>\\
F_{2}:=&\left<( 5, -1, -2),( 5,  0, -5),(-5, -2,  1),(-5, -5, 10)\right>\\
F_{3}:=&\left<(5, -1, -2),(5,  0, -5),(5,  5, -5),(5,  6, -8)\right>\\
F_{4}:=&\left<( 5, -1, -2),( 5,  5, -5),(-5, -5, 10),(-5, -3,  9)\right>\\
F_{5}:=&\left<(-5, -2,  1),(-5,  0,  0),(-5, -5, 10),(-5, -3,  9)\right>\\
F_{6}:=&\left<(-5,  0,  0),( 5,  5, -5),(-5, -3,  9),( 5,  6, -8)\right>
\end{align}


\noindent We have the following correspondences between cones containing a $SU(3)$-singularity and vertices of $P_{-5K_{X^{(4)}}}$
\begin{align}
C_{1}& \longleftrightarrow F_{1}\cap F_{2}\cap F_{5}=\sg (-5,-2,1) \dg\\
C_{4}& \longleftrightarrow F_{2}\cap F_{3}\cap F_{4}=\sg (5,-1,-2) \dg\\
C_{7}& \longleftrightarrow F_{4}\cap F_{5}\cap F_{6}=\sg (-5,-3,9) \dg\\
C_{8}& \longleftrightarrow F_{1}\cap F_{3}\cap F_{6}=\sg (5,6,-8) \dg
\end{align}

\noindent Since in complex dimension $3$ every $SU(3)$-singularity admits a K\"ahler crepant resolution it is then immediate to see that all assumptions of Theorem \ref{maintheoremequiv} are satisfied.

\end{example}


\providecommand{\bysame}{\leavevmode\hbox to3em{\hrulefill}\thinspace}
\providecommand{\MR}{\relax\ifhmode\unskip\space\fi MR }
\providecommand{\MRhref}[2]{%
  \href{http://www.ams.org/mathscinet-getitem?mr=#1}{#2}
}
\providecommand{\href}[2]{#2}


\end{document}